%% file: ms.tex
\documentclass[a4paper, reqno]{amsart}

\usepackage[UKenglish]{babel}
\usepackage{amsmath, amssymb, amsthm}
\usepackage{scrextend}
\usepackage[hidelinks]{hyperref}
\usepackage{xpatch}
\usepackage{color}
\usepackage{tikz-cd}
\usepackage{enumitem}
\usepackage{theoremref}
\usepackage{subcaption}

\usepackage[inline,final]{showlabels}
\showlabels{thlabel}

\setlist[description]{font=\normalfont\scshape}

\usepackage{imakeidx}
\makeindex[name=notes,title=Notes throughout the document (red text),columns=1]

\xpatchcmd{\proof}{\itshape}{\normalfont\bfseries}{}{}
\newtheoremstyle{repeat}{}{}{\itshape}{}{\bfseries}{.}{.5em}{#3, repeated}

\newtheorem{theorem}{Theorem}[section]
\newtheorem{proposition}[theorem]{Proposition}
\newtheorem{lemma}[theorem]{Lemma}
\newtheorem{corollary}[theorem]{Corollary}
\newtheorem{fact}[theorem]{Fact}

\theoremstyle{definition}
\newtheorem{definition}[theorem]{Definition}
\newtheorem{remark}[theorem]{Remark}
\newtheorem{convention}[theorem]{Convention}
\newtheorem{example}[theorem]{Example}
\newtheorem{question}[theorem]{Question}
\newtheorem{claim}{Claim}

\theoremstyle{repeat}
\newtheorem*{repeated-theorem}{Repeat}

\renewcommand{\L}{\mathcal{L}}
\newcommand{\MM}{\mathfrak{M}}
\newcommand{\Z}{\mathbb{Z}}
\newcommand{\E}{\mathcal{E}}

\DeclareMathOperator{\tp}{tp}
\DeclareMathOperator{\Lstp}{Lstp}
\DeclareMathOperator{\qftp}{qftp}
\DeclareMathOperator{\Aut}{Aut}

\renewcommand{\d}{\operatorname{d}}

\renewcommand{\phi}{\varphi}
\newcommand{\equivls}{\equiv^\textup{Ls}}
\newcommand{\op}{{\textup{op}}}
\newcommand{\lex}{{\textup{lex}}}
\newcommand{\EAfield}{{\textup{EA-field}}}
\newcommand{\EA}{{\textup{EA}}}
\newcommand{\ACF}{{\textup{ACF}}}
\newcommand{\real}{{\textup{real}}}

\def\Ind#1#2{#1\setbox0=\hbox{$#1x$}\kern\wd0\hbox to 0pt{\hss$#1\mid$\hss}
\lower.9\ht0\hbox to 0pt{\hss$#1\smile$\hss}\kern\wd0}
\def\ind{\mathop{\mathpalette\Ind{}}}
\def\Notind#1#2{#1\setbox0=\hbox{$#1x$}\kern\wd0\hbox to 0pt{\mathchardef
\nn="3236\hss$#1\nn$\kern1.4\wd0\hss}\hbox to 0pt{\hss$#1\mid$\hss}\lower.9\ht0
\hbox to 0pt{\hss$#1\smile$\hss}\kern\wd0}
\def\nind{\mathop{\mathpalette\Notind{}}}

\tikzset{every picture/.style={line width=0.75pt}} 

\title{Kim-independence in positive logic}
\author{Jan Dobrowolski and Mark Kamsma}
\thanks{The first author was supported by DFG project BA
6785/2-1. He also acknowledges the financial support of his visit to UEA by EPSRC grant EP/S017313/1. The second author was supported by a studentship from the UEA}
\email[Jan Dobrowolski]{dobrowol@math.uni.wroc.pl}
\address[Jan Dobrowolski]{Institute for Mathematical Logic and Basic Research Department of Mathematics and Computer Science at the University of Münster, Orléans-Ring 10, 48149 Münster and \newline
Instytut Matematyczny Uniwersytetu Wroc\l{}awskiego, pl. Grunwaldzki 2/4, 50-383 Wroc\l{}aw}
\email[Mark Kamsma]{m.kamsma@uea.ac.uk}
\urladdr[Mark Kamsma]{https://markkamsma.nl}
\address[Mark Kamsma]{School of Mathematics, University of East Anglia, Norwich, Norfolk, NR4 7TJ, UK}
\date{\today}
 
\begin{document}

\maketitle

\input{tex/abstract.tex}

\tableofcontents

\input{tex/introduction.tex}
\input{tex/preliminaries.tex}
\input{tex/global-ls-invariant-types.tex}
\input{tex/kim-dividing.tex}
\input{tex/CR.tex}
\input{tex/symmetry.tex}
\input{tex/independence-theorem.tex}
\input{tex/transitivity}
\input{tex/kim-pillay.tex}
\input{tex/examples.tex}

\bibliographystyle{alpha}
\bibliography{bibfile}


\end{document}

%% file: tex/abstract.tex
\begin{abstract}
An important dividing line in the class of unstable theories is being NSOP$_1$, which is more general than being simple. In NSOP$_1$ theories forking independence may not be as well-behaved as in stable or simple theories, so it is replaced by another independence notion, called Kim-independence. We generalise Kim-independence over models in NSOP$_1$ theories to positive logic---a proper generalisation of full first-order logic where negation is not built in, but can be added as desired. For example, an important application is that we can add hyperimaginary sorts to a positive theory to get another positive theory, preserving NSOP$_1$ and various other properties. We prove that, in a thick positive NSOP$_1$ theory, Kim-independence over existentially closed models has all the nice properties that it is known to have in an NSOP$_1$ theory in full first-order logic. We also provide a Kim-Pillay style theorem, characterising which thick positive theories are NSOP$_1$ by the existence of a certain independence relation. Furthermore, this independence relation must then be the same as Kim-independence. Thickness is the mild assumption that being an indiscernible sequence is type-definable.

In full first-order logic Kim-independence is defined in terms of Morley sequences in global invariant types. These may not exist in thick positive theories. We solve this by working with Morley sequences in global Lascar-invariant types, which do exist in thick positive theories. We also simplify certain tree constructions that were used in the study of Kim-independence in full first-order logic. In particular, we only work with trees of finite height.
\end{abstract}

%% file: tex/introduction.tex
\section{Introduction}
The study of (ternary) independence relations in model theory goes back to Shelah's notion of forking independence, which is an abstract generalisation of classical independence notions such as linear independence in vector spaces and algebraic independence in algebraically closed fields. Forking independence was initially used to study stable theories, in which it enjoys particularly nice properties. It was later discovered that forking independence can be useful in studying the broader class of simple theories, as it retains most of its features in that class (\cite{Kim98,KiPi}). Moreover,  the fundamental properties of forking independence in simple theories, such as transitivity, symmetry, and local character, fail in \emph{all} non-simple theories, which suggested that forking independence might not be so useful in studying any broader class of theories. On the other hand, some natural examples of non-simple theories admitting useful notions of independence have been known, including the theories of infinite-dimensional vector spaces with a generic bilinear form (\cite{Gra}), $\omega$-free PAC fields (\cite{Cha,Cha1}), and random parametrised equivalence relations. Inspired by some ideas of Kim (cf.\ \cite{Kim09}), and building on \cite{chernikov_model-theoretic_2015}, Kaplan and Ramsey have defined in \cite{kaplan_kim-independence_2020} the notion of Kim-independence (denoted by $\ind^K$), and they have proved that in NSOP$_1$ theories---a class containing all simple theories and, among many others, the three non-simple theories mentioned above  (\cite[Section 6]{chernikov_model-theoretic_2015})---it satisfies over models all the main properties of forking independence in simple theories except base-monotonicity. 

The goal of this paper is to generalise the theory of Kim-independence in NSOP$_1$ theories to the class of thick positive theories. Positive model theory, introduced in \cite{ben-yaacov_positive_2003,BYP} (with some ideas in a similar direction present also in \cite{Hr97} and \cite{Pi00}), provides a framework generalising that of full first-order logic and allows to study a wider range of objects using model-theoretic techniques. An important class of such objects, which motivated the study undertaken in \cite{ben-yaacov_positive_2003}, is that of the hyperimaginary extensions $T^\text{heq}$ of theories $T$ in full first-order logic. In the context of NSOP$_1$ theories, elimination of hyperimaginaries has been assumed in \cite{Kim20} in order to carry out a construction of weak canonical bases. It was asked there (in the discussion following Definition 4.1) whether $T^\text{heq}$ satisfies the existence axiom for forking independence provided that $T$ does. We observe that this is indeed true (\thref{thm:hyperimaginaries-existence-axiom-for-forking}), which might be helpful in eliminating the assumption of elimination of hyperimaginaries in \cite{Kim20} by working with Kim-independence in $T^\text{heq}$.

In \cite{haykazyan_existentially_2021}  Haykazyan and Kirby studied the theory ECEF of existentially closed exponential fields, and, working with an arbitrary JEP-refinement (which, intuitively, corresponds to a completion of an incomplete theory in full first-order logic), they have found an invariant ternary relation satisfying over models the following properties: strong finite character, existence, monotonicity, symmetry, and independence theorem. They have also proved that, for any positive theory, the existence of such a relation implies NSOP$_1$, so in particular the JEP-refinements of ECEF are NSOP$_1$. As in the full first-order setting, a natural question whether every positive NSOP$_1$ theory admits a ternary relation satisfying these properties arises.

Another class of examples of non-simple NSOP$_1$ theories in positive logic comes from a recent work \cite{acf0g}, where d'Elb\'ee, Kaplan and Neuhauser show that for any intergral domain $R$ all JEP-refinements of the theory $F_{R\text{-module}}$ of fields with a generic $R$-submodule are NSOP$_1$ but not simple.  In particular, this applies to the theory of algebraically closed fields of characteristic zero with a generic additive subgroup.

We work under the mild assumption that the theory is thick. This means that being an indiscernible sequence is type-definable. Theories in full first-order logic, and their hyperimaginary extensions, are always thick. The theories ECEF and $F_{R\text{-module}}$ mentioned above are also thick.

\textbf{Main results.} The main results of our paper state that in every thick NSOP$_1$ theory, Kim-independence satisfies: symmetry (\thref{thm:symmetry}), the (strong) independence theorem (\thref{thm:independence-theorem}, \thref{thm:strong-independence-theorem}), transitivity (\thref{thm:transitivity}) and local character (\thref{cor:kim-independence-local-character}), as well as invariance under automorphisms, existence, extension, monotonicity and (strong) finite character. Moreover, we prove a Kim-Pillay style theorem: in any thick positive theory $T$, if there exists a ternary relation $\ind$ satisfying all the above properties, then  $T$ is NSOP$_1$ and $\ind = \ind^K$ (\thref{thm:kim-pillay-style}).

\textbf{Challenges.} In contrast to the full first-order setting, in a positive theory, a type over an existentially closed model may fail to have an invariant global extension. 
This is a fundamental obstacle to generalising Kim-independence to the positive setting, as the original definition of it relies on existence of invariant extensions in the full first-order setting. We show, however, that in a thick theory any type over an existentially closed model $M$ extends to a global $M$-Lascar-invariant type. We define Kim-independence in an arbitrary thick positive theory replacing the use of invariant types by Lascar-invariant types.

One of the difficulties in adapting the results of \cite{kaplan_kim-independence_2020, kaplan_transitivity_2019} to the positive setting is that the tree modelling property \cite[Theorem 4.3]{KKS}, on which most of the constructions there rely, is not available in the positive setting. This forced us in particular to work only with trees of finite height, which turns out to be enough due to compactness and a careful choice of the global types with which we work.  
Consequently, we substitute the notion of a tree Morley sequence used in \cite{kaplan_kim-independence_2020} with a weaker notion of a parallel-Morley sequence. In particular, we do not have a counterpart of the chain condition (\cite[Corollary 5.15]{kaplan_kim-independence_2020}) for parallel-Morley sequences,  which causes some additional technical difficulties in our proof of the strong independence theorem.

Our proofs yield in particular alternative proofs of the results in full first-order logic on Kim-independence not using any combinatorial tools other than Ramsey theorem: while we do use Erd\H{o}s-Rado theorem to extract indiscernible sequences, in the full first-order setting this can be always replaced by the standard use of Ramsey theorem; the technique of extracting strongly indiscernible trees from s-indiscernible trees from \cite{kaplan_kim-independence_2020} relying on Erd\H{o}s-Rado theorem is not used by us.


\textbf{Overview.} The paper is organised as follows. In Section \ref{sec:preliminaries} we review some basic terminology and facts about positive logic and NSOP$_1$ theories, and we make some observations which are used throughout the paper. In Section \ref{sec:global-ls-invariant-types} we define a notion of a Morley sequence in a global Lascar-invariant type, and we prove some basic properties of these.  In Section \ref{sec:kim-dividing} we define Kim-dividing in an arbitrary thick NSOP$_1$ theory, we give several characterisations of Kim-dividing and we establish some basic properties of Kim-independence.
In Section \ref{sec:cr-morley} we develop some tools which we later use in certain tree constructions: the EM-modelling property, which is a weak version of the modelling property used in \cite{kaplan_kim-independence_2020}, parallel-Morley sequences, which serve as our substitute for the notion of a tree Morley sequence from \cite{kaplan_kim-independence_2020}, and $q$-spread-outness, which is a variant of spread-outness used in \cite{kaplan_kim-independence_2020}.
Sections \ref{sec:symmetry}, \ref{sec:independence-theorem} and \ref{sec:transitivity} contain the proofs of the main properties of Kim-independence in thick positive theories: symmetry, independence theorem and transitivity, and Section \ref{sec:kim-pillay-style-theorem} is dedicated to proving a Kim-Pillay-style characterisation of the NSOP$_1$ property among thick positive theories by existence of an abstract independence relation satisfying certain properties, and the characterisation of Kim-independence in NSOP$_1$ theories as the only relation satisfying them. 
In Section \ref{sec:examples} we describe in details some examples of thick NSOP$_1$ theories: Poizat's example of a thick non-semi-Hausdorff theory, (JEP refinements of) the positive theory of existentially closed exponential fields studied in \cite{haykazyan_existentially_2021}, and the
hyperimaginary extensions of  NSOP$_1$ theories.

\textbf{Acknowledgements.} We are grateful to Jonathan Kirby for suggesting this project to us and for many helpful conversations and comments. We also thank Rosario Mennuni for helpful discussions, in particular for pointing out the example described in Subsection  \ref{subsec:thick-non-semi-hausdorff} to us. Finally, we thank the anonymous referee for their remarks which helped improve the presentation of this paper.

%% file: tex/preliminaries.tex
\section{Preliminaries}
\label{sec:preliminaries}
In this section we recall the basics of positive logic that we need in this paper. For a more extensive treatment we refer to \cite{ben-yaacov_positive_2003, poizat_positive_2018}.

Throughout the paper variables will be of arbitrary (possibly infinite) length, unless stated otherwise.
\begin{definition}
\thlabel{def:positve-syntax}
Fix a signature $\L$. A \emph{positive existential formula} in $\L$ is one that is obtained from combining atomic formulas using $\wedge$, $\vee$, $\top$, $\bot$ and $\exists$. An \emph{h-inductive sentence} is a sentence of the form $\forall x(\phi(x) \to \psi(x))$, where $\phi(x)$ and $\psi(x)$ are positive existential formulas. A \emph{positive theory} is a set of h-inductive sentences.
\end{definition}
Note that every positive existential formula $\phi(x)$ is equivalent to one of the form $\exists y \psi(x, y)$, where $\psi(x, y)$ is positive quantifier-free. Positive existential sentences and their negations can be used as axioms in a positive theory, since $\forall x \phi(x)$ and $\forall x \neg \phi(x)$ are equivalent to $\forall x(\top \to \phi(x))$ and $\forall x(\phi(x) \to \bot)$ respectively.

As in full first-order logic, we will assume that $\L$ contains a symbol $=$ interpreted in every $\L$-structure as equality.
\begin{remark}
\thlabel{rem:morleyisation}
We can study full first-order logic as a special case of positive logic. This is done through a process called \emph{Morleyisation}. For this we add a relation symbol $R_\phi(x)$ to our language for every formula $\phi(x)$ in full first-order logic. Then we have our theory (inductively) express that $R_\phi(x)$ and $\phi(x)$ are equivalent. This way every formula in full first-order logic is (equivalent to) a relation symbol, and thus in particular to a positive existential formula.

Many definitions later in this section simplify in this case to familiar concepts. Every homomorphism will be an elementary embedding, and thus in particular an immersion. So every model will be an e.c.\ model. A theory has JEP if and only if it is complete, and the JEP-refinements correspond to completions.
\end{remark}
Since we will only be considering full first-order logic as a special case of positive logic, we will make the following convention.
\begin{convention}
\thlabel{conv:positive-is-standard}
Whenever we say ``formula'' or ``theory'' we will mean ``positive existential formula'' and ``positive theory'' respectively, unless explicitly stated otherwise. This also means that every formula and theory we consider will be implicitly assumed to be positive (existential).
\end{convention}
In full first-order logic we consider elementary embeddings because they preserve and reflect truth of full first-order formulas. Since we do not have negation in positive logic, there is a difference between preserving and reflecting truth of positive existential formulas.
\begin{definition}
\thlabel{def:homomorphism}
A function $f: M \to N$ between $\L$-structures is called a \emph{homomorphism} if for every $\phi(x)$ and every $a \in M$ we have
\[
M \models \phi(a) \implies N \models \phi(f(a)).
\]
We call $f$ an \emph{immersion} if additionally the converse implication holds for all $\phi(x)$ and all $a \in M$.
\end{definition}
In positive model theory we study the existentially closed models.
\begin{definition}
\thlabel{def:ec-model}
We call a model $M$ of $T$ an \emph{existentially closed model} or an \emph{e.c.\ model} if the following equivalent conditions hold:
\begin{enumerate}[label=(\roman*)]
\item every homomorphism $f: M \to N$ with $N \models T$ is an immersion;
\item for every $a \in M$ and $\phi(x)$ such that there is a homomorphism $f: M \to N$ with $N \models T$ and $N \models \phi(f(a))$, we have that $M \models \phi(a)$;
\item for every $a \in M$ and $\phi(x)$ such that $M \not \models \phi(a)$ there is $\psi(x)$ with $T \models \neg \exists x (\phi(x) \wedge \psi(x))$ and $M \models \psi(a)$.
\end{enumerate}
\end{definition}
\begin{fact}
\thlabel{fact:positive-logic-facts}
Let $T$ be some theory.
\begin{enumerate}[label=(\roman*)]
\item \emph{(Unions)} The union of a chain of (e.c.) models is an (e.c.) model.
\item \emph{(Amalgamation)} If one of $M_1 \leftarrow M \to M_2$ is an immersion then there are $M_1 \to N \leftarrow M_2$ making the relevant square commute. In particular, every e.c.\ model is an amalgamation base.
\item \emph{(Existential completion)} For every $M \models T$ there is a homomorphism $f: M \to N$, where $N$ is an e.c.\ model of $T$.
\item \emph{(Compactness)} Let $\Sigma(x)$ be a set of positive existential formulas and suppose that for every finite $\Sigma_0(x) \subseteq \Sigma(x)$ there is $M \models T$ with $a \in M$ such that $M \models \Sigma_0(a)$. Then there is an e.c.\ model $N$ of $T$ with $a \in N$ such that $N \models \Sigma(a)$.
\end{enumerate}
\end{fact}
In the statement of compactness, \thref{fact:positive-logic-facts}(iv), we have explicitly mentioned positive existential formulas because it is crucial that we cannot use all formulas from full first-order logic in $\Sigma(x)$. This is actually one of the big obstacles in this paper. We provide two examples to indicate how full compactness can fail.
\begin{example}
\thlabel{ex:compactness-failure}
Consider the theory $T$ with a symbol for inequality and $\omega$ many disjoint unary predicates $P_n(x)$. Then e.c.\ models of $T$ are precisely those which consist of $\omega$-many disjoint infinite sets, one for each predicate. If we had full compactness then the set
\[
\Sigma(x) = \{ \neg P_n(x) : n < \omega \}
\]
would have a realisation in some e.c.\ model, which is impossible.
\end{example}
\begin{example}
\thlabel{ex:bounded-set-compactness}
It could happen that there is a definable set that is infinite and bounded. This does not contradict compactness: it just means that inequality is not positively definable on that set. Such situations might arise when adding hyperimaginaries as real elements, which can be done in positive logic (see section \ref{subsec:hyperimaginaries}).
\end{example}
\begin{definition}
\thlabel{def:jep}
We say that a theory $T$ has the \emph{joint embedding property} or \emph{JEP} if the following equivalent conditions hold:
\begin{enumerate}[label=(\roman*)]
\item for any two models $M_1$ and $M_2$ there are homomorphisms $M_1 \to N \leftarrow M_2$;
\item if $T \models \neg \phi \vee \neg \psi$ then $T \models \neg \phi$ or $T \models \neg \psi$;
\item for some e.c.\ model $M$ of $T$, all h-inductive sentences that are true in $M$ are implied by $T$.
\end{enumerate}
For a theory $T$ we call an extension $T'$ of $T$ a \emph{JEP-refinement} of $T$ if it has JEP and every e.c.\ model of $T'$ is also an e.c.\ model of $T$.
\end{definition}
As hinted in \thref{rem:morleyisation}, having JEP is like requiring the theory to be complete. We can always find a JEP-refinement (a `completion') by taking the set of h-inductive sentences that are true in some e.c.\ model.

Fix a sufficiently large cardinal $\bar{\kappa}$. We will say a set is $\emph{small}$ if it is of cardinality smaller than $\bar{\kappa}$.
\begin{convention}
\thlabel{conv:monster-model}
We will assume our theory $T$ has JEP so we can work in a \emph{monster model} $\MM$ (sometimes also called a universal domain), that is:
\begin{itemize}
\item \emph{existentially closed}: $\MM$ is an e.c.\ model;
\item \emph{very homogeneous}: any partial immersion $f: \MM \to \MM$ with small domain and codomain extends to an automorphism on all of $\MM$;
\item \emph{very saturated}: any finitely satisfiable small set of formulas over $\MM$ is satisfiable in $\MM$.
\end{itemize}
We will assume all parameter sets considered to be small, except when we consider the monster model as a parameter set. We will use lowercase Latin letters $a, b, \ldots$ for (possibly small infinite) tuples inside the monster model and uppercase Latin letters $A, B, \ldots$ for (small) parameter sets inside the monster model. We will use letters $M$ and $N$ when these sets are e.c.\ models.

As is common, we use the notation $\models \phi(a)$ to abbreviate $\MM \models \phi(a)$.
\end{convention}
The above also means that the right notion of a type in positive model theory is that of a positive existential type. That is, we write $\tp(a/B)$ for the set of all positive existential formulas over $B$ satisfied by $a$. So we have $\tp(a/B) = \tp(a'/B)$ if and only if there is an automorphism $f: \MM \to \MM$ fixing $B$ such that $f(a) = a'$. We also write $a \equiv_B a'$ in this case.
By a type (over $A$) in $T$ we will always mean a maximal consistent with $T$ set of positive existential formulas (over $A$). By a partial type (over $A$) in $T$ we will mean any consistent set of positive existential formulas (over $A$).

There are some subtle differences in possible definitions of saturatedness, see for example \cite[Section 2.4]{poizat_positive_2018}. We are only interested in e.c.\ models, so for us it will mean the following. Constructing models of a certain level of saturation is then standard.
\begin{definition}
\thlabel{def:saturated}
Let $M$ be an e.c.\ model of some theory $T$. We say that $M$ is \emph{$\kappa$-saturated} if for every $A \subseteq M$ with $|A| < \kappa$ we have that a set $\Sigma(x)$ of formulas over $A$ is satisfiable in $M$ if and only if it is finitely satisfiable in $M$.
\end{definition}
\begin{fact}
\thlabel{fact:saturated-models}
For any $\kappa \geq |A| + |T|$ there is a $\kappa^+$-saturated $N \supseteq A$ with $|N| \leq 2^\kappa$.
\end{fact}
The following definitions are taken from \cite{ben-yaacov_positive_2003, ben-yaacov_thickness_2003}. We added the notion of being Boolean.
\begin{definition}
\thlabel{def:hausdorff-thick}
Let $T$ be a theory and work in a monster model. We call a theory $T$:
\begin{itemize}
\item \emph{Boolean} if every formula in full first-order logic is equivalent to a positive existential formula, modulo $T$;
\item \emph{Hausdorff} if for any two distinct types $p(x)$ and $q(x)$ there are $\phi(x) \not \in p(x)$ and $\psi(x) \not \in q(x)$ such that $\models \forall x(\phi(x) \vee \psi(x))$;
\item \emph{semi-Hausdorff} if equality of types is type-definable, so there is a partial type $\Omega(x, y)$ such that $\tp(a) = \tp(b)$ if and only if $\models \Omega(a, b)$;
\item \emph{thick} if being an indiscernible sequence is type-definable, so there is a partial type $\Theta((x_i)_{i < \omega})$ such that $(a_i)_{i < \omega}$ is indiscernible if and only if $\models \Theta((a_i)_{i < \omega})$.
\end{itemize}
\end{definition}
\begin{remark}
\thlabel{rem:hausdorff-thick-names}
The reason for the name Hausdorff is that this corresponds to the type spaces being Hausdorff, where formulas correspond to closed sets. The name thick is based on the notion of thick formulas, which were originally defined in the setting of full first-order logic (see also \cite{ben-yaacov_thickness_2003}).

We have added the notion of a Boolean theory. The name comes from the fact that the Lindenbaum-Tarski algebra of positive existential formulas forms a Boolean algebra, and this is in fact an equivalent assertion. In \cite{haykazyan_spaces_2019} these theories are called ``positively model complete'', but we think this name is more descriptive.

Through Morleyisation, Boolean theories are essentially the same as theories in full first-order logic, and so we will treat them as the same. The list of properties in \thref{def:hausdorff-thick} is really a hierarchy, so Boolean implies Hausdorff implies semi-Hausdorff implies thick.
\end{remark}
\begin{definition}
\thlabel{def:lascar-distance}
Let $a, a'$ be two tuples and let $B$ be any parameter set. We write $\d_B(a, a') \leq n$ if there are $a = a_0, a_1, \ldots, a_n = a'$ such that $a_i$ and $a_{i+1}$ are on a $B$-indiscernible sequence for all $0 \leq i < n$.
\end{definition}
\begin{fact}[{\cite[Proposition 1.5]{ben-yaacov_thickness_2003}}]
\thlabel{fact:thick-lascar-distance}
A theory is thick if and only if the property ``$\d_B(x, x') \leq n$'' is type-definable over $B$ for all $B$ and $n < \omega$.
\end{fact}
The following appears as \cite[Lemma 3.1]{Pi00} and \cite[Lemma 1.2]{ben-yaacov_simplicity_2003}.
\begin{lemma}
\thlabel{lem:indiscernible-sequence-based-on-long-sequence}
Let $A$ be any parameter set, $\kappa$ any cardinal, and let $\lambda = \beth_{(2^{|T| + |A| + \kappa})^+}$. Then for any sequence $(a_i)_{i < \lambda}$ of $\kappa$-tuples there is an $A$-indiscernible sequence $(b_i)_{i < \omega}$ such that for all $n < \omega$ there are $i_1 < \ldots < i_n < \lambda$ with $b_1 \ldots b_n \equiv_A a_{i_1} \ldots a_{i_n}$.
\end{lemma}
\begin{definition}
\thlabel{def:indiscernible-sequence-based-on-long-sequence}
In the notation of \thref{lem:indiscernible-sequence-based-on-long-sequence} we say that $(b_i)_{i < \omega}$ is \emph{based on} the sequence $(a_i)_{i < \lambda}$ (over $A$).
\end{definition}
Often the parameter set $A$ will be clear from the context (it will be the set that the new sequence is indiscernible over), so we may leave out the ``over $A$''.
\begin{definition}
\thlabel{def:lambda-t}
We write  $\lambda_\kappa:=\beth_{(2^\kappa)^+}$ for any cardinal $\kappa$ and $\lambda_T:=\lambda_{|T|}$.
\end{definition}
\begin{lemma}
\thlabel{lem:find-indisc-sequence-through-model}
Let $M$ be a $\lambda_T$-saturated e.c.\ model of a thick theory. Then $a \equiv_M b$ implies $\d_M(a, b) \leq 2$.
\end{lemma}
\begin{proof}
By thickness, $\d_M(x, y) \leq 1$ is $M$-type-definable. Let $\phi(x, y)$ be a finite conjunction of formulas in $\d_M(x, y) \leq 1$. It is enough to show that $\phi(x, a) \wedge \phi(x, b)$ is satisfiable, because then the partial type ``$\d_M(x, a) \leq 1$ and $\d_M(x, b) \leq 1$'' is finitely satisfiable.

Since $\phi$ is just a formula, we may as well assume $a$ and $b$ to be finite. Let $m$ denote the (finite) part of $M$ that appears in $\phi$. By $\lambda_T$-saturatedness of $M$ there is a sequence $(a_i)_{i < \lambda_T}$ in $M$ such that $a_i (a_j)_{j < i} \equiv_m a (a_j)_{j < i}$ for all $i < \lambda_T$. Using \thref{lem:indiscernible-sequence-based-on-long-sequence} we then find $m$-indiscernible $(a_i')_{i < \omega}$ based on $(a_i)_{i < \lambda_T}$. So $\models \phi(a_0', a_1')$, and thus there are $i_0 < i_1 < \lambda_T$ such that $M \models \phi(a_{i_0}, a_{i_1})$. By construction we have that $a_{i_1} a_{i_0} \equiv_m a a_{i_0}$, so $\models \phi(a_{i_0}, a)$. Since $a \equiv_M b$ and $a_{i_0} \in M$ we also have $\models \phi(a_{i_0}, b)$, and we are done.
\end{proof}
\begin{lemma}
\thlabel{lem:lstp-base-on-sequence}
Let $T$ be a thick theory. Let $B \supseteq A$, $\kappa$ any cardinal and set $\lambda = \lambda_{|T| + |B| + \kappa}$. Then for any $A$-indiscernible sequence $(a_i)_{i < \lambda}$ of $\kappa$-tuples, there is $B$-indiscernible $(a_i')_{i < \lambda}$ based on $(a_i)_{i < \lambda}$ such that $\d_A((a_i)_{i < \lambda}, (a_i')_{i < \lambda}) \leq 1$.
\end{lemma}
\begin{proof}
By \thref{lem:indiscernible-sequence-based-on-long-sequence} there is $B$-indiscernible $(b_i)_{i < \omega}$ based on $(a_i)_{i < \lambda}$. Prolong this to $B$-indiscernible $(b_i)_{i < \lambda}$. Define
\[
\Sigma((x_i)_{i < \lambda}) = \tp((b_i)_{i < \lambda}/B) \cup ``\d_A((x_i)_{i < \lambda}, (a_i)_{i < \lambda}) \leq 1 \mbox{''},
\]
and let $\Sigma_0(x_{i_1}, \ldots, x_{i_n}) \subseteq \Sigma((x_i)_{i < \lambda})$ be finite, only mentioning parameters in $B$ and $a_{i_1}, \ldots, a_{i_n}$. Let $j_1 < \ldots < j_n < \lambda$ be such that $a_{j_1} \ldots a_{j_n} \equiv_B b_1 \ldots b_n \equiv_B b_{i_0} \ldots b_{i_n}$. It follows from the proof of \thref{lem:indiscernible-sequence-based-on-long-sequence} that we may choose $j_1$ to be arbitrarily large below $\lambda$, so we may assume $j_1 > i_n$. Then $a_{j_1} \ldots a_{j_n}$ realises $\Sigma_0$. By compactness we find the required $(a_i')_{i < \lambda}$ as a realisation of $\Sigma$.
\end{proof}
The definition of dividing in positive theories is the same as in full first-order logic \cite{Pi00, ben-yaacov_simplicity_2003}. Following \cite{Pi00} we have to adjust forking to allow  infinite disjunctions because compactness can no longer guarantee disjunctions to be finite.
\begin{definition}
\thlabel{def:dividing-forking}
We say that a partial type $\Sigma(x, b)$ \emph{divides over $C$} if there is a $C$-indiscernible sequence $(b_i)_{i < \omega}$ with $b_0 \equiv_C b$ such that $\bigcup_{i < \omega} \Sigma(x, b_i)$ is inconsistent.

We say that $\Sigma(x, b)$ \emph{forks over $C$} if there is a (possibly infinite) set of formulas $\Phi(x)$ with parameters, each of which divides over $C$, such that $\Sigma(x, b)$ implies $\bigvee \Phi(x)$.

We write $a \ind_C^d b$ (or $a \ind_C^f b$) if $\tp(a/Cb)$ does not divide (fork) over $C$.
\end{definition}
\begin{remark}
\thlabel{rem:dividing-strong-finite-character}
We have that $\tp(a/Cb)$ divides over $C$ if and only if there is a formula $\phi(x, b) \in \tp(a/Cb)$ that divides over $C$. This follows directly from compactness. Note that for forking this is no longer necessarily true, because the disjunction may be infinite so we cannot apply compactness.
\end{remark}
For a type $p$ over a set $B$ and a subset $A\subseteq B$, the restriction of $p$ to $A$ is a type over $A$ which we denote by $p|_A$.
We recall the notions of an heir and a coheir, which also make sense in positive logic.
\begin{definition}
\thlabel{def:heir-coheir}
Let $M \subseteq B$ and let $p = \tp(a/B)$ be a type over $B$. We say that $p$ is a \emph{coheir} of $p|_M$, and write $a \ind_M^u B$, if $p$ is finitely satisfiable in $M$. We say that $p$ is an \emph{heir} of $p|_M$ if for every formula $\phi(x, y)$, with parameters in $M$, and every $b \in B$ such that $\phi(x, b) \in p$ there is some $b' \in M$ such that $\phi(x, b') \in p$. In this case we write $a \ind_M^h B$.
\end{definition}
\begin{remark}
\thlabel{rem:heir-dual-to-coheir}
As in full first-order logic, we have $A \ind_M^u B$ iff $B \ind_M^h A$.
\end{remark}
In \thref{prop:independence-strengths} we compare the above notions of independence further.

We recall that $2^{<\omega}$ is the set of all finite sequences of zeroes and ones. For $\eta, \nu \in 2^{<\omega}$ we write $\eta \preceq \nu$ if $\nu$ continues the sequence $\eta$. We write $\eta^\frown \nu$ for concatenation, so for example $\eta^\frown 0$ is the sequence $\eta$ with a 0 concatenated to it.
\begin{definition}
\thlabel{def:sop1}
Let $T$ be a theory and let $\phi(x, y)$ be a formula. We say that $\phi(x, y)$ has \emph{SOP$_1$} if there are $\psi(y_1, y_2)$ and $(a_\eta)_{\eta \in 2^{<\omega}}$ such that:
\begin{enumerate}[label=(\roman*)]
\item for every $\sigma \in 2^\omega$ the set $\{ \phi(x, a_{\sigma|n}) : n < \omega \}$ is consistent;
\item $\psi(y_1, y_2)$ implies that $\phi(x, y_1) \wedge \phi(x, y_2)$ is inconsistent, that is
\[
T \models \forall y_1 y_2 \neg[\psi(y_1, y_2) \wedge \exists x( \phi(x, y_1) \wedge \phi(x, y_2) )];
\]
\item for every $\eta, \nu \in 2^{<\omega}$ such that $\eta^\frown 0 \preceq \nu$ we have $\models \psi(a_{\eta^\frown 1}, a_\nu)$.
\end{enumerate}
We say that $T$ is \emph{NSOP$_1$} if no formula has SOP$_1$.
\end{definition}
\begin{remark}
\thlabel{rem:sop1-inconsistency-witness}
The idea of introducing the inconsistency witness $\psi(y_1, y_2)$ is due to Haykazyan and Kirby, \cite{haykazyan_existentially_2021}. In full first-order logic we can just take $\psi(y_1, y_2)$ to be $\neg \exists x(\phi(x, y_1) \wedge \phi(x, y_2))$, so we see that the definitions coincide there. The point of having $\psi$ is that the inconsistency in (iii) is again definable by a single formula for all relevant $\eta$ and $\nu$. This enables us to apply compactness to make the tree $(a_\eta)_{\eta \in 2^{<\omega}}$ as big as we wish.
\end{remark}
The following lemma, or rather its contrapositive, is what will actually be useful to us. If, in an NSOP$_1$ theory, we have two sequences that are `parallel to each other' in a certain way then we can transfer consistency for a formula along one sequence to the other. We will therefore give it the name ``parallel sequences lemma''.
\begin{lemma}[Parallel sequences lemma]
\thlabel{lem:sequences-give-sop1}
Suppose that $\phi(x, y)$ is a formula, and $(\bar{c}_i) = (c_{i,0}, c_{i,1})_{i \in I}$ is an infinite indiscernible sequence satisfying:
\begin{enumerate}[label=(\roman*)]
\item $c_{i,0} \equiv_{\bar{c}_{<i}} c_{i,1}$ for all $i \in I$;
\item $\{\phi(x; c_{i, 0}) : i \in I\}$ is consistent;
\item $\{\phi(x; c_{i, 1}) : i \in I\}$ is inconsistent.
\end{enumerate}
Then $T$ has SOP$_1$.
\end{lemma}
\begin{proof}
This is the same as \cite[Lemma 2.3]{kaplan_kim-independence_2020} and that proof mostly goes through. We sketch a few small changes that are needed. Obviously we already start with an indiscernible sequence and by compactness we can freely change the order type of $I$ preserving properties (i)--(iii). Then in the claim in that proof we need to make the array $(a_{i,0}, a_{i,1})$ sufficiently long. This can easily be done by elongating the original indiscernible sequence $(\bar{c}_i)$. Then we can find an indiscernible sequence based on $(\bar{a}_i) = (a_{i,0}, a_{i,1})$. Note that properties (1)--(3) in that claim are preserved by this operation. The reason for all this is because we need to start with an indiscernible sequence in \cite[Lemma 2.2]{kaplan_kim-independence_2020} as well. Then the rest of that proof goes through. Finally, inconsistency of $\{ \phi(x, c_{l,1}), \chi(x, d_{l', 0}) \}$ should be witnessed by some formula (similarly for \cite[Lemma 2.2]{kaplan_kim-independence_2020}), but the existence of such a witness easily follows from the construction of $\chi$.
\end{proof}

%% file: tex/global-ls-invariant-types.tex
\section{Global Lascar-invariant types}
\label{sec:global-ls-invariant-types}
The definition of Lascar strong types from the first-order setting easily generalises to (thick) positive logic, see \cite[Definition 3.13, Lemma 3.15]{Pi00} and \cite[Lemma 1.38]{ben-yaacov_simplicity_2003}\footnote{Simplicity is assumed here, but not used in the equivalence of the properties we mention. It is used for what is (iii) there.}.
\begin{definition}
\thlabel{def:lascar-strong-type}
We say $a$ and $b$ have the same \emph{Lascar strong type over $A$}, and write $a \equivls_A  b$, if the following equivalent conditions hold:
\begin{enumerate}[label=(\roman*)]
\item $\d_A(a, b) \leq n$ for some $n < \omega$;
\item for each bounded $A$-invariant equivalence relation $E(x, y)$ we have $E(a, b)$;
\item there are $\lambda_T$-saturated e.c.\ models $M_1, \ldots, M_n$, each containing $A$, and $a = a_0, \ldots, a_n = b$ such that $a_i \equiv_{M_{i+1}} a_{i+1}$ for all $0 \leq i < n$.
\end{enumerate}
We write $\Lstp(a/A)$ for the $\equivls_A$-equivalence class of $a$.
\end{definition}
\begin{lemma}
\thlabel{lem:equivalence-conditions-of-lstps}
The conditions in \thref{def:lascar-strong-type} are equivalent in a thick theory.
\end{lemma}
\begin{proof}
The equivalence of (i) and (ii) is proved in both \cite[Lemma 3.15]{Pi00} and \cite[Lemma 1.38]{ben-yaacov_simplicity_2003}. So we prove (i) $\Leftrightarrow$ (iii).

(i) $\implies$ (iii) Let $a = a_0, \ldots, a_n = b$ such that $a_i$ and $a_{i+1}$ are on an $A$-indiscernible sequence. Let $0 \leq i < n$, $(a'_j)_{j < \omega}$ be an $A$-indiscernible sequence with $a'_0 a'_1 = a_i a_{i+1}$ and let $M \supseteq A$ be some $\lambda_T$-saturated model. By \thref{lem:indiscernible-sequence-based-on-long-sequence} and an automorphism there is $M_{i+1} \equiv_A M$ such that $(a'_j)_{j < \omega}$ is $M_{i+1}$-indiscernible. So in particular $a_i \equiv_{M_{i+1}} a_{i+1}$, as required.

(iii) $\implies$ (i) By \thref{lem:find-indisc-sequence-through-model} $a_i \equiv_{M_i} a_{i+1}$ implies that $\d_{M_i}(a_i, a_{i+1}) \leq 2$ and as $A \subseteq M_i$ we are done.
\end{proof}
\thref{def:lascar-strong-type}(iii) allows for the following definition.
\begin{definition}
\thlabel{def:lascar-strong-automorphism}
Let $\Aut_f(\MM/A)$ be the group generated by
\[
\bigcup \{ \Aut(\MM / M) : M \text{ is a $\lambda_T$-saturated model and } A \subseteq M \}.
\]
We call its elements \emph{Lascar strong automorphisms}. It is clear that in a thick theory $a \equivls_A b$ precisely when there is $f \in \Aut_f(\MM / A)$ such that $f(a) = b$.
\end{definition}
\begin{remark}
\thlabel{rem:semi-hausdorff-drop-saturated}
If $T$ is semi-Hausdorff we may replace ``$\lambda_T$-saturated model'' by ``e.c.\ model'' in \thref{def:lascar-strong-type} and \thref{lem:equivalence-conditions-of-lstps}, see \cite[Proposition 3.13]{ben-yaacov_thickness_2003}.
\end{remark}
\begin{convention}
\thlabel{conv:global-realisations}
Recall that a \emph{global type} is a type over the monster model $\MM$. Building on \thref{conv:monster-model} about the monster model, we will use lowercase Greek letters $\alpha, \beta, \ldots$ for realisations of global types (in a bigger monster).
\end{convention}
\begin{definition}
\thlabel{def:ls-invariant-type}
A global type $q$ is called \emph{$A$-Ls-invariant}, short for \emph{$A$-Lascar-invariant}, if for a realisation $\alpha \models q$ we have that $b \equivls_A b'$ implies $\alpha b \equivls_A \alpha b'$.
\end{definition}
Note that this definition does not depend on the choice of $\alpha$. If $\alpha'$ is any other realisation of $q$, then $\alpha \equiv_\MM \alpha'$. So there is an automorphism $f$ of the bigger monster over $\MM$ with $f(\alpha) = \alpha'$. So if $b \equivls_A b'$ then $\alpha b \equivls_A \alpha b'$ and hence $f(\alpha)f(b) \equivls_{f(A)} f(\alpha)f(b')$ which is just $\alpha' b \equivls_A \alpha'b'$, since $f$ fixes $\MM$.
\begin{remark}
\thlabel{rem:local-ls-realisation-of-global-type}
Let $q$ be any global type in a thick theory, $\alpha \models q$ and let $A$ be any (small) parameter set. Then there is $a \in \MM$ with $a \equivls_A \alpha$. To see this, let $M \supseteq A$ be a $\lambda_T$-saturated model, and take any $a \models q|_M$.
\end{remark}
\begin{lemma}
\thlabel{lem:ls-invariance-easy-facts}
Suppose that $q$ is a global $A$-Ls-invariant type in a thick theory. Then:
\begin{enumerate}[label=(\roman*)]
\item for any $f \in \Aut(\MM/A)$ the type $f(q)$ is $A$-Ls-invariant;
\item for any $B \supseteq A$, $q$ is also $B$-Ls-invariant.
\end{enumerate}
\end{lemma}
\begin{proof}
Point (i) is straightforward, we prove (ii). Let $\alpha \models q$ and $b \equivls_B b'$. Then there are $\lambda_T$-saturated models $M_1, \ldots, M_n$, all containing $B$, and $b = b_0, \ldots, b_n = b'$ such that $b_i \equiv_{M_{i+1}} b_{i+1}$ for all $0 \leq i < n$. Let $0 \leq i < n$, it is enough to show $\alpha b_i \equiv_{M_{i+1}} \alpha b_{i+1}$. We have $b_i M_{i+1} \equivls_{A} b_{i+1} M_{i+1}$, so by $A$-Ls-invariance $\alpha b_i M_{i+1} \equivls_{A} \alpha b_{i+1} M_{i+1}$, which implies the desired result.
\end{proof}
\begin{lemma}
\thlabel{lem:extend-to-global-ls-invariant-type}
Let $T$ be thick and let $p = \tp(a/B)$ be a coheir over $M \subseteq B$. Then there is a global $M$-Ls-invariant type extending $p$.
\end{lemma}
\begin{proof}
Define
\[
\Gamma(x) = p(x) \cup \bigcup \{ \d_M(x c, x c') \leq 1 : c, c' \in \MM \text{ with } \d_M(c, c') \leq 1 \}.
\]
We claim that $\Gamma(x)$ is consistent. For finite $p_0(x) \subseteq p(x)$ there is $d \in M$ such that $d \models p_0$. Then for any $c, c'$ with $\d_M(c, c') \leq 1$ we have that $\d_M(dc, dc') \leq 1$ because $d$ is in $M$. Any maximal extension of $\Gamma(x)$ will be a desired global $M$-Ls-invariant type.
\end{proof}
\begin{definition}
\thlabel{def:extending-lstp}
For $A \subseteq B$ we say that \emph{$\Lstp(c/B)$ extends $\Lstp(c'/A)$} if $c \equivls_A c'$.
\end{definition}
\begin{corollary}
\thlabel{cor:type-over-model-extends-to-global-ls-invariant-type}
In a thick theory we have that $\Lstp(a/M)$ extends to a global $M$-Ls-invariant type for any $a$ and $M$.
\end{corollary}
\begin{proof}
By \thref{lem:extend-to-global-ls-invariant-type} we have that $p = \tp(a/M)$ extends to some global $M$-Ls-invariant type $q$. For $\alpha \models q$ let $a' \equivls_M \alpha$. Then there is $f \in \Aut(\MM/M)$ such that $f(a') = a$. So by \thref{lem:ls-invariance-easy-facts}(i) $f(q)$ is global $M$-Ls-invariant and is exactly what we need.
\end{proof}
\begin{definition}
\thlabel{def:global-invariant-independence}
For a type $p = \tp(a/Cb)$ write $a \ind_C^{iLs} b$ if there is a global $C$-Ls-invariant extension of $p$.
\end{definition}
\begin{proposition}
\thlabel{prop:independence-strengths}
In any thick theory $T$ we have
\[
a \ind_C^u b \implies
a \ind_C^{iLs} b \implies
a \ind_C^f b \implies
a \ind_C^d b.
\]
\end{proposition}
\begin{proof}
Standard, but we write out the arguments to check they hold with the slightly changed definitions for positive logic. The first implication is precisely \thref{lem:extend-to-global-ls-invariant-type}, while the last implication is direct from the definition of dividing and forking.

We prove the middle implication. Assume $a \ind_C^{iLs} b$ and suppose for a contradiction that $p(x) = \tp(a/Cb)$ forks over $C$. Let $\Phi(x)$ be a set of formulas that all divide over $C$, such that $p(x)$ implies $\bigvee\Phi(x)$. Let $q$ be a global $C$-Ls-invariant extension of $p$ and let $\alpha \models q$. Then there must be $\phi(x, d) \in \Phi(x)$ such that $\models \phi(\alpha, d)$. Let $(d_i)_{i < \omega}$ be $C$-indiscernible with $d_0 = d$. For all $i < \omega$ we have $d \equivls_C d_i$ and thus $\alpha d \equivls_C \alpha d_i$. So in particular $\alpha \models \{ \phi(x, d_i) : i < \omega \}$, which contradicts that $\phi(x, d)$ divides over $C$.
\end{proof}
In the remainder of this section we will develop tensoring of global Ls-invariant types. This comes down to verifying that the usual constructions for global invariant types (see e.g.\ \cite[Section 2.2.1]{simon}) work when we carefully replace types by Lascar strong types everywhere.
\begin{lemma}
\thlabel{lem:tensor-lemma}
Suppose $T$ is thick, $q$ a global $A$-Ls-invariant type and $p = \Lstp(a^*/A)$. Then, for $\beta \models q$, the set
\[
R_{p, q}(A) = \{(a,b) \in \MM : a \equivls_A a^* \text{ and } b \equivls_{Aa} \beta \}
\]
is (the set of  realisations of) a Lascar strong type over $A$.
\end{lemma}
\begin{proof}
Clearly this does not depend on the choice of $a^*$ or $\beta$. The set is non-empty, as for any $b \equivls_{Aa^*} \beta$ we have $(a^*, b) \in R_{p, q}(A)$.

Let $(a, b), (a',b') \in R_{p, q}(A)$. Then $a \equivls_A a^* \equivls_A a'$, so by $A$-Ls-invariance $ab \equivls_A a \beta \equivls_A a' \beta \equivls_A a'b'$. Conversely, suppose $(a,b) \in R_{p, q}(A)$ and $ab \equivls_A a'b'$. Then $a' \equivls_A a \equivls_A a^*$. Furthermore, by $A$-Ls-invariance $\beta ab \equivls_A \beta a'b'$, so applying an automorphism to $b \equivls_{Aa} \beta$ we get $b' \equivls_{Aa'} \beta$ and conclude that $(a',b') \in R_{p, q}(A)$.
\end{proof}
\begin{theorem}
\thlabel{thm:tensor-theorem}
Suppose $T$ is thick with global $A$-Ls-invariant types $q$ and $r$. Then there is a unique global $A$-Ls-invariant type $q \otimes r$ such that for any $\alpha \models q$, $\beta \models r$ and $(\alpha', \beta') \models q \otimes r$ the following are equivalent for all $B \supseteq A$ and for all $a, b$:
\begin{enumerate}[label=(\roman*)]
\item $ab \equivls_B \alpha' \beta'$,
\item $a \equivls_B \alpha$ and $b \equivls_{Ba} \beta$.
\end{enumerate}
In particular this implies that also $\alpha' \models q$ and $\beta' \models r$.
\end{theorem}
\begin{proof}
Throughout, let $\alpha \models q$ and $\beta \models r$. For $B \supseteq A$, denote by $q_B$ the Lascar strong type $\Lstp(\alpha/B)$. By \thref{lem:ls-invariance-easy-facts}(ii) and \thref{lem:tensor-lemma}, we have a well-defined Lascar strong type $R_{q_B, r}(B)$.

\vspace{\baselineskip}\noindent
\underline{Claim.} For $A \subseteq B \subseteq C$ we have $R_{q_C, r}(C) \subseteq R_{q_B, r}(B)$.

\noindent
\underline{Proof of claim.} Let $(a, b) \in R_{q_C, r}(C)$. Then $a \equivls_C \alpha$ and $b \equivls_{Ca} \beta$. Hence $a \equivls_B \alpha$ and $b \equivls_{Ba} \beta$, so $(a,b) \in R_{q_B, r}(B)$.

\vspace{\baselineskip}
For $M \supseteq A$ a $\lambda_T$-saturated model $R_{q_M, r}(M)$ corresponds to the usual syntactic type over $M$. So viewing $R_{q_M, r}(M)$ as a set of formulas over $M$, we get by the claim that the following is a well-defined global type:
$$
q \otimes r := \bigcup \{ R_{q_M, r}(M) : M \text{ is a $\lambda_T$-saturated model and } A \subseteq M \}.
$$

First we verify that $q \otimes r$ satisfies the universal property we claimed. So let $(\alpha', \beta') \models q \otimes r$ and $B \supseteq A$. Let $M \supseteq B$ be a $\lambda_T$-saturated model and pick $a'b' \equivls_M \alpha' \beta'$. Then by construction $(a', b') \in R_{q_M, r}(M)$ and so by the claim $(a', b') \in R_{q_B, r}(B)$. So for any $a,b$ we have $ab \equivls_B \alpha' \beta'$ if and only if $ab \equivls_B a'b'$ if and only if $(a, b) \in R_{q_B, r}(B)$ if and only if $a \equivls_B \alpha$ and $b \equivls_{Ba} \beta$.

Uniqueness follows because any global type satisfying this universal property must restrict to $R_{q_M, r}(M) = (q \otimes r)|_M$ for all $\lambda_T$-saturated $M \supseteq A$.

Finally we prove $A$-Ls-invariance. Let $d \equivls_A d'$, and pick $a,b$ in $\MM$ such that $ab \equivls_{Add'} \alpha' \beta'$. So $a \equivls_{Add'} \alpha'$ and thus by $A$-Ls-invariance of $q$:
\[
ad \equivls_A \alpha' d \equivls_A \alpha' d' \equivls_A ad'.
\]
Then $A$-Ls-invariance of $r$ gives us $\beta' ad \equivls_A \beta' ad'$. From the universal property we get $b \equivls_{Add'a} \beta'$, so $abd \equivls_A abd'$. Because, by assumption, $ab \equivls_{Add'} \alpha' \beta'$, we conclude that $\alpha' \beta' d \equivls_A \alpha' \beta' d'$ and we are done.
\end{proof}
\begin{lemma}
\thlabel{lem:tensor-facts}
For any global $A$-Ls-invariant types $p, q, r$ in a thick theory we have:
\begin{enumerate}[label=(\roman*)]
\item associativity: $(p \otimes q) \otimes r = p \otimes (q \otimes r)$;
\item monotonicity: for any $q'(x_0)= q(x_0, x_1)|_{x_0}\subseteq q(x_0, x_1)$ and $r'(y_0) =r(y_0, y_1)|_{y_0}\subseteq r(y_0, y_1)$, we have $q' \otimes r' \subseteq q \otimes r$.
\end{enumerate}
\end{lemma}
\begin{proof}
(i) Let $(\alpha, \beta, \gamma) \models (p \otimes q) \otimes r$ and $(\alpha', \beta', \gamma') \models p \otimes (q \otimes r)$. We will prove that $\alpha \beta \gamma \equivls_B \alpha' \beta' \gamma'$ for all $B \supseteq A$. Let $abc \equivls_B \alpha \beta \gamma$, then $b \equivls_{Ba} \beta$ and $c \equivls_{Bab} \gamma$. So we have $bc \equivls_{Ba} \beta' \gamma'$. Since also $a \equivls_B \alpha$ we thus conclude that $abc \equivls_B \alpha' \beta' \gamma'$.

(ii) Let $(\alpha, \beta) = ((\alpha_0, \alpha_1), (\beta_0, \beta_1)) \models q \otimes r$ and let $ab \equivls_B \alpha \beta$, where $B \supseteq A$ is arbitrary. Then in particular $a_0 \equivls_B \alpha_0$ and $b_0 \equivls_{Ba_0} \beta_0$. So if we let $(\alpha', \beta') \models q' \otimes r'$ then $\alpha_0 \beta_0 \equivls_B a_0 b_0 \equivls_B \alpha' \beta'$. So $(\alpha_0, \beta_0) \models q' \otimes r'$ and we are done.
\end{proof}
\begin{definition}
\thlabel{def:ordinal-tensor-and-morley-sequence}
For a global $A$-Ls-invariant type, we define $q^{\otimes \delta}$ for an ordinal $\delta \geq 1$ by induction as follows.
\begin{itemize}
\item $q^{\otimes 1} = q$,
\item $q^{\otimes \delta + 1} = q^{\otimes \delta} \otimes q$,
\item $q^{\otimes \delta} = \bigcup_{\gamma < \delta} q^{\otimes \gamma}$ when $\delta$ is a limit.
\end{itemize}
A \emph{Morley sequence in $q$ (over $A$)} is a sequence $(a_i)_{i < \delta}$ such that $(a_i)_{i < \delta} \equivls_A (\alpha_i)_{i < \delta}$, where $(\alpha_i)_{i < \delta} \models q^{\otimes \delta}$.
\end{definition}
Note that we define Morley sequences in terms of Lascar strong types here. So saying that $(a_i)_{i < \omega}$ is a Morley sequence in $q$ over $A$ is generally a stronger statement than just saying $(a_i)_{i < \omega} \models q^{\otimes \omega}|_A$. Of course, if $A$ is a $\lambda_T$-saturated model in a thick theory then the two coincide.
\begin{lemma}
\thlabel{lem:tensor-smaller-ordinal}
Suppose that $q$ is a global $A$-Ls-invariant type and let $(\alpha_i)_{i < \delta} \models q^{\otimes \delta}$. Then for any strictly increasing sequence $(i_\eta)_{\eta < \gamma}$ in $\delta$ we have that $(\alpha_{i_\eta})_{\eta < \gamma} \models q^{\otimes \gamma}$.
\end{lemma}
\begin{proof}
From the construction of $q^{\otimes \delta}$ it is clear that for $\gamma < \delta$ and $(\alpha_i)_{i < \delta} \models q^{\otimes \delta}$ we have $(\alpha_i)_{i < \gamma} \models q^{\otimes \gamma}$.

We prove the lemma by induction to $\gamma$. The base case and the limit step are easy, so we prove the successor step. So suppose $(\alpha_{i_\eta})_{\eta < \gamma} \models q^{\otimes \gamma}$. We will prove $(\alpha_{i_\eta})_{\eta < \gamma}\alpha_{i_\gamma} \equivls_B \alpha_{<\gamma}\alpha_\gamma$ for all $B \supseteq A$. Let $a_{\leq i_\gamma} \equivls_B \alpha_{\leq i_\gamma}$, then in particular $(a_{i_\eta})_{\eta < \gamma} \equivls_B (\alpha_{i_\eta})_{\eta < \gamma}$ and $a_{i_\gamma} \equivls_{B(a_{i_\eta})_{\eta < \gamma}} \alpha_{i_\gamma}$. By the induction hypothesis and the universal property this means $(a_{i_\eta})_{\eta < \gamma} a_{i_\gamma} \equivls_B \alpha_{< \gamma} \alpha_\gamma$, which concludes the successor step. 
\end{proof}
By \thref{lem:tensor-smaller-ordinal} we have that $(a_i)_{i < \delta} \models q^{\otimes \delta}|_A$ if and only if $(a_{i_1}, \ldots, a_{i_n}) \models q^{\otimes n}|_A$ for all $i_1 < \ldots < i_n < \delta$. From this perspective it makes sense to make the following convention, even though we technically have not defined $q^{\otimes I}$ for arbitrary linear orders $I$.
\begin{convention}
\thlabel{conv:modelling-tensor-type}
Let $I$ be any linear order and let $q$ be a global $A$-Ls-invariant type. Then by $(a_i)_{i \in I} \models q^{\otimes I}|_A$ we mean that for any $i_1 < \ldots < i_n$ in $I$ we have $(a_{i_1}, \ldots, a_{i_n}) \models q^{\otimes n}|_A$.
\end{convention}
\begin{proposition}
\thlabel{prop:properties-morley-sequence-in-q}
For any Morley sequence $(a_i)_{i < \delta}$ in a global $A$-Ls-invariant type $q$ the following hold:
\begin{enumerate}[label=(\roman*)]
\item for all $i < \delta$, $a_i \equivls_{Aa_{<i}} \alpha$, where $\alpha \models q$;
\item $(a_i)_{i < \delta}$ is $A$-indiscernible.
\end{enumerate}
\end{proposition}
\begin{proof}
We first prove (i). Let $(\alpha_i)_{i < \delta} \models q^{\otimes \delta}$ and $i < \delta$. Then $a_{<i} a_i \equivls_A \alpha_{<i} \alpha_i$. As $\alpha_{<i} \alpha_i \models q^{\otimes i} \otimes q$, the universal property yields $a_i \equivls_{Aa_{<i}} \alpha_i$, as required.

For (ii), consider any $i_1 < \ldots < i_n < \delta$. By \thref{lem:tensor-smaller-ordinal} we have that 
$ \alpha_{i_1} \ldots \alpha_{i_n} \equiv_\MM \alpha_1 \ldots \alpha_n$, so in particular $ \alpha_{i_1} \ldots \alpha_{i_n} \equivls_A \alpha_1 \ldots \alpha_n$. 
As $(a_i)_{i<\delta}\equiv_A (\alpha_i)_{i<\delta}$, we conclude that 
$ a_{i_1} \ldots a_{i_n} \equiv_A a_1 \ldots a_n$.
\end{proof}

%% file: tex/kim-dividing.tex
\section{Kim-dividing}
\label{sec:kim-dividing}
The idea of Kim-dividing is to restrict dividing witnesses to non-forking Morley sequences. Proving the existence of such sequences over arbitrary sets turns out to be difficult, and is in fact an open problem for NSOP$_1$ theories in full first-order logic, see \cite[Remark 2.6, Question 6.6]{dobrowolski_independence_2019}. In \cite{kaplan_kim-independence_2020} this is solved by using Morley sequences in some global invariant type. In full first-order logic any type over a model extends to a global invariant type. In positive logic we need to assume the theory to be semi-Hausdorff to find global invariant extensions \cite[Lemma 3.11]{ben-yaacov_thickness_2003}, because they may not exist otherwise (see Subsection \ref{subsec:thick-non-semi-hausdorff}). In the more general setting of thick positive theories we can always find global Ls-invariant extensions and the notion of a Morley sequence makes sense in such a global Ls-invariant type, see Section \ref{sec:global-ls-invariant-types}. Since we can generally only extend types over e.c.\ models to global Ls-invariant types, we will consider Kim-dividing only over e.c.\ models (cf.\ \thref{q:sets}).
\begin{definition}
\thlabel{def:q-dividing}
Let $\Sigma(x, b)$ be a partial type in a thick theory, possibly with parameters in $M$, and let $q$ be a global $M$-Ls-invariant extension of $\tp(b/M)$. We say that $\Sigma(x, b)$ \emph{$q$-divides over $M$} if for any (equivalently: some) Morley sequence $(b_i)_{i < \omega}$ in $q$ (over $M$) the set $\bigcup_{i < \omega} \Sigma(x, b_i)$ is inconsistent.
\end{definition}
By compactness $q$-dividing does not depend on the length of the Morley sequence, as long as it is infinite.
\begin{proposition}
\thlabel{prop:q-dividing-equivalent-statements}
Let $T$ be thick, let $q$ be a global $M$-Ls-invariant extension of $\tp(b/M)$ and write $p(x, y) = \tp(ab/M)$. Then the following are equivalent.
\begin{enumerate}[label=(\roman*)]
\item The type $p(x, b)$ does not $q$-divide.
\item For any $f \in \Aut(\MM/M)$ the type $p(x, b)$ does not $f(q)$-divide.
\item For any (equivalently some) $(b_i)_{i < \omega} \models q^{\otimes \omega}|_M$ the set $\bigcup_{i < \omega} p(x, b_i)$ is consistent.
\item There is an $Ma$-indiscernible sequence $(b_i)_{i < \omega} \models q^{\otimes \omega}|_M$ with $b_0 = b$.
\end{enumerate}
\end{proposition}
\begin{proof}
\underline{(i) $\Leftrightarrow$ (ii) $\Leftrightarrow$ (iii)} This follows because consistency of $\bigcup_{i < \omega} p(x, b_i)$ only depends on $\tp((b_i)_{i < \omega}/M)$, together with the fact that given a Morley sequence $(b_i)_{i < \omega}$ in $q$ we have that $(f(b_i))_{i < \omega}$ is a Morley sequence in $f(q)$.

\underline{(i) $\implies$ (iv)} Let $(b_i)_{i < \lambda}$ be a Morley sequence in $q$ for big enough $\lambda$. Let $a^*$ realise $\bigcup_{i < \lambda} p(x, b_i)$ and let $(b_i')_{i < \omega}$ be $Ma^*$-indiscernible, based on $(b_i)_{i < \lambda}$. So there is $i < \lambda$ such that $a^* b_0' \equiv_M a^* b_i \equiv_M ab$. Let $(b_i'')_{i < \omega}$ with $b''_0 = b$ be such that $a (b_i'')_{i < \omega} \equiv_M a^* (b_i')_{i < \omega}$. Then $(b_i'')_{i < \omega}$ is $Ma$-indiscernible. Furthermore, since $(b_i)_{i < \lambda}$ was already $M$-indiscernible, we have $(b_i'')_{i < \omega} \equiv_M (b_i')_{i < \omega} \equiv_M (b_i)_{i < \omega}$, so $(b_i'')_{i < \omega} \models q^{\otimes \omega}|_M$.

\underline{(iv) $\implies$ (iii)} For such an $Ma$-indiscernible sequence $(b_i)_{i < \omega}$ we have $ab = ab_0 \equiv_M ab_i$ for all $i < \omega$. So $a$ realises $\bigcup_{i < \omega} p(x, b_i)$.
\end{proof}
\begin{proposition}
\thlabel{prop:q-dividing-extending-to-complete-types}
Let $T$ be thick, let $\Sigma(x, b)$ be a partial type with parameters in $M$ and let $q$ be a global $M$-Ls-invariant extension of $\tp(b/M)$. If $\Sigma(x, b)$ does not $q$-divide over $M$ then there is a complete $p(x, b) \supseteq \Sigma(x, b)$ that does not $q$-divide over $M$.
\end{proposition}
\begin{proof}
Let $(b_i)_{i < \lambda} \models q^{\otimes \lambda}|_M$ with $b_0 = b$. Then there is some $a \models \bigcup_{i < \lambda} \Sigma(x, b_i)$. Then, assuming we chose $\lambda$ large enough, there is some $i_0 < \lambda$ such that for infinitely many $i < \lambda$ we have $a b_i \equiv_M a b_{i_0}$. Set $p(x, y) = \tp(a b_{i_0} / M)$, then $p(x, b_{i_0})$ does not $q$-divide while also $\Sigma(x, b_{i_0}) \subseteq p(x, b_{i_0})$. By invariance $p(x, b)$ does not $q$-divide.
\end{proof}
The following lemma is the core of the connection between Kim-dividing and NSOP$_1$ theories. It tells us that $q$-dividing does not depend on the global Lascar-invariant type $q$. More discussion on the origins of this lemma can be found in \cite{kaplan_kim-independence_2020}. Briefly put: Kim proved that in simple theories a formula divides with respect to every Morley sequence if and only if it divides with respect to some Morley sequence \cite[Proposition 2.1]{Kim98}. The lemma below is an analogue of that for NSOP$_1$ theories.
\begin{proposition}[Kim's lemma]
\thlabel{prop:kims-lemma}
If $T$ is thick NSOP$_1$, then $q$-dividing does not depend on $q$. That is, if $q$ and $r$ are global $M$-invariant types extending $\tp(b/M)$ then a partial type $\Sigma(x, b)$ $q$-divides if and only if it $r$-divides.
\end{proposition}
\begin{proof}
This is essentially the proof of \cite[Proposition 3.15]{kaplan_kim-independence_2020}, adapted to the thick positive logic setting. By \thref{prop:q-dividing-equivalent-statements}(ii) we may assume that $q$ and $r$ extend $\Lstp(b/M)$. Suppose that $\Sigma(x, b)$ does not $q$-divide while it $r$-divides. We will prove that $T$ has SOP$_1$. Let $(\bar{b}_i)_{i < \omega} = (b_{i,0}, b_{i,1})_{i<\omega}$ be a Morley sequence in $q \otimes r$. By \thref{lem:tensor-facts}(ii) and induction, $(b_{i,0})_{i < \omega}$ and $(b_{i,1})_{i < \omega}$ are Morley sequences in $q$ and $r$ respectively.

Since $\Sigma(x, b)$ $r$-divides, the set $\bigcup_{i < \omega} \Sigma(x, b_{i,1})$ is inconsistent. So by compactness there is an $M$-formula $\phi(x, y) \in \Sigma(x, y)$ such that $\{\phi(x, b_{i,1}) : i < \omega\}$ is inconsistent. Because $\Sigma(x, b)$ does not $q$-divide we have that $\{\phi(x, b_{i,0}) : i < \omega\}$ is consistent.

We wish to apply the parallel sequences lemma (\thref{lem:sequences-give-sop1}) to $\phi(x, y)$ and $(\bar{b}_i)_{i < \omega^\op}$ where $\omega^\op$ carries the opposite order of $\omega$. So we are left to prove that $b_{i,0} \equiv_{M \bar{b}_{> i}} b_{i,1}$ for all $i < \omega$. We do this by proving that $b_{i,0} (\bar{b_i})_{i < j < n} \equiv_M b_{i,1} (\bar{b_i})_{i < j < n}$ for all $i < n < \omega$. Let $(\bar{\beta}_i)_{i < \omega} \models (q \otimes r)^{\otimes \omega}$. By \thref{lem:tensor-facts}(i) we have $(q \otimes r)^{\otimes n} = (q \otimes r)^{\otimes i+1} \otimes (q \otimes r)^{\otimes n - i - 1}$. So we have $\bar{\beta}_{< n} \models (q \otimes r)^{\otimes i+1} \otimes (q \otimes r)^{\otimes n - i - 1}$ and as $\bar{b}_{<n} \equivls_M \bar{\beta}_{<n}$ we have $(\bar{b}_j)_{i<j<n} \equivls_{M \bar{b}_{\leq i}} (\bar{\beta}_j)_{i<j<n}$. As $b_{i,0} \equivls_M b \equivls_M b_{i,1}$ we get by $M$-Ls-invariance that $b_{i,0} (\bar{\beta}_j)_{i<j<n} \equivls_M b_{i,1} (\bar{\beta}_j)_{i<j<n}$. Putting the two together yields the required result.
\end{proof}
\begin{definition}
\thlabel{def:kim-dividing}
We say $\Sigma(x, b)$ \emph{Kim-divides (over $M$)} if it $q$-divides for some global $M$-Ls-invariant $q$ that extends $\tp(b/M)$. We write $a \ind_M^K b$ when $\tp(a/Mb)$ does not Kim-divide over $M$ and call this \emph{Kim-independence}.
\end{definition}
\begin{remark}
\thlabel{rem:kim-dividing-over-ec-models}
By \thref{lem:extend-to-global-ls-invariant-type} we can extend any type over an e.c.\ model $M$ in a thick theory to a global $M$-Ls-invariant type. So assuming NSOP$_1$, we have by \thref{prop:kims-lemma} that $\tp(a/Mb)$ Kim-divides if and only if it $q$-divides for \emph{any} global $M$-invariant extension $q$ of $\tp(b/M)$.
\end{remark}
In some constructions it will be necessary to stay within the same Lascar strong type. For this we introduce the technical tool of $q$-Ls-dividing.
\begin{definition}
\thlabel{def:q-ls-dividing}
Let $T$ be thick and let $q$ be a global $M$-Ls-invariant extension of $\Lstp(b/M)$. We say that $\Lstp(a/Mb)$ \emph{does not $q$-Ls-divide (over $M$)} if there is a Morley sequence $(b_i)_{i < \omega}$ in $q$ with $b_0 = b$ that is $Ma$-indiscernible.
\end{definition}
\begin{remark}
\thlabel{rem:q-dividing-length-does-not-matter}
The length of the Morley sequence does not matter in \thref{def:q-ls-dividing}, as long as it is infinite. However, the argument here takes a little more care than for $q$-dividing.

One direction is clear: if there is an $Ma$-indiscernible Morley sequence $(b_i)_{i < \delta}$ in $q$ for some $\delta \geq \omega$, then we can just take an initial segment. For the other direction we let $N \supseteq M$ be $\lambda_T$-saturated and $(b_i)_{i < \omega} \models q^{\otimes \omega}|_N$. Then $(b_i)_{i < \omega}$ is a Morley sequence in $q$. Applying a Lascar strong automorphism we find $a' b_0 \equivls_M ab$ such that $(b_i)_{i < \omega}$ is $Ma'$-indiscernible. Let $n$ be such that $\d_M(a'b_0, ab) \leq n$. Consider the set of formulas
\[
q^{\otimes \delta}|_N((y_i)_{i < \delta}) \cup \text{``$(xy_i)_{i < \delta}$ is $M$-indiscernible''} \cup \d_M(x y_0, ab) \leq n.
\]
This set is finitely satisfiable, hence it has a realisation. So we find an $Ma''$-indiscernible Morley sequence $(b_i')_{i < \delta}$ in $q$ with $a'' b_0' \equivls_M ab$. The result follows by applying a Lascar strong automorphism.
\end{remark}
\begin{lemma}
\thlabel{lem:q-ls-dividing-vs-q-dividing}
Let $T$ be thick and let $q$ be a global $M$-Ls-invariant extension of $\Lstp(b/M)$. A type $p = \tp(a/Mb)$ does not $q$-divide if and only if there is a realisation $a' \models p$ such that $\Lstp(a'/Mb)$ does not $q$-Ls-divide.
\end{lemma}
\begin{proof}
The right to left direction is clear by \thref{prop:q-dividing-equivalent-statements}(iv). For the other direction we let $(b_i')_{i < \omega}$ be a Morley sequence in $q$ with $b_0' = b$. By \thref{prop:q-dividing-equivalent-statements}(iv) there is $(b_i)_{i < \omega} \models q^{\otimes \omega}|_M$ that is $Ma$-indiscernible with $b_0 = b$. Pick $a'$ such that $a'(b_i')_{i < \omega} \equiv_M a(b_i)_{i < \omega}$ and we are done.
\end{proof}
\begin{corollary}
\thlabel{cor:q-ls-dividing-is-q-dividing}
Let $T$ be thick and let $q$ be a global $M$-Ls-invariant extension of $\Lstp(b/M)$. Suppose that there is $M \subseteq N \subseteq b$ such that $N$ is $\lambda_T$-saturated. Then $\tp(a/Mb)$ does not $q$-divide if and only if $\Lstp(a/Mb)$ does not $q$-Ls-divide.
\end{corollary}
\begin{proof}
By \thref{lem:q-ls-dividing-vs-q-dividing} we only need to prove the left to right direction. So suppose that  $\tp(a/Mb)$ does not $q$-divide. Then there is $a'$ with $a' \equiv_{Mb} a$ such that $\Lstp(a'/Mb)$ does not $q$-Ls-divide. In particular we have that $a'b \equiv_N ab$, so $a'b \equivls_M ab$. It follows that $\Lstp(a/Mb)$ does not $q$-Ls-divide.
\end{proof}
\begin{proposition}
\thlabel{prop:obvious-properties-kim-dividing}
In a thick NSOP$_1$ theory Kim-independence always satisfies the following properties.
\begin{enumerate}[label=(\roman*)]
\item Strong finite character: if $a \nind_M^K b$ then there is a formula $\varphi(x, b, m) \in \tp(a/Mb)$ such that for any $a' \models \varphi(x, b, m)$ we have $a' \nind_M^K b$.
\item Existence over models: $a \ind_M^K M$.
\item Monotonicity: $aa' \ind_M^K bb' \implies a \ind_M^K b$.
\end{enumerate}
\end{proposition}
\begin{proof}
All follow directly from the definitions, using compactness for (i).
\end{proof}
\begin{remark}
\thlabel{rem:kim-dividing-implies-dividing}
Let $T$ be a thick theory. Then Kim-dividing implies dividing because any Morley sequence in some $q$ is in particular an indiscernible sequence. So by \thref{prop:independence-strengths}:
\[
a \ind_M^u b \implies
a \ind_M^{iLs} b \implies
a \ind_M^f b \implies
a \ind_M^d b \implies
a \ind_M^K b
\]
\end{remark}
\begin{proposition}
\thlabel{prop:ls-dividing-extension}
Let $T$ be a thick theory, $M$ an e.c.\ model of $T$, and let $a, b, c$ be tuples. Let also $q(x, y)$ be a global $M$-Ls-invariant extension of $\Lstp(bc/M)$ and write $r(x) = q|_x$. If $\Lstp(a/Mb)$ does not $r$-Ls-divide then there is $c^* b \equivls_M cb$ such that $\Lstp(a/Mbc^*)$ does not $q$-Ls-divide.
\end{proposition}
\begin{proof}
Let $(b_i, c_i)_{i < \lambda}$ be a Morley sequence over $M$ in $q$ for some big enough $\lambda$. Since $(b_i)_{i < \lambda}$ is a Morley sequence over $M$ in $r$ and $\Lstp(a/Mb)$ does not $r$-divide there is $a'$ with $a'b_0 \equivls_M ab$ such that $(b_i)_{i < \lambda}$ is $Ma'$-indiscernible.

Let $f \in \Aut_f(\MM / M)$ be such that $f(a'b_0) = ab$ and put $(b_i', c_i') = (f(b_i), f(c_i))$. Then $b_0' = b$, $(b_i')_{i < \lambda}$ is $Ma$-indiscernible and $(b_i', c_i')_{i < \lambda}$ is a Morley sequence over $M$ in $q$.

Let $M' \supseteq Ma$ be $\lambda_T$-saturated and use \thref{lem:lstp-base-on-sequence} to find $M'$-indiscernible $(b_i'', c_i'')_{i < \lambda}$ based on $(b_i', c_i')_{i < \lambda}$ and such that $\d_M((b_i'', c_i'')_{i < \lambda}, (b_i', c_i')_{i < \lambda}) \leq 1$. In particular $(b_i'', c_i'')_{i < \lambda}$ is a Morley sequence over $M$ in $q$. Let $i < \lambda$ be such that $b_0'' \equiv_{M'} b_i'$ then $b_0'' \equivls_{Ma} b_i' \equivls_{Ma} b_0'= b$. So there is $g \in \Aut_f(\MM/Ma)$ such that $g(b_0'') = b$. Set $c^* = g(c_0'')$, so $bc^* \equivls_M b_0'' c_0'' \equivls_M b_0' c_0' \equivls_M b_0 c_0 \equivls_M b c$. Finally, since $(g(b_i''), g(c_i''))_{i < \lambda}$ is a Morley sequence over $M$ in $q$ starting with $bc^*$ that is $Ma$-indiscernible, we conclude that $\Lstp(a/Mbc^*)$ does not $q$-Ls-divide.
\end{proof}
\begin{corollary}[Extension]
\thlabel{cor:extension}
In a thick NSOP$_1$ theory we have that if $a \ind_M^K b$ then for any $c$ there is $c' \equivls_{Mb} c$ such that $a \ind_M^K bc'$.
\end{corollary}
\begin{proof}
We first prove a weaker version where we conclude $c' \equiv_{Mb} c$ instead of $c' \equivls_{Mb} c$.

Let $q(x, y)$ be an $M$-Ls-invariant extension of $\Lstp(bc/M)$ and write $r(x) = q|_x$, where $x$ matches $b$. Since $a \ind_M^K b$ there is $a'b \equiv_M ab$ such that $\Lstp(a'/Mb)$ does not $r$-Ls-divide. By \thref{prop:ls-dividing-extension} we thus find $b c^* \equivls_M b c$ such that $\Lstp(a'/Mbc^*)$ does not $q$-Ls-divide. Letting $c'$ be such that $abc' \equiv_M a'bc^*$ then satisfies $a \ind_M^K bc'$ and furthermore we have $bc' \equiv_M bc^* \equiv_M bc$.

Now we use the weaker version to prove the full version. Let $N \supseteq Mb$ be some $\lambda_T$-saturated model. By the above we can find $N' \equiv_{Mb} N$ such that $a \ind_M^K N'$. Then using the above again we find $c' \equiv_{N'} c$ such that $a \ind_M^K N'c'$. Since $Mb \subseteq N'$ we thus get $c' \equivls_{Mb} c$ and $a \ind_M^K bc'$, as required.
\end{proof}

%% file: tex/CR.tex
\section{EM-modelling and parallel-Morley sequences}
\label{sec:cr-morley}
In this section we will introduce some tools which will be useful later in certain tree constructions.
\begin{definition}[{\cite[Defintion 2.1]{KKS}}]
\thlabel{def:tree-languages}
The \emph{Shelah language}
\[
L_s = \{\trianglelefteq, \wedge, <_{lex},(P_\alpha)_{\alpha<\omega}\}
\]
consists of binary relation symbols $\trianglelefteq, <_{lex}$, a binary function symbol $\wedge$, and unary relation symbols $P_\alpha$.  We will consider a tree $\omega^{\leq k}$ (with $k<\omega$) as an $L_s$-structure, where $\trianglelefteq$ is interpreted as the containment relation, $<_{lex}$ as the lexicographic order, $\wedge$ as the meet function and  $P_\alpha$ as the $\alpha$-th level of the tree.
\end{definition}
\begin{definition}[{\cite[Definition 3.7]{KKS}}]
\thlabel{def:em-types}
Let $I$ be an arbitrary index structure and $C$ an arbitrary set of parameters. The \emph{EM-type} of a tuple $A = (a_i)_{i \in I}$ over $C$ is the partial type in variables $(x_i)_{ i \in I}$, consisting of all the formulas of the form $\phi(x_{\bar{i}})$ over $C$ (where $\bar{i}$ is a tuple in $I$) satisfying the following property: $\models \phi(a_{\bar{j}})$ holds whenever $\bar{j}$ is a tuple in $I$ with $\qftp_I(\bar{j}) = \qftp_I(\bar{i})$. We let $EM_I(A/C)$ denote this partial type.

In particular, we will write $EM_s(A/C)$  [respectively, $EM_<(A/C)$] for $EM_I(A/C)$ where $I$ is considered as an $L_s$-structure [respectively, a $\{<\}$-structure].
\end{definition}
\begin{definition}
\thlabel{def:em-based}
 Let $I$ be an index structure and let $A=(a_i)_{i\in I}$ and $B=(b_i)_{i\in I}$ be $I$-indexed tuples of compatible parameters. We will say that $A$ is \emph{$EM_I$-based on $B$ over $C$} if $EM_I(A/C)\supseteq EM_I(B/C)$. 
\end{definition}
\begin{corollary}
\thlabel{cor:em-based-indiscernible-sequence}
If $A$ is any set of parameters, then for any compatible sequence $(a_i)_{i<\omega}$ there is an $A$-indiscernible sequence $(b_i)_{i<\omega}$ which is $EM_<$-based on $(a_i)_{i<\omega}$ over $A$.
\end{corollary}
\begin{proof}
By compactness there is a sequence $(a'_i)_{i<\lambda_{|T|+|A|+|a_0|}}$ which is $EM_<$-based on $(a_i)_{i<\omega}$ over $A$. 
Then by \thref{lem:indiscernible-sequence-based-on-long-sequence} there is an $A$-indiscernible sequence $(b_i)_{i<\omega}$ which is   $EM_<$-based on $(a'_i)_{i<\lambda_{|T|+|A|+|a_0|}}$ over $A$, hence $EM_<$-based on $(a_i)_{i<\omega}$ over $A$.
\end{proof}
In what follows we consider $\omega^{\leq k}$ as an $L_s$-structure (see \thref{def:tree-languages}). We will only work with trees of width $\omega$, as we will only need those, but everything naturally works for arbitrary (infinite) widths.
\begin{definition}
\thlabel{def:s-indiscernible}
We call a tree $(a_\eta)_{\eta \in \omega^{\leq k}}$ \emph{s-indiscernible over $C$} if for any $\bar{\eta}, \bar{\nu} \subseteq \omega^{\leq k}$ such that $\bar{\eta} \equiv_{qf} \bar{\nu}$ we have that $a_{\bar{\eta}} \equiv_C a_{\bar{\nu}}$.
\end{definition}
\begin{lemma}\thlabel{lem:L_s-extending}
Suppose $\bar{\eta}=(\eta_0,\dots,\eta_{n-1})\equiv_{qf}\bar{\nu}=(\nu_0,\dots,\nu_{n-1})$ are tuples of elements of $\omega^{\leq k}$ for some $k<\omega$.
Then there exists a sequence $I$ of $n$-tuples of elements of $\omega^{\leq k}$  such that $\bar{\eta}\smallfrown I$ and $\bar{\nu}\smallfrown I$ are qf-indiscernible sequences in $\omega^{\leq k}$.
\end{lemma}
\begin{proof}
Let $l<\omega$ be such that $\bar{\eta},\bar{\nu}\subseteq \{\emptyset\}\cup\{\xi\in \omega^{\leq k}\backslash \{\emptyset\}:\xi(0)<l\}$.
For every $0<m<\omega$ choose a tuple $\bar{\chi}^m\subseteq \{\emptyset\}\cup\{\xi\in \omega^{\leq k}\backslash \{\emptyset\}:ml\leq \xi(0)<(m+1)l\}$ such that $\bar{\chi}^m\equiv_{qf}\bar{\eta}\equiv_{qf}\bar{\nu}$  (for example, for every $n'<n$ put $\chi^m_{n'}(0)=\eta_{n'}(0)+ml$ and $\chi^m_{n'}(i)=\eta_{n'}(i)$ for every $0<i\leq k$). Finally, put $I=(\bar{\chi}^m)_{0<m<\omega}$.
\end{proof}
\begin{corollary}\thlabel{cor:E_s-typedef}
If $T$ is thick then s-indiscernibility is type-definable, i.e.\ for every $k<\omega$ and a tuple of variables $y$ there is a partial type $\pi((x_\eta)_{\eta\in \omega^{\leq k}},y)$ over $\emptyset$ such that for all $D$ with $|D| = |y|$: $((a_\eta)_{\eta\in \omega^{\leq k}},D)\models \pi $ if and only if $(a_\eta)_{\eta\in \omega^{\leq k}}$ is s-indiscernible over $D$.

More specifically, we can take $\pi((x_\eta)_{\eta\in \omega^{\leq k}},y)$ to be the partial type that expresses that for any $(\eta_0,\dots,\eta_{n-1})\equiv_{qf} (\nu_0,\dots,\nu_{n-1})$ the Lascar distance of $(x_{\eta_0},\dots,x_{\eta_{n-1}})$ and $(x_{\nu_0},\dots,x_{\nu_{n-1}})$ over $y$ is at most 2. 
\end{corollary}
\begin{proof}
Let $\pi$ be as above and consider arbitrary $(a_\eta)_{\eta\in \omega^{\leq k}}$ and $D$. If $((a_\eta)_{\eta\in \omega^{\leq k}},D)\models \pi$ then $(a_\eta)_{\eta\in \omega^{\leq k}}$ is indiscernible over $D$ as being at Lascar distance at most 2 over $D$ implies equality of types over $D$.

Conversely, if $((a_\eta)_{\eta\in \omega^{\leq k}},D)$ is s-indiscernible over $D$ and $\bar{\eta}=(\eta_0,\dots,\eta_{n-1})\equiv_{qf} \bar{\nu}=(\nu_0,\dots,\nu_{n-1})$, then with $I=(\bar{\chi}^m)_{0<m<\omega}$ given by  \thref{lem:L_s-extending} we have that $a_{\bar{\eta}}\smallfrown(a_{\bar{\chi}^m})_{0<m<\omega}$ and  $a_{\bar{\nu}}\smallfrown(a_{\bar{\chi}^m})_{0<m<\omega}$ are both indiscernible sequences over $D$, so $a_{\bar{\eta}}$ and $a_{\bar{\nu}}$ are at Lascar distance at most 2 over $D$.
\end{proof}
We now adapt the proof of \cite[Theorem 4.3]{KKS} to obtain the $EM_s$-modeling property for positive logic.
\begin{proposition}
\thlabel{prop:em-modeling}
Suppose $T$ is thick and consider arbitrary set of parameters $D$ and $k<\omega$. Then for any  tree $A=(a_\eta)_{\eta\in \omega^{\leq k}}$ of compatible tuples there is an s-indiscernible over $D$ tree  $C=(c_\eta)_{\eta\in \omega^{\leq k}}$ which is $EM_s$-based on $A$ over $D$.
\end{proposition}
\begin{proof}
 We proceed by induction on $k$. The case $k=0$ is trivial. Suppose the assertion holds for some $k$ and consider any $A=(a_\eta)_{\eta\in \omega^{\leq k+1}}$. For any $i<\omega$ consider an $\omega^{\leq k} $-indexed tree $A_i:=(a_{i\smallfrown \eta})_{\eta \in \omega^{\leq k}}$. Using the inductive hypothesis we choose inductively for each $i<\omega$ a tree $B_i=(b^i_\eta)_{\eta\in \omega^{\leq k}}$ which is s-indiscernible over $Da_\emptyset B_{<i} A_{>i}$ and $EM_s$-based on $A_i$ over $Da_\emptyset B_{<i} A_{>i}$.
 Let $B=(b_\eta)_{\eta\in \omega^{\leq k+1}}$ where $b_\emptyset=a_\emptyset$ and $b_{i\smallfrown \xi}=b^i_\xi$ for every $i<\omega$ and $\xi\in \omega^{\leq k}$.
 \begin{claim}
 \thlabel{claim:mut-ind}
   $B_i$ is s-indiscernible over $Db_{\emptyset}B_{\neq i}$ for every $i<\omega$.
 \end{claim}
Fix $i<\omega$. We will show by induction on $j$ that $B_i$ is s-indiscernible over $Db_{\emptyset}B_{<i}B_{i+1}\dots B_{j-1}A_{\geq j}$ for every $j>i$, which is enough by \thref{cor:E_s-typedef}. For $j=i+1$ this follows directly from the choice of $B_i$. Now suppose the assertion holds for some $j>i$.
By \thref{cor:E_s-typedef} there is a type $\pi((x_\eta)_{\eta\in \omega^{\leq k}},\bar{y})$ over $D':=Db_{\emptyset}B_{<i}B_{i+1}\dots B_{j-1}A_{>j}$, where $\bar{y}=(y_\eta)_{\eta\in \omega ^{\leq k}}$, expressing that $(x_\eta)_{\eta\in \omega^{\leq k}}$ is s-indiscernible over $D'\bar{y}$. Then $B_iA_j\models \pi$. Note that the type $\pi(B_i,\bar{y})$ is invariant under all permutations of  $\bar{y}$, hence if $\phi(y_{\eta_0},\dots,y_{\eta_{n-1}})\in \pi(B_i,\bar{y})$ then  $\phi(y_{\nu_0},\dots,y_{\nu_{n-1}})\in \tp(A_j/D'B_i)$ for all $\nu_0,\dots,\nu_{n-1}\in \omega^{\leq k}$. In particular,  $\pi(B_i,\bar{y})\subseteq EM_s(A_j/D'B_i)$. Thus, by the choice of $B_j$, we have that $\pi(B_i,\bar{y})\subseteq EM_s(B_j/D'B_i)$, so in particular $B_iB_j\models \pi$. Hence $B_i$ is indiscernible over $D'B_j=Db_{\emptyset}B_{<i}B_{i+1}\dots B_{j}A_{\geq j+1}$, as required.
\begin{claim}
$B$ is $EM_s$-based on $A$ over $D$.
\end{claim}
Consider any $i<\omega$ and the trees $E=(e_\eta)_{\eta\in \omega^{\leq k+1}}$ and $F=(f_\eta)_{\eta\in \omega^{\leq k+1}}$ given by $e_\emptyset=f_\emptyset=a_\emptyset$, $e_{j\smallfrown \eta}=\begin{cases}
b_{j\smallfrown \eta}\mbox{ for } j<i \\                                                                                                                                                                                                      
 a_{j\smallfrown \eta} \mbox{ for } j\geq i                                                                                                                                                                                                   \end{cases}
$, and $f_{j\smallfrown \eta}=\begin{cases}
b_{j\smallfrown \eta}\mbox{ for } j\leq i \\                                                                                                                                                                                                      
 a_{j\smallfrown \eta} \mbox{ for } j> i                                                                                                                                                                                                   \end{cases}
$.
We will prove that $\pi_0:=EM_s(E/D) \subseteq EM_s(F/D) =:\pi_1$
which clearly is sufficient to prove the claim.
Let $\bar{x}=(x_\eta)_{\eta\in \omega^{\leq k+1}}$ be a tuple of variables compatible with $a_\eta$'s. We naturally view $\pi_0$ and $\pi_1$ as partial types in the variable $\bar{x}$. Consider any formula $\phi(x_{\eta_0},\dots,x_{\eta_l},x_{\eta_{l+1}},\dots,x_{\eta_{l'}})\in \pi_0$ over $D$ with $\eta_0,\dots,\eta_l\in K_i:=\{i\smallfrown \xi:\xi\in \omega^{\leq k}\}$ and $\eta_{l+1},\dots,\eta_{l'}\in \omega^{\leq k+1}\backslash K_i$. We will be done if we show that $\models \phi(f_{\eta_0},\dots, f_{\eta_{l'}})$.
Write $\eta_t=i\frown \xi_t$ for $t=0,1,\dots,l$. For any $\xi'_0,\dots,\xi'_l\in \omega^{\leq k}$ with $\qftp_{L_s}(\xi'_0,\dots,\xi'_l)=\qftp_{L_s}(\xi_0,\dots,\xi_l)$ we have $\qftp_{L_s}(\eta_0,\dots, \eta_{l'})=\qftp_{L_s}(i\smallfrown \xi_0,\dots,i\smallfrown \xi_l,\eta_{l+1},\dots,\eta_{l'})=\qftp_{L_s}(i\smallfrown \xi'_0,\dots,i\smallfrown \xi'_l,\eta_{l+1},\dots,\eta_{l'})$, hence, as $\phi\in \pi_0$, we get that $\models \phi(e_{i\smallfrown \xi'_0},\dots,e_{i\smallfrown \xi'_l},e_{\eta_{l+1}},\dots,e_{\eta_{l'}})$. This shows that $\phi(y_{\xi_0},\dots,y_{\xi_l},e_{\eta_{l+1}},\dots,e_{\eta_{l'}})\in EM_s(A_i/a_{\emptyset}A_{< i}B_{>i})$ where $A_i$ is naturally indexed by $ \omega^{\leq k}$, so, by the choice of $B_i$ we get that $\models\phi(b^i_{\xi_0},\dots,b^i_{\xi_{l}},e_{\eta_{l+1}},\dots,e_{\eta_{l'}})$. As $(b^i_{\xi_0},\dots,b^i_{\xi_{l}},e_{\eta_{l+1}},\dots,e_{\eta_{l'}})=(f_{i\smallfrown \xi_0},\dots,f_{i\smallfrown \xi_l},f_{\eta_{l+1}},\dots,f_{\eta_{l'}})=(f_{\eta_0},\dots,f_{\eta_{l'}})$, this means that $\models\phi(f_{\eta_0},\dots,f_{\eta_{l'}})$, as required.

By \thref{cor:em-based-indiscernible-sequence} we find a sequence $(C_i)_{i < \omega} = ((c^i_\eta)_{\eta \in \omega^{\leq k}})_{i < \omega}$ which is $EM_<$-based on $(B_i)_{i<\omega}$ over $Db_{\emptyset}$ and indiscernible over $Db_{\emptyset}$. Let $C=(c_\eta)_{\eta\in \omega^{\leq k+1}}$ be given by $c_\emptyset=b_\emptyset$ and $c_{i\smallfrown\xi}=c^i_\xi$ for any $\xi \in \omega^{\leq k}$ and $i < \omega$. By \thref{claim:mut-ind} and \thref{cor:E_s-typedef} we get that $C_i$ is s-indiscernible over $C_{\neq i}Dc_\emptyset$ for every $i<\omega$, which, together with $Dc_{\emptyset}$-indiscernibility of $(C_i)_{i<\omega}$ easily gives that $C$ is s-indiscernible over $D$ (as in \cite{KKS}). It is left to prove:
\begin{claim}
$C$ is $EM_s$-based on $B$ (and hence on $A$) over $D$. 
\end{claim}
 Consider any formula $\phi(x_{i_1\smallfrown \xi_1},\dots, x_{i_l\smallfrown \xi_l}, x_\emptyset)\in EM_s(B/D)$ with $i_1,\dots,i_l\in \omega$ and $\xi_1,\dots,\xi_l\in \omega^{\leq k}$. Then for every $j_1,\dots,j_l\in \omega$ with $\qftp_{\{<\}}(j_1,\dots,j_l)=\qftp_{\{<\}}(i_1,\dots,i_l)$ we have that $\qftp_{L_s}( j_1\smallfrown \xi_1,\dots,j_l\smallfrown \xi_l, \emptyset)=\qftp_{L_s}( i_1\smallfrown \xi_1,\dots,i_l\smallfrown \xi_l, \emptyset)$, so $\models \phi(b_{j_1\smallfrown \xi_1},\dots,b_{j_l\smallfrown \xi_l}, b_\emptyset)$. This means that \[
 \phi(x_{i_1\smallfrown \xi_1},\dots, x_{i_l\smallfrown \xi_l}, b_\emptyset)\in EM_<((B_i)_{i<\omega }/b_\emptyset D),
 \]
 hence by the choice of $C$ we have $\models \phi(c_{i_1\smallfrown \xi_1},\dots,c_{i_l\smallfrown \xi_l},c_\emptyset)$ so $\phi(x_{i_1\smallfrown \xi_1},\dots, x_{i_l\smallfrown \xi_l}, x_\emptyset)\in EM_s(C/D)$, as required.
\end{proof}

\begin{definition}
\thlabel{def:cr-morley-sequence}
Let $I$ be a linearly ordered set. For a global $M$-$Ls$-invariant type $q$, we will call a sequence $(a_i)_{i\in I}$ a \emph{parallel-Morley sequence in $q$ over $M$}, if there is  some $(b_i)_{i\in I}\models q^{\otimes I}|_{M}$ such that the pair $(a_i,b_i)$ starts an $Ma_{>i}b_{>i}$-indiscernible sequence for every $i\in I$. We will say that  $(a_i)_{i\in I}$ is a parallel-Morley sequence in $\tp(a/M)$ if it is a parallel-Morley sequence in some global $M$-Ls-invariant type $q\supseteq \tp(a/M)$.
\end{definition}

In the semi-Hausdorff case we can replace the condition ``$(a_i,b_i)$ starts an $Ma_{>i}b_{>i}$-indiscernible sequence'' by ``$a_i \equiv_{Ma_{>i} b_{>i}} b_i$''. The reason for which we need the stronger condition in thick theories is that equality of types is not necessarily type-definable there, so some of the compactness arguments below would not work with the weaker condition.

Note that a parallel-Morley sequence is \emph{not} required to be indiscernible. The reason for the name ``parallel-Morley sequence'' is because such a sequence is parallel to a Morley sequence, in the sense of the parallel sequences lemma (\thref{lem:sequences-give-sop1}). We make this precise in \thref{cor:cr-kims-lemma}, for which we first slightly reformulate the parallel sequences lemma.
\begin{lemma}
\thlabel{lem:sequences-give-sop1-not-indiscernible}
Let $T$ be thick and suppose $\phi(x,y)$ is a formula and $(c_{i,0},c_{i,1})_{i\in I}$ is an infinite sequence of pairs with $(c_{i,1})_{i\in I}$ indiscernible, such that:
\begin{enumerate}[label=(\roman*)]
\item for every $i\in I$, the pair $(c_{i,0},c_{i,1})$ starts a $c_{>i,0}c_{>i,1}$-indiscernible sequence;
\item $\{\phi(x; c_{i,0}) : i \in I\}$ is consistent;
\item $\{\phi(x; c_{i,1}) : i \in I\}$ is inconsistent.
\end{enumerate}
Then $T$ has SOP$_1$.
\end{lemma}
\begin{proof}
We may assume the tuples $c_{i,0}$ and $c_{i,1}$ to be finite. As $(c_{i,1})_{i\in I}$ is indiscernible and $\{\phi(x,c_{i,1}):i\in I\}$ is inconsistent, there is some $\psi(y_1, \ldots, y_k)$ that implies $\neg \exists x(\phi(x, y_1) \wedge \ldots \wedge \phi(x, y_k))$ such that  for any $i_1 < \ldots < i_k \in I$ we have $\models \psi(c_{i_1, 1}, \ldots, c_{i_k, 1})$. Call this $\psi$-inconsistent. By compactness there is a sequence of pairs $(\bar{c'}_i)_{i<\lambda_T} = (c'_{i,0},c'_{i,1})_{i<\lambda_T}$ such that $(c'_{i,0},c'_{i,1})$ starts a $\bar{c'}_{>i}$-indiscernible sequence for every $i<\lambda_T$, $\{\phi(x,c'_{i,0}):i<\lambda_T\}$ is consistent and $\{\phi(x,c'_{i,1}):i<\lambda_T\}$ is $\psi$-inconsistent. Then an indiscernible sequence based on $(\bar{c'}_i)_{i<\lambda_T}$ will satisfy the assumptions of \thref{lem:sequences-give-sop1}, so $T$ has SOP$_1$.
\end{proof}
By Kim's Lemma (\thref{prop:kims-lemma}) and \thref{lem:sequences-give-sop1-not-indiscernible} we easily get the following.
\begin{corollary}
\thlabel{cor:cr-kims-lemma}
Suppose $T$ is thick NSOP$_1$ with an e.c.\ model $M$, $\Sigma(x,b)$ is a partial type, $I$ is an infinite linearly ordered set, and $(b_i)_{i\in I}$ a parallel-Morley sequence in $\tp(b/M)$. If $\bigcup \{\Sigma(x,b_i) : i\in I\}$ is consistent then $\Sigma(x, b)$ does not Kim-divide over $M$. If $(b_i)_{i\in I}$ is indiscernible over $M$, then the converse also holds.
\end{corollary}
\begin{definition}
\thlabel{def:spread-out}
Let $M$ be an e.c.\ model and $q$ a global $M$-Ls-invariant type.
\begin{enumerate}[label=(\roman*)]
\item We will say that a tree $(c_\eta)_{\eta\in \omega^{\leq k}}$ is \emph{$q$-spread-out over $M$} if for any $\eta_{1} \in \omega^{1},\eta_{2}\in \omega^{2},\dots,\eta_{k}\in \omega^{k}$ such that $\eta_1>_{lex} \eta_2>_{lex}\dots>_{lex}\eta_{k}$ and $(\forall  l<l'\leq k)(\eta_{l'}\wedge \eta_{l}\in \omega^{l-1})$ we have that $(c_{\eta_k},\dots ,c_{\eta_{1}})$ is a Morley sequence in $q$ over $M$.
\item We will say that $(c_\eta)_{\eta\in \omega^{\leq k}}$ is \emph{weakly $q$-spread-out over $M$} if $(c_{\eta_k},\dots ,c_{\eta_{1}})\models q^{\otimes k}|_{M}$ for $\eta_i$'s as in (i).
\end{enumerate}
\end{definition}
\begin{figure}[ht]
\centering
\input{graphics/spread-out}
\caption{An example of $\eta_i$'s from \thref{def:spread-out}}
\label{fig:spread-out}
\end{figure}
Clearly $q$-spread-outness implies weak $q$-spread-outness. We will freely use the above definition for trees of parameters indexed by trees naturally isomorphic to trees of the form $\omega'^{\leq k'}$, e.g.\ subtrees of $\omega^{\leq k}$ consisting of all nodes extending a fixed node.

The point of the conditions on the $\eta_i$'s in \thref{def:spread-out} is that this is quantifier-free definable by an $L_s$-formula. This is useful for preservation when $EM_s$-basing trees on one another, as we do in the following lemma.
\begin{lemma}
\thlabel{lem:q-spread-out-trees}
Let $k$ be a natural number, $M$ an e.c.\ model and $q$ a global $M$-Ls-invariant type.
\begin{enumerate}[label=(\roman*)]
\item If $((c_{i \smallfrown \eta})_{\eta \in \omega^{\leq k-1}})_{i<\omega}$ is a Morley sequence in a global $M$-Ls-invariant type $r(x,z)\supseteq q(x)$ over $M$, where $x$ corresponds to the elements $c_i$ and $(c_{0 \smallfrown \eta})_{\eta \in \omega^{\leq k-1}}$ is $q$-spread-out over $M$ then also $(c_\eta)_{\eta \in \omega^{\leq k}}$ is $q$-spread-out over $M$  for any choice of root $c_\emptyset$.
\item If $(c_\eta)_{\eta\in \omega^{\leq k}}$ is weakly $q$-spread-out over $M$ and $(c'_\eta)_{\eta\in \omega^{\leq k}}\models EM_s((c_\eta)_{\eta\in \omega^{\leq k}}/M)$, then also $(c'_\eta)_{\eta\in \omega^{\leq k}}$ is weakly $q$-spread-out over $M$.
\item If $(c_\eta)_{\eta\in \omega^{\leq k}}$ is weakly $q$-spread-out over $M$ and s-indiscernible over $M$, then for $a_i=c_{0^{k-i}}$ we have that $(a_i)_{i<k}$ is a parallel-Morley sequence in $q$ over $M$.
\end{enumerate}
\end{lemma}
\begin{figure}[ht]
\centering
\begin{subfigure}[b]{0.45\textwidth}
    \centering
    \input{graphics/spread-out-from-morley}
    \caption{Part (i)}
    \label{fig:spread-out-from-morley}
\end{subfigure}
\begin{subfigure}[b]{0.45\textwidth}
    \centering
    \input{graphics/cr-morley-in-q}
    \caption{Part (iii)}
    \label{fig:cr-morley-in-q}
\end{subfigure}
\caption{Illustrations of \thref{lem:q-spread-out-trees}}
\label{fig:q-spread-out-trees}
\end{figure}
\begin{proof}
(i) Let $\eta_k \in \omega^k, \ldots, \eta_1 \in \omega^1$ be such that $\eta_1 >_{lex} \ldots >_{lex} \eta_k$ and $(\forall  l<l'\leq k)(\eta_{l'}\wedge \eta_{l}\in \omega^{l-1})$. We will prove that $(c_{\eta_k}, \ldots, c_{\eta_1})$ is a Morley sequence in $q$. For each $\ell \geq 2$ let $\beta_\ell \in \omega^1$ be such that $\eta_{\ell} \trianglerighteq \beta_\ell$. For every $\ell > 2$ we have by assumption that $\eta_2 \wedge \eta_\ell = \eta_2|_1 = \beta_2$, hence $\beta_\ell = \beta_2 =: \beta$ (and $\eta_1 >_{\lex} \beta$ as $\eta_1 >_{\lex} \eta_2$). In particular $(c_{\eta_k}, \ldots, c_{\eta_2})$ is contained in $(c_{\beta \smallfrown \eta})_{\eta \in \omega^{\leq k-1}}$, which has the same Lascar strong type over $M$ as $(c_{0 \smallfrown \eta})_{\eta \in \omega^{\leq k-1}}$. So, as $(c_{0 \smallfrown \eta})_{\eta \in \omega^{\leq k-1}}$ is $q$-spread-out by assumption, $(c_{\eta_k}, \ldots, c_{\eta_2})$ is a Morley sequence in $q$. As $((c_{i \smallfrown \eta})_{\eta \in \omega^{\leq k-1}})_{i<\omega}$ is a Morley sequence in $r$, we have that $(c_{\eta_1 \smallfrown \eta})_{\eta \in \omega^{\leq k-1}}$, which contains $c_{\eta_1}$, has the same Lascar strong type over $M (c_{\beta \smallfrown \eta})_{\eta\in \omega^{\leq k-1}}$, which contains $Mc_{\eta_k} \ldots c_{\eta_2}$, as some realisation of $r$. Since $q(x) = r|_x$ we see that $c_{\eta_1}$ has the same Lascar strong type over $Mc_{\eta_k}, \ldots, c_{\eta_2}$ as some realisation of $q$. So we conclude that $(c_{\eta_k}, \ldots, c_{\eta_1})$ is indeed a Morley sequence in $q$.\\
(ii) This holds because the condition on $(\eta_1,\dots,\eta_k)$ in the definition of weak $q$-spread-outness is expressible by a quantifier-free $L_s$-formula.\\
(iii) Put $a'_i:=c_{0^{k-i-1}\smallfrown 1}$ for $i<k$. Then $(a'_i)_{i<k}\models q^{\otimes k}|_{M}$ by weak $q$-spread-outness, and $(a_i,a'_i)$ starts an $M_{a_{>i}a'_{>i}}$-indiscernible sequence for each $i<k$ by $s$-indiscernibility.
\end{proof}

%% file: graphics/spread-out.tex
\begin{tikzpicture}[x=0.75pt,y=0.75pt,yscale=-0.6,xscale=0.8]

\draw    (270,190) -- (370,240) ;
\draw    (170,40) -- (370,240) ;
\draw    (220,90) -- (220,40) ;
\draw    (220,90) -- (270,40) ;
\draw    (270,140) -- (270,90) ;
\draw    (270,140) -- (320,90) ;
\draw    (320,190) -- (320,140) ;
\draw    (320,190) -- (370,140) ;
\draw    (320,190) -- (420,140) ;
\draw    (170,90) -- (270,140) ;
\draw    (370,240) -- (370,190) ;
\draw    (420,190) -- (370,240) ;
\draw  [draw opacity=0][fill={rgb, 255:red, 0; green, 0; blue, 0 }  ,fill opacity=1 ] (215,40) .. controls (215,37.24) and (217.24,35) .. (220,35) .. controls (222.76,35) and (225,37.24) .. (225,40) .. controls (225,42.76) and (222.76,45) .. (220,45) .. controls (217.24,45) and (215,42.76) .. (215,40) -- cycle ;
\draw  [draw opacity=0][fill={rgb, 255:red, 0; green, 0; blue, 0 }  ,fill opacity=1 ] (315,90) .. controls (315,87.24) and (317.24,85) .. (320,85) .. controls (322.76,85) and (325,87.24) .. (325,90) .. controls (325,92.76) and (322.76,95) .. (320,95) .. controls (317.24,95) and (315,92.76) .. (315,90) -- cycle ;
\draw  [draw opacity=0][fill={rgb, 255:red, 0; green, 0; blue, 0 }  ,fill opacity=1 ] (415,140) .. controls (415,137.24) and (417.24,135) .. (420,135) .. controls (422.76,135) and (425,137.24) .. (425,140) .. controls (425,142.76) and (422.76,145) .. (420,145) .. controls (417.24,145) and (415,142.76) .. (415,140) -- cycle ;
\draw  [draw opacity=0][fill={rgb, 255:red, 0; green, 0; blue, 0 }  ,fill opacity=1 ] (415,190) .. controls (415,187.24) and (417.24,185) .. (420,185) .. controls (422.76,185) and (425,187.24) .. (425,190) .. controls (425,192.76) and (422.76,195) .. (420,195) .. controls (417.24,195) and (415,192.76) .. (415,190) -- cycle ;

\draw (231,37.31) node [anchor=south west] [inner sep=0.75pt]    {$\eta _{4}$};
\draw (331,87.31) node [anchor=south west] [inner sep=0.75pt]    {$\eta _{3}$};
\draw (431,137.31) node [anchor=south west] [inner sep=0.75pt]    {$\eta _{2}$};
\draw (431,187.31) node [anchor=south west] [inner sep=0.75pt]    {$\eta _{1}$};

\end{tikzpicture}

%% file: graphics/spread-out-from-morley.tex
\begin{tikzpicture}[x=0.75pt,y=0.75pt,yscale=-0.8,xscale=0.8]
\path (180, 240);

\draw    (230,120) -- (290,180) ;
\draw    (250,140) -- (250,120) ;
\draw    (250,140) -- (270,120) ;
\draw    (310,120) -- (330,140) ;
\draw    (350,120) -- (290,180) ;
\draw    (330,120) -- (330,140) ;
\draw  [draw opacity=0][fill={rgb, 255:red, 0; green, 0; blue, 0 }  ,fill opacity=1 ] (225,120) .. controls (225,117.24) and (227.24,115) .. (230,115) .. controls (232.76,115) and (235,117.24) .. (235,120) .. controls (235,122.76) and (232.76,125) .. (230,125) .. controls (227.24,125) and (225,122.76) .. (225,120) -- cycle ;
\draw  [draw opacity=0][fill={rgb, 255:red, 0; green, 0; blue, 0 }  ,fill opacity=1 ] (245,140) .. controls (245,137.24) and (247.24,135) .. (250,135) .. controls (252.76,135) and (255,137.24) .. (255,140) .. controls (255,142.76) and (252.76,145) .. (250,145) .. controls (247.24,145) and (245,142.76) .. (245,140) -- cycle ;
\draw  [draw opacity=0][fill={rgb, 255:red, 0; green, 0; blue, 0 }  ,fill opacity=1 ] (245,120) .. controls (245,117.24) and (247.24,115) .. (250,115) .. controls (252.76,115) and (255,117.24) .. (255,120) .. controls (255,122.76) and (252.76,125) .. (250,125) .. controls (247.24,125) and (245,122.76) .. (245,120) -- cycle ;
\draw  [draw opacity=0][fill={rgb, 255:red, 0; green, 0; blue, 0 }  ,fill opacity=1 ] (265,120) .. controls (265,117.24) and (267.24,115) .. (270,115) .. controls (272.76,115) and (275,117.24) .. (275,120) .. controls (275,122.76) and (272.76,125) .. (270,125) .. controls (267.24,125) and (265,122.76) .. (265,120) -- cycle ;
\draw  [draw opacity=0][fill={rgb, 255:red, 0; green, 0; blue, 0 }  ,fill opacity=1 ] (285,180) .. controls (285,177.24) and (287.24,175) .. (290,175) .. controls (292.76,175) and (295,177.24) .. (295,180) .. controls (295,182.76) and (292.76,185) .. (290,185) .. controls (287.24,185) and (285,182.76) .. (285,180) -- cycle ;
\draw  [draw opacity=0][fill={rgb, 255:red, 0; green, 0; blue, 0 }  ,fill opacity=1 ] (325,140) .. controls (325,137.24) and (327.24,135) .. (330,135) .. controls (332.76,135) and (335,137.24) .. (335,140) .. controls (335,142.76) and (332.76,145) .. (330,145) .. controls (327.24,145) and (325,142.76) .. (325,140) -- cycle ;
\draw  [draw opacity=0][fill={rgb, 255:red, 0; green, 0; blue, 0 }  ,fill opacity=1 ] (305,120) .. controls (305,117.24) and (307.24,115) .. (310,115) .. controls (312.76,115) and (315,117.24) .. (315,120) .. controls (315,122.76) and (312.76,125) .. (310,125) .. controls (307.24,125) and (305,122.76) .. (305,120) -- cycle ;
\draw  [draw opacity=0][fill={rgb, 255:red, 0; green, 0; blue, 0 }  ,fill opacity=1 ] (325,120) .. controls (325,117.24) and (327.24,115) .. (330,115) .. controls (332.76,115) and (335,117.24) .. (335,120) .. controls (335,122.76) and (332.76,125) .. (330,125) .. controls (327.24,125) and (325,122.76) .. (325,120) -- cycle ;
\draw  [draw opacity=0][fill={rgb, 255:red, 0; green, 0; blue, 0 }  ,fill opacity=1 ] (345,120) .. controls (345,117.24) and (347.24,115) .. (350,115) .. controls (352.76,115) and (355,117.24) .. (355,120) .. controls (355,122.76) and (352.76,125) .. (350,125) .. controls (347.24,125) and (345,122.76) .. (345,120) -- cycle ;
\draw  [color={rgb, 255:red, 65; green, 117; blue, 5 }  ,draw opacity=1 ] (210,125) .. controls (210,111.19) and (227.91,100) .. (250,100) .. controls (272.09,100) and (290,111.19) .. (290,125) .. controls (290,138.81) and (272.09,150) .. (250,150) .. controls (227.91,150) and (210,138.81) .. (210,125) -- cycle ;
\draw  [color={rgb, 255:red, 65; green, 117; blue, 5 }  ,draw opacity=1 ] (290,125) .. controls (290,111.19) and (307.91,100) .. (330,100) .. controls (352.09,100) and (370,111.19) .. (370,125) .. controls (370,138.81) and (352.09,150) .. (330,150) .. controls (307.91,150) and (290,138.81) .. (290,125) -- cycle ;
\draw  [dash pattern={on 0.84pt off 2.51pt}]  (390,120) -- (420,120) ;
\draw  [color={rgb, 255:red, 74; green, 144; blue, 226 }  ,draw opacity=1 ] (300,110) -- (360,110) -- (360,130) -- (300,130) -- cycle ;
\draw  [color={rgb, 255:red, 74; green, 144; blue, 226 }  ,draw opacity=1 ] (320,130) -- (340,130) -- (340,150) -- (320,150) -- cycle ;
\draw  [color={rgb, 255:red, 65; green, 117; blue, 5 }  ,draw opacity=1 ] (190,80) -- (440,80) -- (440,170) -- (190,170) -- cycle ;
\draw [color={rgb, 255:red, 74; green, 144; blue, 226 }  ,draw opacity=1 ]   (360,120) -- (380,100) ;
\draw [color={rgb, 255:red, 74; green, 144; blue, 226 }  ,draw opacity=1 ]   (340,140) -- (380,150) ;

\draw (386.28,97.6) node [anchor=south west] [inner sep=0.75pt]  [color={rgb, 255:red, 74; green, 144; blue, 226 }  ,opacity=1 ]  {$z$};
\draw (386.28,150) node [anchor=west] [inner sep=0.75pt]  [color={rgb, 255:red, 74; green, 144; blue, 226 }  ,opacity=1 ]  {$x$};
\draw (318,61.5) node  [color={rgb, 255:red, 65; green, 117; blue, 5 }  ,opacity=1 ] [align=left] {Morley sequence in $\displaystyle r( x,\ z)$};

\end{tikzpicture}

%% file: graphics/cr-morley-in-q.tex
\begin{tikzpicture}[x=0.75pt,y=0.75pt,yscale=-0.6,xscale=0.6]

\draw    (251.86,75.14) -- (401.86,225.14) ;
\draw    (281.86,105.14) -- (311.86,75.14) ;
\draw    (311.86,135.14) -- (341.86,105.14) ;
\draw    (341.86,165.14) -- (371.86,135.14) ;
\draw    (371.86,195.14) -- (401.86,165.14) ;
\draw    (431.86,195.14) -- (401.86,225.14) ;
\draw  [draw opacity=0][fill={rgb, 255:red, 0; green, 0; blue, 0 }  ,fill opacity=1 ] (246.86,75.14) .. controls (246.86,72.38) and (249.1,70.14) .. (251.86,70.14) .. controls (254.62,70.14) and (256.86,72.38) .. (256.86,75.14) .. controls (256.86,77.9) and (254.62,80.14) .. (251.86,80.14) .. controls (249.1,80.14) and (246.86,77.9) .. (246.86,75.14) -- cycle ;
\draw  [draw opacity=0][fill={rgb, 255:red, 0; green, 0; blue, 0 }  ,fill opacity=1 ] (306.86,75.14) .. controls (306.86,72.38) and (309.1,70.14) .. (311.86,70.14) .. controls (314.62,70.14) and (316.86,72.38) .. (316.86,75.14) .. controls (316.86,77.9) and (314.62,80.14) .. (311.86,80.14) .. controls (309.1,80.14) and (306.86,77.9) .. (306.86,75.14) -- cycle ;
\draw  [draw opacity=0][fill={rgb, 255:red, 0; green, 0; blue, 0 }  ,fill opacity=1 ] (276.86,105.14) .. controls (276.86,102.38) and (279.1,100.14) .. (281.86,100.14) .. controls (284.62,100.14) and (286.86,102.38) .. (286.86,105.14) .. controls (286.86,107.9) and (284.62,110.14) .. (281.86,110.14) .. controls (279.1,110.14) and (276.86,107.9) .. (276.86,105.14) -- cycle ;
\draw  [draw opacity=0][fill={rgb, 255:red, 0; green, 0; blue, 0 }  ,fill opacity=1 ] (306.86,135.14) .. controls (306.86,132.38) and (309.1,130.14) .. (311.86,130.14) .. controls (314.62,130.14) and (316.86,132.38) .. (316.86,135.14) .. controls (316.86,137.9) and (314.62,140.14) .. (311.86,140.14) .. controls (309.1,140.14) and (306.86,137.9) .. (306.86,135.14) -- cycle ;
\draw  [draw opacity=0][fill={rgb, 255:red, 0; green, 0; blue, 0 }  ,fill opacity=1 ] (336.86,105.14) .. controls (336.86,102.38) and (339.1,100.14) .. (341.86,100.14) .. controls (344.62,100.14) and (346.86,102.38) .. (346.86,105.14) .. controls (346.86,107.9) and (344.62,110.14) .. (341.86,110.14) .. controls (339.1,110.14) and (336.86,107.9) .. (336.86,105.14) -- cycle ;
\draw  [draw opacity=0][fill={rgb, 255:red, 0; green, 0; blue, 0 }  ,fill opacity=1 ] (336.86,165.14) .. controls (336.86,162.38) and (339.1,160.14) .. (341.86,160.14) .. controls (344.62,160.14) and (346.86,162.38) .. (346.86,165.14) .. controls (346.86,167.9) and (344.62,170.14) .. (341.86,170.14) .. controls (339.1,170.14) and (336.86,167.9) .. (336.86,165.14) -- cycle ;
\draw  [draw opacity=0][fill={rgb, 255:red, 0; green, 0; blue, 0 }  ,fill opacity=1 ] (366.86,135.14) .. controls (366.86,132.38) and (369.1,130.14) .. (371.86,130.14) .. controls (374.62,130.14) and (376.86,132.38) .. (376.86,135.14) .. controls (376.86,137.9) and (374.62,140.14) .. (371.86,140.14) .. controls (369.1,140.14) and (366.86,137.9) .. (366.86,135.14) -- cycle ;
\draw  [draw opacity=0][fill={rgb, 255:red, 0; green, 0; blue, 0 }  ,fill opacity=1 ] (366.86,195.14) .. controls (366.86,192.38) and (369.1,190.14) .. (371.86,190.14) .. controls (374.62,190.14) and (376.86,192.38) .. (376.86,195.14) .. controls (376.86,197.9) and (374.62,200.14) .. (371.86,200.14) .. controls (369.1,200.14) and (366.86,197.9) .. (366.86,195.14) -- cycle ;
\draw  [draw opacity=0][fill={rgb, 255:red, 0; green, 0; blue, 0 }  ,fill opacity=1 ] (396.86,165.14) .. controls (396.86,162.38) and (399.1,160.14) .. (401.86,160.14) .. controls (404.62,160.14) and (406.86,162.38) .. (406.86,165.14) .. controls (406.86,167.9) and (404.62,170.14) .. (401.86,170.14) .. controls (399.1,170.14) and (396.86,167.9) .. (396.86,165.14) -- cycle ;
\draw  [draw opacity=0][fill={rgb, 255:red, 0; green, 0; blue, 0 }  ,fill opacity=1 ] (396.86,225.14) .. controls (396.86,222.38) and (399.1,220.14) .. (401.86,220.14) .. controls (404.62,220.14) and (406.86,222.38) .. (406.86,225.14) .. controls (406.86,227.9) and (404.62,230.14) .. (401.86,230.14) .. controls (399.1,230.14) and (396.86,227.9) .. (396.86,225.14) -- cycle ;
\draw  [draw opacity=0][fill={rgb, 255:red, 0; green, 0; blue, 0 }  ,fill opacity=1 ] (426.86,195.14) .. controls (426.86,192.38) and (429.1,190.14) .. (431.86,190.14) .. controls (434.62,190.14) and (436.86,192.38) .. (436.86,195.14) .. controls (436.86,197.9) and (434.62,200.14) .. (431.86,200.14) .. controls (429.1,200.14) and (426.86,197.9) .. (426.86,195.14) -- cycle ;
\draw  [color={rgb, 255:red, 74; green, 144; blue, 226 }  ,draw opacity=1 ] (291.86,55.14) .. controls (297.38,49.62) and (337.68,80.96) .. (381.86,125.14) .. controls (426.04,169.32) and (457.38,209.62) .. (451.86,215.14) .. controls (446.34,220.66) and (406.04,189.32) .. (361.86,145.14) .. controls (317.68,100.96) and (286.34,60.66) .. (291.86,55.14) -- cycle ;
\draw  [color={rgb, 255:red, 65; green, 117; blue, 5 }  ,draw opacity=1 ] (231.86,55.14) .. controls (237.38,49.62) and (277.68,80.96) .. (321.86,125.14) .. controls (366.04,169.32) and (397.38,209.62) .. (391.86,215.14) .. controls (386.34,220.66) and (346.04,189.32) .. (301.86,145.14) .. controls (257.68,100.96) and (226.34,60.66) .. (231.86,55.14) -- cycle ;
\draw [color={rgb, 255:red, 65; green, 117; blue, 5 }  ,draw opacity=1 ]   (231.86,55.14) -- (240,30) ;
\draw [color={rgb, 255:red, 74; green, 144; blue, 226 }  ,draw opacity=1 ]   (451.86,215.14) -- (460,240) ;

\draw (240.85,76) node [anchor=east] [inner sep=0.75pt]    {$a_{0}$};
\draw (269.86,104.14) node [anchor=east] [inner sep=0.75pt]    {$a_{1}$};
\draw (299.86,134.14) node [anchor=east] [inner sep=0.75pt]    {$a_{2}$};
\draw (329.86,164.14) node [anchor=east] [inner sep=0.75pt]    {$a_{3}$};
\draw (359.86,196.14) node [anchor=east] [inner sep=0.75pt]    {$a_{4}$};
\draw (322.87,76.14) node [anchor=west] [inner sep=0.75pt]    {$a'_{0}$};
\draw (351.88,104.29) node [anchor=west] [inner sep=0.75pt]    {$a'_{1}$};
\draw (381.88,134.29) node [anchor=west] [inner sep=0.75pt]    {$a'_{2}$};
\draw (411.88,164.29) node [anchor=west] [inner sep=0.75pt]    {$a'_{3}$};
\draw (441.88,196.29) node [anchor=west] [inner sep=0.75pt]    {$a'_{4}$};
\draw (423,249.4) node [anchor=north west][inner sep=0.75pt]  [color={rgb, 255:red, 74; green, 144; blue, 226 }  ,opacity=1 ]  {$\models q^{\otimes 5} |_{M}$};
\draw (289,28) node [anchor=south east] [inner sep=0.75pt]  [color={rgb, 255:red, 65; green, 117; blue, 5 }  ,opacity=1 ] [align=left] {parallel-Morley in $\displaystyle q$};

\end{tikzpicture}

%% file: tex/symmetry.tex
\section{Symmetry}
\label{sec:symmetry}
\begin{lemma}[Chain condition]
\thlabel{lem:chain-lemma}
Let $T$ be a thick NSOP$_1$ theory and let $M$ be an e.c.\ model. Let $(b_i)_{i < \kappa}$ be a Morley sequence in some global $M$-Ls-invariant $q(x)$. If $(b_i)_{i < \kappa}$ is $Ma$-indiscernible then $a \ind_M^K (b_i)_{i < \kappa}$.
\end{lemma}
\begin{proof}
We will prove that $a \ind_M^K b_{i_1} \ldots b_{i_k}$ for all $i_1 < \ldots < i_k < \kappa$. This is indeed enough by finite character. By $Ma$-indiscernibility of $(b_i)_{i < \kappa}$ we may assume $\{i_1, \ldots, i_k\} = \{0, \ldots, k-1\}$.

We have $(b_i)_{i < \omega} \equivls_M (\beta_i)_{i < \omega}$ for some $(\beta_i)_{i < \omega} \models q^{\otimes \omega}$.  Define the tuple $\gamma_i = (\beta_{ik}, \beta_{ik+1}, \ldots, \beta_{ik+k-1})$ for all $i < \omega$. Then $(\gamma_i)_{i < \omega} \models (q^{\otimes k})^{\otimes \omega}$ by associativity of tensoring (\thref{lem:tensor-facts}). We let $c_i = (b_{ik}, b_{ik+1}, \ldots, b_{ik+k-1})$ for all $i < \omega$. Then $(c_i)_{i < \omega} \equivls_M (\gamma_i)_{i < \omega}$. So $(c_i)_{i < \omega}$ is a Morley sequence in $q^{\otimes k}$ over $M$ and $(c_i)_{i < \omega}$ is $Ma$-indiscernible. So $\tp(a / M c_0) = \tp(a / M b_0 \ldots b_{k-1})$ does not $q^{\otimes k}$-divide, and thus $a \ind_M^K b_0 \ldots b_{k-1}$ as required.
\end{proof}
\begin{definition}
Suppose $M$ is an e.c.\ model, $q$  a global type extending $\Lstp(a/M)$ and $\lambda$ a cardinal. We will say that the extension $q \supseteq \Lstp(a/M)$ satisfies $(*)_\lambda$ if for every $c$ with $|c|\leq \lambda$ there is a global $M$-Ls-invariant type $r(x,y)\supseteq \Lstp(ac/M)$ extending $q(x)$ (in particular, $q$ is $M$-Ls-invariant).
\end{definition}
\begin{lemma}\thlabel{lem:(*)}
For any e.c.\ model $M$, tuple $a$ and cardinal $\lambda$ there is $q\supseteq \Lstp(a/M)$ satisfying $(*)_\lambda$.
\end{lemma}
\begin{proof}
Let $M$, $a$ and $\lambda$ be as in the statement. Choose a small tuple $d$  such that for any $c$ with $|c|\leq \lambda$ there is some $d'\subseteq d$ with $\Lstp(ad'/M)=\Lstp(ac/M)$ (this is possible as the number of Lascar types of tuples of fixed length over $M$ is bounded by \thref{lem:find-indisc-sequence-through-model}). Now take a global $M$-Ls-invariant extension $r(x,y)$ of $\Lstp(ad/M)$, where $x$ corresponds to $a$. Then $q:=r|_x$ is an extension of $\Lstp(a/M)$ satisfying $(*)_\lambda$. 
\end{proof}
\begin{remark}
If $q\supseteq \Lstp(a/M)$ is finitely satisfiable in $M$ then it satisfies $(*)_{\lambda}$ for any cardinal $\lambda$ (\cite[Lemma 3.4]{Mennuni}). However, finitely satisfiable extensions may not exist in thick theories.
\end{remark}
\begin{theorem}[Symmetry]
\thlabel{thm:symmetry}
In a thick NSOP$_1$ theory $a \ind_M^K b$ implies $b \ind_M^K a$.
\end{theorem}
\begin{proof}
We may assume that $b$ enumerates a $\lambda_T$-saturated model containing $M$. If this is not the case let $N \supseteq Mb$ be a $\lambda_T$-saturated model. By extension, \thref{cor:extension}, we find $N' \equiv_{Mb} N$ such that $a \ind^K_M N'$. Now we replace $b$ by $N'$ and we continue the proof.

Set $\lambda = |ab|$. By \thref{lem:(*)} we can choose a global extension $q\supseteq \Lstp(a/M)$ satisfying $(*)_\lambda$. Let $p(y, a) = \tp(b/Ma)$. We will show that there is a parallel-Morley sequence $(a_i)_{i < \omega}$ in $q$ over $M$ such that  $\bigcup_{i < \omega} p(y, a_i)$ is consistent, which is enough by \thref{cor:cr-kims-lemma}.  All the properties we wish $(a_i)_{i < \omega}$ to have are type-definable. It is thus enough to find such a sequence of length $k$ for every $k < \omega$.

So fix any $k < \omega$. By backward induction on $k' = k+1, k, \ldots, 1$ we will define trees $(c_\eta)_{\eta \in S_{k'}}$  where $S_{k'} = \{ \xi \in \omega^{\leq k+1} : 0^{k'-1} \trianglelefteq \xi\}$. We will write $S_{k'}^*$ for $S_{k'}$ without the root, so $S_{k'}^* = S_{k'} - \{0^{k'-1}\}$. For each $k'$ the tree $(c_\eta)_{\eta \in S_{k'}}$ will satisfy the following conditions:
\begin{enumerate}[label=(A\arabic*)$_{k'}$]
\item $c_\eta c_\nu \equivls_M ab$ for all $\nu \triangleright \eta \in S_{k'}$ with $\nu\in {\omega}^{k+1}$ and $\eta\in {\omega}^{\leq k}$;

\item $(c_\eta)_{\eta\in S_{k'}\cap \omega^{\leq k}}$ is $q$-spread-out over $M$;

\item (the root is independent from the rest) we have $c_{0^{k'-1}} \ind_M^K (c_\eta)_{\eta \in S_{k'}^*}$.
\end{enumerate}

For $k' = k + 1$ we let $t$ be a global $M$-Ls-invariant extension of $\Lstp(b/M)$. Since $a \ind_M^K b$ we have that $\tp(a/Mb)$ does not $t$-divide. By \thref{cor:q-ls-dividing-is-q-dividing} and our assumption on $b$ this means that $\Lstp(a/Mb)$ does not $t$-Ls-divide. So we find an $Ma$-indiscernible Morley sequence $(c_{0^k \smallfrown \alpha})_{\alpha < \omega}$ in $t$ with $c_{0^{k+1}} = b$. By \thref{lem:chain-lemma} we have that $a \ind_M^K (c_{0^k \smallfrown \alpha})_{\alpha < \omega}$. So we pick $c_{0^k} = a$ and directly satisfy (A3)$_{k'}$. Condition (A2)$_{k'}$ is vacuous and (A1)$_{k'}$ follows directly from $Ma$-indiscernibility of $(c_{0^k \smallfrown \alpha})_{\alpha < \omega}$ and the fact that $c_{0^{k+1}} = b$.

For the inductive step, suppose that we have constructed $(c_\eta)_{\eta \in S_{k'}}$. By (A1)$_{k'}$ there is a tuple $d$ such that $c_{0^{k'-1}}(c_\eta)_{\eta \in S_{k'}^*} \equivls_M a d$. So, by $(*)_\lambda$ there is a global $M$-Ls-invariant type $r(x, z)\supseteq q(x)$ extending $\Lstp(c_{0^{k'-1}}(c_\eta)_{\eta \in S_{k'}^*}/M)$. By (A3)$_{k'}$ we have that $c_{0^{k'-1}} \ind_M^K (c_\eta)_{\eta \in S_{k'}^*}$. So since $b \subseteq (c_\eta)_{\eta \in S_{k'}^*}$ and using our assumption on $b$ we have by \thref{cor:q-ls-dividing-is-q-dividing} that $\Lstp(c_{0^{k'-1}} / M (c_\eta)_{\eta \in S_{k'}^*})$ does not $r|_{z}$-Ls-divide. By extension for Ls-dividing, \thref{prop:ls-dividing-extension}, we find $c$ such that $c (c_\eta)_{\eta \in S_{k'}^*} \equivls_M c_{0^{k'-1}}(c_\eta)_{\eta \in S_{k'}^*}$ and $\Lstp(c / M (c_\eta)_{\eta \in S_{k'}})$ does not $r$-Ls-divide. So there is an $Mc$-indiscernible Morley sequence $((d_{\eta, i})_{\eta \in S_{k'}})_{i < \omega}$ in $r$ such that $(d_{\eta, 0})_{\eta \in S_{k'}} = (c_\eta)_{\eta \in S_{k'}}$. We set $c_{0^{k'-2}} = c$ and $c_{0^{k'-2} \smallfrown i \smallfrown \zeta} = d_{0^{k'-1} \smallfrown \zeta, i}$. Again, using \thref{lem:chain-lemma} we directly get (A3)$_{k'-1}$. 

Now (A2)$_{k'-1}$ follows from \thref{lem:q-spread-out-trees}(i). We verify (A1)$_{k'-1}$. Everything above the root consists of copies (via a Lascar strong automorphism over $M$) of $(c_\eta)_{\eta \in S_{k'}}$, so we only need to check that $c_{0^{k'-2}} c_\nu \equivls_M ab$ for all $\nu \in S_{k'-1} \cap \omega^{k+1}$. By indiscernibility we may assume $\nu \in S_{k'} \cap \omega^{k+1}$. Then (A1)$_{k'-1}$ follows from (A1)$_{k'}$ and the fact that $c_{0^{k'-2}} (c_\eta)_{\eta \in S_{k'}^*} \equivls_M c_{0^{k'-1}}(c_\eta)_{\eta \in S_{k'}^*}$.

Thus the inductive step, and hence the construction of the tree $(c_\eta)_{\eta \in \omega^{k+1}} = (c_\eta)_{\eta \in S_1}$, is completed.

Consider the following condition:
\begin{enumerate}[label=(A\arabic*')$_1$]

\item $c_\eta c_\nu \equiv_M ab$ for all $\nu \triangleright \eta$ with $\nu \in {\omega}^{k+1}$ and $\eta \in {\omega}^{\leq k}$;


\end{enumerate}
which is clearly implied by (A1)$_1$ as it is seen by the $EM_s$-type of $(c_\eta)_{\eta \in \omega^{\leq k+1}}$ over $M$. Letting $(c_\eta')_{\eta \in \omega^{\leq k+1}}$ be an $s$-indiscernible tree which is $EM_s$-based on $(c_\eta)_{\eta \in \omega^{\leq k+1}}$ over $M$, we get that $(c_\eta')_{\eta \in \omega^{\leq k+1}}$ satisfies (A1')$_1$, and $(c_\eta')_{\eta \in \omega^{\leq k}}$ is weakly $q$-spread-out over $M$ by \thref{lem:q-spread-out-trees}(ii).

Put $a_i = c_{0^{k+1-i}}'$. Then $(a_1, \ldots, a_k)$ is a parallel-Morley sequence in $q$ over $M$ by \thref{lem:q-spread-out-trees}(iii), and by (A1')$_1$ we have that $\bigcup_{1 \leq i \leq k} p(y, a_i)$ is consistent because it is realised by $c_{0^{k+1}}'$. This completes the proof.
\end{proof}
\begin{lemma}
\thlabel{lem:sop1-implies-failure-of-weak-independence-theorem}
Let $T$ be a thick theory. Suppose that $\phi(x, y)$ has SOP$_1$, witnessed by $\psi(y_1, y_2)$. Then there is an e.c.\ model $M$ and $b_1$, $b_2$, $c_1$, $c_2$ such that $c_1 \ind_M^u c_2$, $c_1 \ind_M^u b_1$, $c_2 \ind_M^u b_2$ and $b_1 c_1 \equivls_M b_2 c_2$ and $\models \phi(b_1, c_1) \wedge \phi(b_2, c_2) \wedge \psi(c_1, c_2)$.
\end{lemma}
\begin{proof}
The proof is mostly the same as \cite[Proposition A.7]{haykazyan_existentially_2021} but we have to adjust a few things throughout to get equality of Lascar strong types rather than just equality of types. As in that proof, we will use a Skolemisation technique for positive logic \cite[Lemma A.6]{haykazyan_existentially_2021}. In such a Skolemised theory the positively definable closure of any set is an e.c.\ model and the reduct of an e.c.\ model (to the original language) is an e.c.\ model (of the original theory). It is not directly clear whether this Skolemisation construction preserves thickness, but that is not a problem. Ultimately we are interested in Lascar strong types in our original theory. So even though we technically work in a Skolemised theory the (type-definable) predicate $\d(x, y) \leq 1$ should be taken as in our original theory.

Let $\kappa$ be any cardinal. By compactness we find parameters $(a_\eta)_{\eta \in 2^{< \kappa}}$ such that:
\begin{enumerate}[label=(\roman*)]
\item for every $\sigma \in 2^\kappa$ the set $\{\phi(x, a_{\sigma|i}) : i < \kappa\}$ is consistent,
\item for every $\eta, \nu \in 2^{< \kappa}$ such that $\eta^\frown 0 \preceq \nu$ we have $\models \psi(a_{\eta^\frown 1}, a_\nu)$.
\end{enumerate}
For a big enough cardinal $\lambda$, we construct by induction a sequence $(\eta_i, \nu_i)_{i < \lambda}$ with $\eta_i,\nu_i\in 2^{<\kappa}$ such that:
\begin{enumerate}[label=(\arabic*)]
\item $\eta_i\trianglelefteq \eta_j$ and $\eta_i\trianglelefteq \nu_j$ for all $i<j<\lambda$;
\item $\eta_{i}\trianglerighteq (\eta_i\wedge \nu_i)\frown 0$, $\nu_{i}= (\eta_i\wedge \nu_i)\frown 1$, and $(a_{\eta_i}, a_{\nu_i})$ starts an $a_{\eta_{<i}}a_{\nu_{<i}}$-indiscernible sequence   for every $i<\lambda$.
\end{enumerate}

Assume $(\eta_j, \nu_j)_{j < i}$ has been constructed and set $\eta = \bigcup_{j < i} \eta_j$. If we chose $\kappa$ to be large enough then, by applying \thref{lem:indiscernible-sequence-based-on-long-sequence} to $(a_{\eta^\frown 0^{\alpha \frown} 1})_{\alpha>0}$, it follows that there are $0 < \alpha < \beta < \kappa$ such that $(a_{\eta^\frown 0^{\alpha \frown} 1}, a_{\eta^\frown 0^{\beta \frown} 1})$ starts an ${\{\eta_j, \nu_j : j < i\}}$-indiscernible sequence. We set $\nu_i = \eta^\frown 0^{\alpha \frown} 1$ and $\eta_i = \eta^\frown 0^{\beta \frown} 1$.

By (i) and (1), there is $b_2$ realising $\{\phi(x, a_{\eta_i}) : i < \lambda\}$. Now let $(e_i, d_i)_{i < \omega+2}$ be indiscernible over $b_2$ based on $(a_{\eta_i}, a_{\nu_i})_{i < \lambda}$.

Let $M$ be the positively definable closure of $\{e_i, d_i : i < \omega\}$. As discussed, we may assume $M$ to be an e.c.\ model. Set $c_1 = d_\omega$ and $c_2 = e_{\omega+1}$. Then $c_1 \ind_{\{e_i, d_i : i < \omega\}}^u c_2$ and $c_2 \ind_{\{e_i, d_i : i < \omega\}}^u b_2$ by indiscernibility. So $c_1 \ind_M^u c_2$, $c_2 \ind_M^u b_2$ and $\models \phi(b_2, c_2)$. By construction $c_1 c_2 = d_\omega e_{\omega+1} \equiv a_{\nu_{i_0}} a_{\eta_{i_1}}$ for some $i_0 < i_1 < \lambda$ and thus $\models \psi(c_1, c_2)$ by (ii), (1), and (2).

To find $b_1$ we first claim that $\d_M(e_\omega, d_\omega) \leq 1$. By compactness it suffices to prove that $\d_A(e_\omega, d_\omega) \leq 1$ for all finite $A \subseteq M$. By how we constructed $M$ it then suffices to prove that $(e_\omega, d_\omega)$ starts an indiscernible sequence over $\{e_i, d_i : i < n\}$ for all $n < \omega$. To prove this last statement we let $i_0 < \ldots < i_{n+1} < \lambda$ be such that
\[
e_0 d_0 \ldots e_n d_n e_\omega d_\omega \equiv a_{\eta_{i_0}} a_{\nu_{i_0}} \ldots a_{\eta_{i_n}} a_{\nu_{i_n}} a_{\eta_{i_{n+1}}} a_{\nu_{i_{n+1}}}.
\]
By how we constructed $(\eta_i, \nu_i)_{i < \lambda}$ we have $(a_{\eta_{i_{n+1}}}, a_{\nu_{i_{n+1}}})$ starts an indiscernible sequence over $\{a_{\eta_{i_0}} a_{\nu_{i_0}} \ldots a_{\eta_{i_n}} a_{\nu_{i_n}}\}$. So the claim follows after applying the automorphism.

Now we leave the Skolemised theory and work in the original theory, so $\d(x, y) \leq 1$ corresponds to actually having Lascar distance one. We have $c_2=e_{\omega+1}\equivls_M e_\omega\equivls_M d_\omega=c_1$, so there is $f \in \Aut_f(\MM / M)$ such that $f(c_2) =  c_1$. Let $b_1 = f(b_2)$. Then  $c_2 b_2 \equivls_M c_1 b_1$, hence also $\models \phi(b_1, c_1)$ and $c_1 \ind_M^u b_1$, as required.
\end{proof}
\begin{theorem}
\thlabel{thm:symmetry-iff-nsop1}
Let $T$ be a thick theory. The following are equivalent:
\begin{enumerate}[label=(\roman*)]
\item $T$ is NSOP$_1$;
\item (symmetry) $a \ind_M^K b$ implies $b \ind_M^K a$;
\item (weak symmetry) $a \ind_M^{iLs} b$ implies $b \ind_M^K a$.
\end{enumerate}
\end{theorem}
\begin{proof}
\thref{thm:symmetry} is precisely (i) $\implies$ (ii). For (ii) $\implies$ (iii) we just note that $a \ind_M^{iLs} b$ implies $a \ind_M^K b$. Finally, for (iii) $\implies$ (i) we proceed is as in \cite[Proposition 3.22]{kaplan_kim-independence_2020} replacing their reference to \cite{chernikov_model-theoretic_2015} by \thref{lem:sop1-implies-failure-of-weak-independence-theorem} and being careful about using global Ls-invariant types instead of just global invariant types.

We prove the contrapositive, so assume $T$ has SOP$_1$. Then by \thref{lem:sop1-implies-failure-of-weak-independence-theorem} there is an e.c.\ model $M$ and $b_1$, $b_2$, $c_1$, $c_2$ such that $c_1 \ind_M^u c_2$, $c_1 \ind_M^u b_1$, $c_2 \ind_M^u b_2$ and $b_1 c_1 \equivls_M b_2 c_2$. Furthermore, for $p(x, c_1) = \tp(b_1 c_1 / M)$, we have that $p(x, c_1) \cup p(x, c_2)$ is inconsistent. In particular we have that $\Lstp(c_1 / M c_2)$ extends to a global $M$-Ls-invariant $q$. Then as $c_1 \equivls_M c_2$ there is a Morley sequence $(d_i)_{i < \omega}$ in $q$ with $d_0 d_1 = c_2 c_1$. We thus have that $\bigcup \{p(x, d_i) : i < \omega\}$ is inconsistent. So $b_2 \nind_M^K c_2$. Since also $c_2 \ind_M^u b_2$ and thus $c_2 \ind_M^{iLs} b_2$ we see that weak symmetry fails and this concludes our proof.
\end{proof}

%% file: tex/independence-theorem.tex
\section{Independence theorem}
\label{sec:independence-theorem}
We recall the following facts. The first is the same as \cite[Lemma 7.4]{kaplan_kim-independence_2020} and the second is the same as the claim in \cite[Lemma 5.3]{dobrowolski_independence_2019}. Their proofs work in our setting as well.
\begin{fact}
\thlabel{fact:kim-dividing-indep-theorem-facts}
The following hold in any thick NSOP$_1$ theory.
\begin{enumerate}[label=(\roman*)]
\item If $a \ind_M^d bc$ and $b \ind_M^K c$ then $ab \ind_M^K c$.
\item If $a \ind_M^K b$ and $a \ind_M^K c$ then there is $c'$ with $ac' \equiv_M ac$ such that $a \ind_M^K bc'$.
\end{enumerate}
\end{fact}
For the following lemma we borrow a trick from \cite[Lemma 5.4]{dobrowolski_independence_2019}.
\begin{lemma}
\thlabel{lem:indep-theorem-add-c}
Let $T$ be thick NSOP$_1$ and let $a \equivls_M a'$, $a \ind_M^K b$ and $a' \ind_M^K c$. Then there is $c'$ such that $ac' \equivls_M a'c$ and $a \ind_M^K bc'$.
\end{lemma}
\begin{proof}
Let $c^*$ be such that $a c^* \equivls_M a' c$, so $a \ind_M^K c^*$. Let $N' \supseteq M$ be $\lambda_T$-saturated and let $q$ be a global $M$-Ls-invariant extension of $\Lstp(N'/M)$. Let $N$ realise $q|_{Mabc^*}$, so we have $N \ind_M^{iLs} abc^*$. By \thref{fact:kim-dividing-indep-theorem-facts}(i) we then have $Na \ind_M^K b$ and $Na \ind_M^K  c^*$. So by fact \thref{fact:kim-dividing-indep-theorem-facts}(ii) we find $c'$ with $Nac' \equiv_M Nac^*$ and $Na \ind_M^K bc'$. We thus have $ac' \equivls_M ac^* \equivls_M a'c$, as required.
\end{proof}
\begin{definition}
\thlabel{def:star-independence}
We write $b \ind_M^* c$ to mean that $\Lstp(b/Mc)$ extends to a global $M$-Ls-invariant type $\tp(N/\MM)$ for some $\beth_\omega(\lambda_T + |Mbc|)$-saturated model $N \supseteq M$. Extending $\Lstp(b/Mc)$ here means that there is some $\beta \in N$ with $\beta \equivls_{Mc} b$.
\end{definition}
The point of the enormous cardinal $\beth_\omega(\lambda_T + |Mbc|)$ is that we will want to find a $\lambda_T$-saturated model $M'$ containing $M$ and a copy of $b$ in $N$, and then again some $\lambda_T$-saturated $M'' \supseteq M'$ inside $N$. By \thref{fact:saturated-models} we can choose these $\lambda_T$-saturated models small enough so that this process can be repeated any finite number of times.

We easily see that $\ind^*$ is invariant under automorphisms and, assuming thickness, that $b \ind_M^* M$ for all $M$.
\begin{lemma}
\thlabel{lem:star-independence-extension}
We have that $\ind^*$ satisfies the following extension properties.
\begin{enumerate}[label=(\roman*)]
\item (left extension) If $b \ind_M^* c$ and $|d| < \beth_\omega(\lambda_T + |Mbc|)$, then there is $d' \equivls_{Mb} d$ such that $bd' \ind_M^* c$.
\item (right extension) If $b \ind_M^* c$ and $|d| < \beth_\omega(\lambda_T + |Mbc|)$, then there is $d' \equivls_{Mc} d$ such that $b \ind_M^* cd'$.
\end{enumerate}
\end{lemma}
\begin{proof}
In both cases we assume $b \ind_M^* c$. So let $q = \tp(N/\MM)$ be a global $M$-Ls-invariant extension of $\Lstp(b/Mc)$ for some $\beth_\omega(\lambda_T + |Mbc|)$-saturated $N \supseteq M$.

We first prove left extension. Let $N' \equivls_{Mc} N$ be in $\MM$. By moving things by a Lascar strong automorphism over $Mc$ we may assume $b \in N'$. By \thref{fact:saturated-models} there is $Mb \subseteq M' \subseteq N'$ where $M'$ is $\lambda_T$-saturated and of cardinality $\leq 2^{\lambda_T+|Mb|}$. Let $d'$ realise $\tp(d/M')$ in $N'$. So $d' \equivls_{Mb} d$ while $q$ also extends $\Lstp(bd'/Mc)$, so indeed $bd' \ind_M^* c$.

Now we prove right extension. Let $\beta \in N$ be such that $\beta \equivls_{Mc} b$. Pick $b' \in \MM$ such that $b' \equivls_{Mcd} \beta$. Then clearly $b' \ind_M^* cd$. We finish the proof by picking $d'$ such that $bd' \equivls_{Mc} b'd$.
\end{proof}
\begin{proposition}[Weak independence theorem]
\thlabel{prop:weak-independence-theorem}
Let $T$ be thick NSOP$_1$. Suppose that $a \equivls_M a'$, $a \ind_M^K b$, $a' \ind_M^K c$ and $b \ind_M^* c$. Then there is $a''$ with $a'' \equivls_{Mb} a$ and $a'' \equivls_{Mc} a'$ such that $a'' \ind_M^K bc$.
\end{proposition}
\begin{proof}
We may assume that $b$ and $c$ both enumerate a $\lambda_T$-saturated model containing $M$. If this is not the case let $N \supseteq Mb$ be $\lambda_T$-saturated and such that $|N| < \beth_\omega(\lambda_T + |Mbc|)$. By left extension from \thref{lem:star-independence-extension} we then find $N' \equivls_{Mb} N$ with $N' \ind_M^* c$. By \thref{cor:extension} we find $a_0$ with $a_0 \equivls_{Mb} a$ and $a_0 \ind_M^K N'$. Now we can replace $a$ by $a_0$ and $b$ by $N'$ and continue the proof. The case for $c$ is analogous.

By \thref{lem:indep-theorem-add-c} there is $c'$ such that $ac' \equivls_M a'c$ and $a \ind_M^K bc'$. Apply left extension from \thref{lem:star-independence-extension} to $b \ind_M^* c$ and $c'$ to find $c'' \equivls_{Mb} c$ with $bc' \ind_M^* c''$. Let $b^*$ be such that $b^* c'' \equivls_M bc'$ and apply right extension from \thref{lem:star-independence-extension} to $bc' \ind_M^* c''$ and $b^*$ to find $b'' \equivls_{Mc''} b^*$ with $bc' \ind_M^* b'' c''$. In particular we have $b'' c'' \equivls_M bc'$ and $\Lstp(bc'/Mb''c'')$ extends to a global $M$-Ls-invariant type $q$. So there is a Morley sequence $(b_i c_i)_{i < \omega}$ in $q$ with $(b_0, c_0) = (b'', c'')$ and $(b_1, c_1) = (b, c')$. As $a \ind_M^K bc'$, we can find $a^*$ with $a^*b''c'' \equiv_M abc'$ such that $(b_i c_i)_{i < \omega}$ is $Ma^*$-indiscernible. By construction we had $c'' \equivls_{Mb} c$, so there is a Lascar strong automorphism $\sigma$ over $Mb$ such that $\sigma(c'') = c$. Set $a'' = \sigma(a^*)$, we check that this is indeed the $a''$ we are looking for.

By the chain condition (\thref{lem:chain-lemma}) we have $a^* \ind_M^K (b_i c_i)_{i < \omega}$, so we have $a^* \ind^K_M b c''$ and $a'' \ind_M^K bc$ then follows by invariance. By $Ma^*$-indiscernibility we have $a'' b \equiv_M a^* b \equiv_M a^* b'' \equiv_M a b$. We assumed $b$ to enumerate a $\lambda_T$-saturated model, so indeed $a'' \equivls_{Mb} a$. By construction of $c'$ we have $a'' c \equiv_M a^* c'' \equiv_M a c' \equiv_M a' c$. We assumed $c$ to enumerate a $\lambda_T$-saturated model, so indeed $a'' \equivls_{Mc} a'$, which concludes the proof.
\end{proof}
\begin{fact}
\thlabel{fact:global-ls-invariant-type-determined-by-restriction}
In a thick theory, if $N \supseteq M$ is $(2^{|M| + \lambda_T})^+$-saturated and $q$ and $r$ are global $M$-Ls-invariant types with $q|_N = r|_N$ then $q = r$.
\end{fact}
\begin{proof}
By \thref{fact:saturated-models} there is $M \subseteq M' \subseteq N$ where $M'$ is a $\lambda_T$-saturated model and $|M'| < (2^{|M| + \lambda_T})^+$. Let $\phi(x, b)$ be any formula with parameters $b$. Let $b' \in N$ realise $\tp(b/M')$. Then $b \equivls_M b'$. By $M$-Ls-invariance and $q|_N = r|_N$ we have
\[
\phi(x, b) \in q \quad \Leftrightarrow \quad
\phi(x, b') \in q \quad \Leftrightarrow \quad
\phi(x, b') \in r \quad \Leftrightarrow \quad
\phi(x, b) \in r,
\]
which concludes the proof.
\end{proof}
\begin{theorem}[Independence theorem]
\thlabel{thm:independence-theorem}
Let $T$ be a thick NSOP$_1$ theory. Suppose that $a \equivls_M a'$, $a \ind_M^K b$, $a' \ind_M^K c$ and $b \ind_M^K c$. Then there is $a''$ with $a'' \equivls_{Mb} a$, $a'' \equivls_{Mc} a'$ and $a'' \ind_M^K bc$.
\end{theorem}
\begin{proof}
We may assume that $b$ and $c$ both enumerate a $\lambda_T$-saturated model containing $M$. If this is not the case let $N \supseteq Mb$ be $\lambda_T$-saturated. By extension (\thref{cor:extension}) and symmetry then find $N' \equivls_{Mb} N$ with $N' \ind_M^K c$. Applying extension again we find $a_0$ with $a_0 \equivls_{Mb} a$ and $a_0 \ind_M^K N'$. Now we can replace $a$ by $a_0$ and $b$ by $N'$ and continue the proof. The case for $c$ is analogous.

Let $N_0 \supseteq M$ be $(2^{|M| + \lambda_T})^+$-saturated and let $\kappa$ be a big enough cardinal (depending only on $|N_0bc|$). Pick some global $M$-Ls-invariant type $q(y, z)$ extending $\Lstp(bc/M)$ such that $q$ also extends to a global $M$-Ls-invariant type $\tp(N/\MM)$ for some saturated enough $N \supseteq M$ (depending only on $\kappa$). So there is $\beta$ realising $q|_y$ with $\beta \equivls_M b$. Let $(b_i c_i)_{i < \kappa}$ be a Morley sequence in $q$ with $b_0 = b$ and let $b_\kappa \equivls_{M (b_i c_i)_{i < \kappa}} \beta$. Then we have $b_i c_i \ind_M^* b_{<i} c_{<i}$ for all $i < \kappa$ and $b_\kappa \ind_M^* (b_i c_i)_{i < \kappa}$.

We will inductively construct a sequence $(b_i')_{i \leq \kappa}$ with $b_0' = b$ such that at step $i$:
\begin{enumerate}[label=(\roman*)]
\item $c \ind_M^K b_{\leq i}'$,
\item $cb_i' \equivls_M cb$,
\item $b_{\leq i}' \equivls_M b_{\leq i}$.
\end{enumerate}
The base case is already fixed: $b_0' = b$. So suppose we have constructed $b_{\leq i}'$. By induction hypothesis (iii) we can find $b^* b_{\leq i}' \equivls_M b_{i+1} b_{\leq i}$. So $b^* \ind_M^* b_{\leq i}'$. Let $c^*$ be such that $c^* b^* \equivls_M cb$, so $c^* \ind_M^K b^*$. So also using (i) from the induction hypothesis we can apply the weak independence theorem (\thref{prop:weak-independence-theorem}) to find $c'$ such that $c' \ind_M^K b_{\leq i}' b^*$, $c' \equivls_{Mb^*} c^*$ and $c' \equivls_{M b_{\leq i}'} c$. We now pick $b_{i+1}'$ to be such that $c b_{i+1}' \equivls_{M b_{\leq i}'} c' b^*$. Then indeed $c \ind_M^K b_{\leq i+1}'$. We also have $b_{\leq i}' b_{i+1}' \equivls_M b_{\leq i}' b^* \equivls_M b_{\leq i} b_{i+1}$. Finally, $c b_{i+1}' \equivls_M c' b^* \equivls_M c^* b^* \equivls_M cb$. So this concludes the successor step. For the limit stage we assume we have constructed $b_{<i}'$. We then have $c \ind_M^K b_{<i}'$ by finite character. We also have $b_{\leq j}' \equivls_M b_{\leq j}$ for all $j < i$. So we have $b_{< i}' \equiv_M b_{< i}$. We assumed $b$ to enumerate a $\lambda_T$-saturated model containing $M$, so because $b_0' = b = b_0$ we do in fact have $b_{< i}' \equivls_M b_{< i}$. We then construct $b_i'$ in an analogous way to the successor step.

We let $(c_i')_{i < \kappa}$ be such that $b_\kappa' (b_i' c_i')_{i < \kappa} \equivls_M b_\kappa (b_i c_i)_{i < \kappa}$. So by $M$-Ls-invariance of $q|_y$ we have $\beta b_\kappa' (b_i' c_i')_{i < \kappa} \equivls_M \beta b_\kappa (b_i c_i)_{i < \kappa}$ and thus by how we chose $b_\kappa$ we have $b_\kappa' \equivls_{M (b_i' c_i')_{i < \kappa}} \beta$.

Since $q \subseteq \tp(N/\MM)$ for some saturated enough $N$ we can find $\beta \gamma (\beta_i, \gamma_i)_{i < \kappa} \equivls_M b_\kappa' c (b_i', c_i')_{i < \kappa}$ in $N$, where $\beta \gamma \models q$. Here we used the fact that $b_\kappa' c \equivls_M bc$. Set $q'((y_i, z_i)_{i < \kappa}, y, z) = \tp((\beta_i, \gamma_i)_{i < \kappa} \beta \gamma/\MM)$. Then $q'$ is global $M$-Ls-invariant because $\tp(N/\MM)$ is global $M$-Ls-invariant. By \thref{fact:global-ls-invariant-type-determined-by-restriction} and our choice of $\kappa$ we get that some global $M$-Ls-invariant type $q'|_{y_i z_i y z}$ occurs for $\kappa$ many $i$ (modulo identifying the the variables for different $i$'s). We now focus on a subsequence of length $\omega$ such that (after relabelling) $q'|_{y_i z_i y z}$ does not depend on $i$, and we forget about $\kappa$. We also relabel $b_\kappa'$ to $b'$. \\
\\
\textbf{Claim 1.} In summary, we have just constructed the following.
\begin{enumerate}[label=(\roman*)]
\item A Morley sequence $(b_i' c_i')_{i < \omega}$ in $q$, where $q$ is a global $M$-Ls-invariant extension of $\Lstp(bc/M)$.
\item For every $i < \omega$ we have $b_i' c \equivls_M b'c \equivls_M bc$.
\item Let $\beta \models q|_y$ then $b' \equivls_{M(b_i' c_i')_{i < \omega}} \beta$.
\item $q(y, z) \subseteq q'((y_i, z_i)_{i < \omega}, y, z)$ and $q'$ is global $M$-Ls-invariant and extends $\Lstp((b_i', c_i')_{i < \omega} b' c / M)$.
\item There is some sufficiently saturated $N$ such that $q' \subseteq \tp(N/\MM)$ and $\tp(N/\MM)$ is $M$-Ls-invariant.
\item The type $q'|_{y_i z_i y z}$ does not depend on $i$, modulo identifying variables for different $i$'s.
\end{enumerate}
\textbf{Claim 2.} For every $k < \omega$ there are $g_0 h_0 g_1 h_1 \ldots g_{k-1} h_{k-1} g_k$, $g_0' h_0' g_1' h_1' \ldots g_{k-1}' h_{k-1}'$ and $h_0'' g_1'' h_1'' \ldots g_{k-1}'' h_{k-1}'' g_k''$ such that:
\begin{enumerate}[label=(\roman*)]
\item $(g_i' h_i')_{i < k} \models (q'|_{y_0, z})^{\otimes k}|_M$,
\item $(h_i'' g_{i+1}'')_{i < k} \models (q'|_{z_0, y})^{\otimes k}|_M$,
\item $(g_i h_i, g_i' h_i')$ starts an $M g_{>i} h_{>i} g_{>i}' h_{>i}'$-indiscernible sequence for every $i < k$,
\item $(h_i g_{i+1}, h_i'' g_{i+1}'')$ starts an $M h_{>i} g_{>i+1} h_{>i}'' g_{>i+1}''$-indiscernible sequence for every $i < k$.
\end{enumerate}
We first prove that the theorem follows from Claim 2. We set $p_0(x, y) = \tp(ab/M)$ and $p_1(x, z) = \tp(a'c/M)$. We will prove that $p_0(x, b) \cup p_1(x, c)$ does not Kim-divide over $M$. This is enough, because by \thref{prop:q-dividing-extending-to-complete-types} we can then extend it to a complete type that does not Kim-divide over $M$. Since we assumed $b$ and $c$ to enumerate $\lambda_T$-saturated models containing $M$, any realisation $a''$ of that complete type is then what we needed to construct.

By compactness we can find $M$-indiscernible $(g_i h_i g_i' h_i' g_i'' h_i'')_{i \in \Z}$ such that $(g_i' h_i')_{i \in \Z} \models (q'|_{y_0, z})^{\otimes \Z}|_M$ and $(h_i'' g_{i+1}'')_{i \in \Z} \models (q'|_{z_0, y})^{\otimes \Z}|_M$. Furthermore, we can make it so that for every $i \in \Z$ we have $g_i h_i \equiv_{M g_{>i} h_{>i} g_{>i}' h_{>i}'} g_i' h_i'$ and $h_i g_{i+1} \equiv_{M h_{>i} g_{>i+1} h_{>i}'' g_{>i+1}''} h_i'' g_{i+1}''$. We have that $q'|_{y, z_0} \supseteq \tp(b' c_0'/M)$, by Claim 1(iv). So by Claim 1(iii) and (v) we have that $b' \ind_M^* c_0'$. Then by \thref{prop:weak-independence-theorem} we have that $p_0(x, g_1'') \cup p_1(x, h_0'')$ does not Kim-divide. Then because $(h_i'' g_{i+1}'')_{i \geq n} \models (q'|_{z_0, y})^{\otimes \omega}|_M$ for all $n \in \Z$, we get that $\bigcup_{i \in \Z} p_0(x, g_{i+1}'') \cup p_1(x, h_i'')$ is consistent. By the parallel sequences lemma (\thref{lem:sequences-give-sop1}) we thus have that $\bigcup_{i \in \Z} p_0(x, g_{i+1}) \cup p_1(x, h_i)$ is consistent. This is the same set as $\bigcup_{i \in \Z} p_0(x, g_i) \cup p_1(x, h_i)$. So again by the parallel sequences lemma we get that $\bigcup_{i \in \Z} p_0(x, g_i') \cup p_1(x, h_i')$ is consistent. By Claim 1(ii) and (iii) we have that $q'|_{y_0, z}$ extends $\Lstp(bc/M)$. So we conclude that $p_0(x, b) \cup p_1(x, c)$ does not Kim-divide over $M$, as required.

We are left to verify Claim 2. We fix $k$ and by backwards induction on $k' = 2k, 2k-1, \ldots, 1$ we will define trees $(d_\eta e_\eta)_{\eta \in S_{k'}}$ where $S_{k'} = \{ \xi \in \omega^{\leq 2k+1} : 0^{k'-1} \trianglelefteq \xi\}$ such that for each $k'$ the tree $(d_\eta e_\eta)_{\eta \in S_{k'}}$ satisfies the following condition:

\textbf{(P)$_{k'}$} For every $\eta \in \omega^{\leq 2k-1}$ and $i < \omega$ such that $\eta \smallfrown i \in S_{k'}$ we have that:
\[
(d_{\eta \smallfrown i \smallfrown j} e_{\eta \smallfrown i \smallfrown j})_{j < \omega} d_{\eta \smallfrown i} e_{\eta \smallfrown i} \equivls_{M (d_{\trianglerighteq \eta \smallfrown i'} e_{\trianglerighteq \eta \smallfrown i'})_{i' < i}} (\beta_j \gamma_j)_{j < \omega} \beta \gamma.
\]
Recall that $q' = \tp((\beta_j \gamma_j)_{j < \omega} \beta \gamma/\MM)$. So in particular $(d_{\eta \smallfrown j} e_{\eta \smallfrown j})_{j < \omega} d_\eta e_\eta \equivls_M (\beta_j \gamma_j)_{j < \omega} \beta \gamma$ for all $\eta \in \omega^{\leq 2k} \cap S_{k'}$.

For $k' = 2k$ we let $(d_\eta e_\eta)_{\eta \in S_{2k}}$ just be $(b_i' c_i')_{i < \omega} b'c$. Suppose now that we have constructed $(d_\eta e_\eta)_{\eta \in S_{k'}}$. By (P)$_{k'}$ we have that $(d_{0^{k'-1} \smallfrown i} e_{0^{k'-1} \smallfrown i})_{i < \omega} d_{0^{k'-1}} e_{0^{k'-1}} \equivls_M (\beta_i \gamma_i)_{i < \omega} \beta \gamma$. So by Claim 1(v) there is global $M$-Ls-invariant $r \supseteq q'$ such that $r$ also extends $\Lstp((d_\eta e_\eta)_{\eta \in S_{k'}}/M)$. Here we match $(d_{0^{k'-1} \smallfrown i} e_{0^{k'-1} \smallfrown i})_{i < \omega} d_{0^{k'-1}} e_{0^{k'-1}}$ with the variables in $q'$. Let $((d_{\eta,i} e_{\eta,i})_{\eta \in S_{k'}})_{i < \omega}$ be a Morley sequence in $r$ with $(d_{\eta,0} e_{\eta,0})_{\eta \in S_{k'}} = (d_\eta e_\eta)_{\eta \in S_{k'}}$. We set $d_{0^{k'-2} \smallfrown i \smallfrown \xi} e_{0^{k'-2} \smallfrown i \smallfrown \xi} = d_{0^{k'-1} \smallfrown \xi, i} e_{0^{k'-1} \smallfrown \xi, i}$ for all $i < \omega$ and $\xi \in \omega^{\leq 2k+2-k'}$. We directly get (P)$_{k'-1}$ for $\eta \in S_{k'} - \{0^{k'-2}\}$ by virtue of $((d_{\eta,i} e_{\eta,i})_{\eta \in S_{k'}})_{i < \omega}$ being a Morley sequence. By Claim 1(iv) we have that $(d_{0^{k'-2} \smallfrown i} e_{0^{k'-2} \smallfrown i})_{i < \omega}$ is a Morley sequence in $q$. So we can find $d_{0^{k'-2}} e_{0^{k'-2}}$ such that $(d_{0^{k'-2} \smallfrown i} e_{0^{k'-2} \smallfrown i})_{i < \omega} d_{0^{k'-2}} e_{0^{k'-2}} \equivls_M (\beta_i \gamma_i)_{i < \omega} \beta \gamma$ and that concludes the construction of $(d_\eta e_\eta)_{\eta \in S_{k'-1}}$.

Similarly as in the proof of \thref{lem:q-spread-out-trees}, we will now show by induction on $n \leq k$ that the following holds.

\textbf{(Q)$_n$} Let $\eta_{2k-2m} \in \omega^{2k-2m}$ and $\nu_{2k-2m+1} \in \omega^{2k-2m+1}$ for $0 \leq m \leq n$. Suppose that $\eta_{2k-2m} \triangleleft \nu_{2k-2m+1}$ for all $0 \leq m \leq n$, $\eta_{2k} >_\lex \eta_{2k-2} >_\lex \ldots >_\lex \eta_{2k - 2n}$ and for all $0 \leq m' < m \leq n$ we have that $\eta_{2k - 2m} \wedge \eta_{2k - 2m'} \in \omega^{2k - 2m - 1}$. Then $(d_{\nu_{2k - 2m + 1}} e_{\nu_{2k - 2m + 1}} d_{\eta_{2k-2m}} e_{\eta_{2k-2m}})_{m \leq n}$ is a Morley sequence in $q'|_{y_0 z_0 y z}$.

For $n = 0$ this follows immediately from (P)$_1$ and Claim 1(vi). So suppose (Q)$_n$ holds for some $n < k$ and let $\eta_{2k-2m} \in \omega^{2k-2m}$ and $\nu_{2k-2m+1} \in \omega^{2k-2m+1}$ for $0 \leq m \leq n+1$ be as in the statement of (Q)$_{n+1}$. For any $m < n$ we have that $\eta_{2k-2m} \wedge \eta_{2k - 2n - 2} = \eta_{2k - 2n - 2}|_{2k-2n-3}$. So we can write $\eta_{2k - 2n - 2} = \xi \smallfrown i$ for some $\xi \in \omega^{2k-2n-3}$ and $i < \omega$. We then have $\eta_{2k-2m} \trianglerighteq \xi \smallfrown i'$ for some $i' < i$ for all $m \leq n$. So it follows from (P)$_1$, Claim 1(vi) and the induction hypothesis that $(d_{\nu_{2k - 2m + 1}} e_{\nu_{2k - 2m + 1}} d_{\eta_{2k-2m}} e_{\eta_{2k-2m}})_{m \leq n+1}$ is a Morley sequence in $q'|_{y_0 z_0 y z}$.

By exactly the same argument we also have the following condition. It differs from (Q)$_n$ in that the levels have been shifted by one (so we only consider it for $n < k$).

\textbf{(Q')$_n$} Let $\eta_{2k-2m-1} \in \omega^{2k-2m-1}$ and $\nu_{2k-2m} \in \omega^{2k-2m}$ for $0 \leq m \leq n$. Suppose that $\eta_{2k-2m-1} \triangleleft \nu_{2k-2m}$ for all $0 \leq m \leq n$, $\eta_{2k-1} >_\lex \eta_{2k-3} >_\lex \ldots >_\lex \eta_{2k - 2n - 1}$ and for all $0 \leq m' < m \leq n$ we have that $\eta_{2k - 2m - 1} \wedge \eta_{2k - 2m' - 1} \in \omega^{2k - 2m - 2}$. Then $(d_{\nu_{2k - 2m}} e_{\nu_{2k - 2m}} d_{\eta_{2k-2m-1}} e_{\eta_{2k-2m-1}})_{m \leq n}$ is a Morley sequence in $q'|_{y_0 z_0 y z}$.

Now let $(d_\eta' e_\eta')_{\eta \in \omega^{2k+1}}$ be an $s$-indiscernible over $M$ tree which is $EM_{s}$-based on $(d_\eta e_\eta)_{\eta \in \omega^{2k+1}}$ over $M$. We put $g_i = d_{0^{2(k-i)+1}}'$ for $i \leq k$, and for $i < k$ we put $h_i = e_{0^{2(k-i)}}'$, $g_i' = d_{0^{2(k-i)-1} \smallfrown 1 \smallfrown 0}'$, $h_i' = e_{0^{2(k-i)-1} \smallfrown 1}'$, $g_{i+1}'' = d_{0^{2(k-i-1)} \smallfrown 1}'$ and $h_i'' = e_{0^{2(k-i-1)} \smallfrown 1 \smallfrown 0}'$, see Figure \ref{fig:independence-theorem-tree}. Then conditions (i) and (ii) from Claim 2 follow from (Q)$_k$ and (Q')$_{k-1}$, while conditions (iii) and (iv) follow from $s$-indiscernibility.
\begin{figure}[ht]
\centering
\input{graphics/indep-thm-tree}
\caption{Choice of the $g_i h_i g'_i h'_i g''_i h''_i$.}
\label{fig:independence-theorem-tree}
\end{figure}
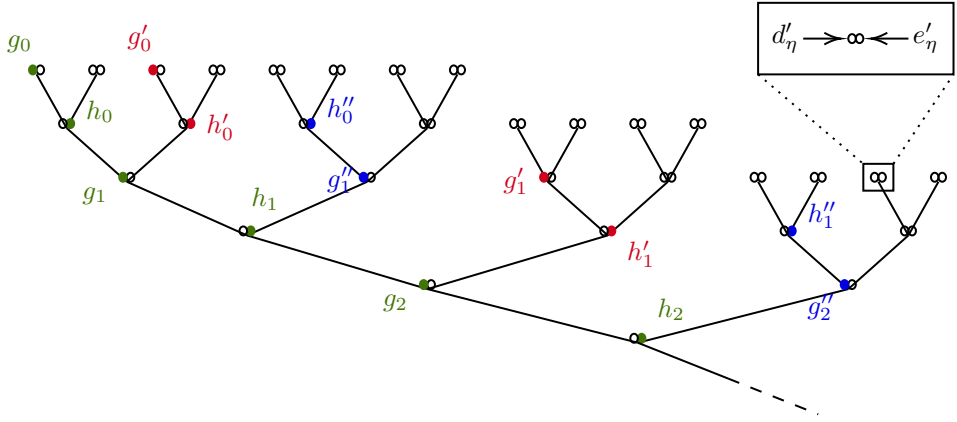
\end{proof}
Now that we have proved the independence theorem, we first note some useful immediate consequences in \thref{cor:g-compact}. After that, the rest of this section will be devoted to proving a stronger version of the independence theorem, \thref{thm:strong-independence-theorem}.
\begin{definition}
\thlabel{def:kim-morley-sequence}
Let $I$ be a linear order. We will say that $(a_i)_{i\in I}$ is a \emph{$\ind_M^K$-independent sequence} if $a_i \ind_M^K a_{<i}$ for every $i\in I$. We will say that $(a_i)_{i\in I}$ is \emph{$\ind_M^K$-Morley} if it is $\ind_M^K$-independent and $M$-indiscernible.
\end{definition}
\begin{lemma}
\thlabel{lem:kim-dividing-type-def}
Let $T$ be thick NSOP$_1$ with an e.c.\ model $M$, and let $a,b,c$ be any tuples of parameters and $x$ a tuple of variables. Then there exists a (partial) type $\Sigma(x,y)$ over $Mab$ such that for any $x,y$ we have that $$\models \Sigma(x,y) \iff (y\equiv_{Mb} c)\wedge (xa\ind^K_M yb).$$
In particular, taking $y=\emptyset$, we get that the condition $xa\ind^K_M b$ is type definable over $Mab$ in the variable $x$. 
\end{lemma}
\begin{proof}
Let $q(y, z)$ be a global $M$-Ls-invariant type extending $\tp(cb/M)$. Then, by Kim's Lemma, for any $y\equiv_{Mb}c$ and any $x$, the condition $ xa\ind^K_M yb$ is equivalent to:
\[
\exists (y_iz_i)_{i<\omega} \left( q^{\otimes \omega}|_{M}((y_iz_i)_{i<\omega}) \text{ and } y_0z_0 = yb \text{ and } (y_iz_i)_{i<\omega} \text{ is $Max$-indiscernible} \right),
\]
which is clearly a type-definable over $Mab$ condition by thickness. 
\end{proof}
In particular, we get that being an $\ind^K_M$-independent sequence in a fixed type over $M$ is type-definable over $M$ in thick NSOP$_1$ theories. That is, for a linear order $I$ we can use the type
\[
\bigcup_{i \in I} \Sigma(x_{<i}, x_i),
\]
where $\Sigma$ is as in \thref{lem:kim-dividing-type-def}. Then by symmetry, \thref{thm:symmetry}, this (partial) type expresses exactly what we wanted.
\begin{corollary}
\thlabel{cor:g-compact}
Suppose $T$ is thick NSOP$_1$ with an e.c.\ model $M$.
\begin{enumerate}[label=(\roman*)]
\item If $a\ind^K_M b$ and $a\equivls_M b$ then there exists an infinite $M$-indiscernible sequence starting with $(a,b)$.
\item If $a\equivls_M b$ then $a$ and $b$ are at Lascar distance at most 2 over $M$. In particular, Lascar equivalence over e.c.\ models is type-definable.
\item (Generalised independence theorem) Let $(a_i)_{i<\kappa}$ be an $\ind^K_M$-independent sequence. Suppose $b_i\equivls_M b$ and $b_i\ind^K_M a_i$ for every $i<\kappa$. Then there exists $b'$ such that $b'a_i\equivls_{M}b_ia_i$ for every $i<\kappa$ and $b'\ind^K_M (a_i)_{i<\kappa}$.
\end{enumerate}
\end{corollary}
\begin{proof}
(i) We can inductively find a sequence $(c_i)_{i<\omega}$ such that $c_0c_1=ab$, $c_i\equivls_M b$, $c_i\ind^K_M{c_{<i}}$ and $c_ic_j\equiv_M ab$ for all $i<j<\omega$: indeed, if we have constructed $c_{\leq i}$ then by the independence theorem we can choose $c_{i+1}$ such that $c_{i+1}\equivls_{Mc_{<i}} c_i$, 
$c_ic_{i+1}\equivls_M ab$ and $c_{i+1}\ind^K_M c_{\leq i}$. 

By compactness we can find a sequence $(c'_i)_{i<\lambda_{|T|+|Ma|}}$ with $c'_ic'_j\equiv_M ab$ for all $i<j<\lambda_{|T|+|Ma|}$. Choose an $M$-indiscernible sequence $(d_i)_{i<\omega}$ based on $(c'_i)_{i<\lambda_{|T|+|Ma|}}$ over $M$. Then $d_0d_1\equiv_M ab$, so we conclude that the pair $(a,b)$ starts an $M$-indiscernible sequence.\\
(ii) By extension (\thref{cor:extension}) we can choose $c\equivls_M a$ with $c\ind^K_M ab$. By (i) we get that $(a,c)$ and $(b,c)$ both start $M$-indiscernible sequences.\\
(iii) We choose inductively a sequence $(b'_j)_{j\leq \kappa}$ such that $b_j'a_i\equivls_{M}b_ia_i$ for every $i< j$ and $b'_j\ind^K_M (a_i)_{i< j}$, so that we can put $b':=b_\kappa$. The successor step follows directly by the independence theorem, and the limit step follows by type-definability of Lascar equivalence over $M$, \thref{lem:kim-dividing-type-def} and compactness.
\end{proof}
\begin{definition}
\thlabel{def:spread-out-over-m}
We will say that a tree $(c_\eta)_{\eta\in \omega^{\leq k}}$ is \emph{spread-out over $M$} if $(c_{\trianglerighteq \eta\smallfrown i})_{i<\omega}$ is a Morley sequence in some global $M$-Ls-invariant type for every $\eta\in \omega^{\leq k-1}$.
\end{definition}
There are two differences between being spread-out over $M$ and being $q$-spread-out over $M$ (see \thref{def:spread-out} for the latter). In the latter the global $M$-Ls-invariant type involved has to be $q$, while the former just requires some global $M$-Ls-invariant type. The second difference is in the sequence in the tree that is required to be a Morley sequence. In the former we consider a sequence of subtrees above some fixed node, all at the same level. In the latter we consider a sequence of nodes in the tree, one in every level (except for the root), as pictured in Figure \ref{fig:spread-out}.

The following lemma follows from the independence theorem exactly as in \cite[Lemma 6.2/Remark 6.3]{kaplan_kim-independence_2020}, so we omit the proof.
\begin{fact}
\thlabel{fact:spread-out-tree-gives-consistence}
Suppose $T$ is thick NSOP$_1$, $M$ an e.c.\ model, $a\ind^K_M b$, $(b_\eta)_{\eta\in \omega^{\leq k}}$ (with $k < \omega$) is a spread-out over $M$ tree such that $b_\eta\ind^K_M b_{\triangleright \eta}$ and $b_\eta \equivls_M b$ for every $\eta \in \omega^{\leq k}$. Then, writing $p(x,b)= tp(a/Mb)$, there exists $a'\models \bigcup_{\eta\in \omega^{\leq k}} p(x,b_\eta)$ with $a'\ind^K_M (b_\eta)_{\eta\in \omega^{\leq k}}$ and $a' \equivls_M a$.
\end{fact}
\begin{lemma}
\thlabel{lem:find-kim-morley-cr-morley-sequence}
Suppose $T$ is thick NSOP$_1$, $M$ an e.c.\ model, $b\equivls_M b'$, $b\ind^K_M b'$ and $I$ is a linear order with two distinct elements $0$ and $1$. Then there is a $\ind_M^K$-Morley parallel-Morley in $\tp(b/M)$ sequence $(b_i)_{i\in I}$  with $b_0=b$ and $b_1=b'$.
\end{lemma}
\begin{proof}
By extension (\thref{cor:extension}) there is a $\lambda_T$-saturated model $N\supseteq Mb$ with $N\ind^K_M b'$. Then there is a $\lambda_T$-saturated model $N'\supseteq Mb'$ with $N'\equivls_M N$. Hence, again by extension, we can find $N''\equivls_{Mb'} N'$ with $N\ind^K_M N''$. So replacing $b$ and $b'$ by $N$ and $N''$ we may assume without loss of generality that $b$ and $b'$ are $\lambda_T$-saturated models containing $M$. Put $\lambda=|b|$ and (using \thref{lem:(*)}) choose a global $M$-Ls-invariant extension $q$ of $\Lstp(b'/M)$ satisfying $(*)_{\lambda}$. 

We claim that it is enough to show that for any $1<k<\omega$ there is a $\ind_M^K$-independent parallel-Morley sequence $(a_i)_{i<k}$ in $q$ over $M$ with $a_i\equivls_M b'$ and $a_ia_j\equiv_M bb'$ for all $i<j<k$: indeed, if we show this, then, as all these conditions are type-definable by \thref{lem:kim-dividing-type-def} and \thref{cor:g-compact}(ii), we can find by compactness a $\ind^K$-independent over $M$ parallel-Morley sequence $(a_i)_{i < \lambda_{|T|+|b|}}$ in $q$ over $M$ with $a_ia_j\equiv_M bb'$ for each $i<j$, and then taking an $M$-indiscernible sequence indexed by $I$ which is based on  $(a_i)_{i < \lambda_{|T|+|Mb|}}$ over $M$ and moving it by an automorphism to guarantee  that $b_0b_1=bb'$ (note this may change $q$) will do the job.

So fix any  $1<k<\omega$ and put $p=tp(b'/Mb)$. By backward induction on $k'=k+1,k,\dots,1$ we will define trees $(c_\eta)_{\eta\in S_{k'}}$ where $S_{k'}:=\{\xi\in {\omega}^{\leq k}:0^{k'-1}\trianglelefteq \xi\}$ such that for each $k'$ the tree  $(c_\eta)_{\eta\in S_{k'}}$ is spread-out over $M$ and satisfies the following conditions:
\begin{enumerate}[label=(A\arabic*)$_{k'}$]
\item $c_{\eta}c_\nu\equiv_M bb'$ for any $\nu,\eta \in S_{k'}$ with $\nu\triangleleft \eta$ and $c_\eta \equivls_M b'$ for any $\eta\in S_{k'}$;
\item $(c_\eta)_{\eta\in S_{k'}}$ is $q$-spread-out over $M$;
\item $c_\eta\ind^K_M c_{\triangleright \eta}$ for every $\eta\in S_{k'}$.
\end{enumerate}
For $k'=k+1$ putting $c_{0^k}=b'$ works. 
Now suppose we are done for some $k'\leq k+1$. By \thref{fact:spread-out-tree-gives-consistence} we can find $ c'\models \bigcup_{\eta\in S_{k'}} p(x,c_\eta)$ with $c'\equivls_M b'$ and
$c'\ind^K_M  (c_\eta)_{\eta\in S_{k'}}$.
By (A1)$_{k'}$ there is a tuple $d$ such $c_{0^{k'-1}}(c_{\eta})_{\eta\in S^*_{k'}}\equivls_M b'd$. Now, by $(*)_{\lambda}$ there is some global $M$-Ls-invariant type $r(x,z)\supseteq q(x)$  which extends $\Lstp(b'd/M)=\Lstp( c_{0^{k'-1}}(c_{\eta})_{\eta\in S^*_{k'}}/M)$. Also, as $c'\ind^K_M (c_\eta)_{\eta\in S_{k'}}$ and  $c_\eta$'s are $\lambda_T$-saturated models (as $b'$ is), we get by \thref{cor:q-ls-dividing-is-q-dividing} that  $\Lstp(c'/M(c_\eta)_{\eta\in S_{k'}})$ does not $r(x,z)$-Ls-divide over $M$. Hence, there is an $Mc'$-indiscernible Morley sequence  $I:=((c_{\eta,i})_{\eta \in S_{k'}})_{i<{\omega}}$ in $r(x,z)$ over $M$ with $c_{\eta,0}=c_\eta$ for each $\eta\in S_{k'}$. By the chain condition (\thref{lem:chain-lemma}) we have that $c'\ind^K_M I$.  Thus, putting $c_{0^{k'-2}\smallfrown i\smallfrown \zeta}:=c_{0^{k'-1}\smallfrown \zeta,i}$ for all $i<\omega$, $\zeta\in {\omega}^{\leq k+1-k'}  $, and $c_{0^{k'-2}}:=c'$,  we immediately get that the tree $(c_{\eta})_{\eta\in S_{k'-1}} $ satisfies (A3)$_{k'-1}$. (A1)$_{k'-1}$ follows from (A1)$_{k'}$, the choice of $c'$ and $Mc'$ indiscernibility of $I$. (A2)$_{k'-1}$ follows from (A2)$_{k'}$ and \thref{lem:q-spread-out-trees}(i). This completes the inductive construction.

Letting $(c'_{\eta})_{\eta\in {\omega}^{\leq k}}$ be an $s$-indiscernible over $M$ tree which is $EM_{s}$-based on $(c_{\eta})_{\eta\in {\omega}^{\leq k}}$ over $Mb'$, we get that $(c'_{\eta})_{\eta\in {\omega}^{\leq k}}$ satisfies (A1)$_{1}$ and (A3)$_1$ (by \thref{lem:kim-dividing-type-def} and \thref{cor:g-compact}(ii)) and is weakly $q$-spread-out over $M$ by \thref{lem:q-spread-out-trees}(ii).

Put $a_i:=c'_{0^{k-i}}$ for $i<k$. Then by \thref{lem:q-spread-out-trees}(iii) we have that $(a_i)_{i<k}$ is parallel-Morley in $q$ over $M$. Also, $a_ia_j\equiv_M bb'$ for all $i<j<k$ by (A1)$_{1}$, and $(a_i)_{i<k}$ is $\ind^K_M$-independent over $M$ by (A3)$_1$. This completes the proof.
\end{proof}

\begin{lemma}[Chain condition for $\ind^K$-Morley sequences]
\thlabel{lem:chain-lemma-kim-morley-sequences}
Suppose $T$ is thick NSOP$_1$ with an e.c.\ model $M$, $(d_i)_{i\in I}$ is an infinite $\ind_M^K$-Morley sequence and $a\ind^K_M d_{i_0}$ for some $i_0\in I$. Then there exists $a^*d_{i_0}\equivls_{M}ad_{i_0}$ such that $(d_i)_{i\in I}$ is indiscernible over $Ma^*$ and $a^*\ind^K_M (d_i)_{i\in I}$.
\end{lemma}
\begin{proof}
 By compactness there is a $\ind_M^K$-Morley sequence $(d''_i)_{i<\lambda}$  such that $(d_i)_{i\in I}\frown (d''_i)_{i<\lambda}$ is $M$-indiscernible, 
 where $\lambda=\lambda_{|T|+|Mad_0|+|I|}$.
As $d_{i_0}\equivls_Md''_0$, 
 $a\ind^K_M d_{i_0}$ and $(d''_i)_{i<\lambda}$ is  $\ind^K$-independent over $M$ , we get by \thref{cor:g-compact}(iii) that there exists $a'$ with $a'd''_i\equivls_{M} ad_{i_0}$ for every $i<\lambda$ and $a'\ind^K_M (d''_i)_{i<\lambda}$. Let $(d'_i)_{i\in I}$ be an $Ma'$-indiscernible sequence based on $(d''_i)_{i<\lambda}$ over $Maa'(d_i)_{i\in I}$. Then (by finite character and invariance of $\ind^K$) $a'\ind^K_M(d'_i)_{i\in I}$, $(d'_i)_{i\in I}\equivls_M(d_i)_{i\in I}$ (as $(d_i)_{i\in I}\frown (d'_i)_{i\in I}$ is indiscernible over $M$), and $a'd'_{i_0}\equivls_M ad_{i_0}$. Hence, letting $f$ be a Lascar strong automorphism over $M$ sending $(d'_i)_{i\in I}$ to $(d_i)_{i\in I}$ and putting $a^*=f(a')$ we get that $a^*\ind^K_M (d_i)_{i\in I}$ and $(d_i)_{i\in I}$ is $Ma^*$-indiscernible. Also $a^*d_{i_0}\equivls_M a'd'_{i_0}\equivls_M ad_{i_0}$ as required.
 \end{proof}
\begin{theorem}[Strong independence theorem]
\thlabel{thm:strong-independence-theorem}
Suppose $T$ is thick NSOP$_1$ with an e.c.\ model $M$, $a_0\ind_M^K b$, $a_1\ind^K_M c$, $b\ind^K_M c$, and $a_0\equivls_M a_1$. Then there is $a$ such that $a \equivls_{Mb} a_0$, $a \equivls_{Mc} a_1$, 
$a\ind ^K_M bc$, $b\ind^K_M ac$, $c\ind^K_M ab$. 
\end{theorem}
\begin{proof}
By a similar trick as at the start of the proof of \thref{thm:independence-theorem} we may assume that $b$ and $c$ enumerate $\lambda_T$-saturated models containing $M$.

By the independence theorem there is $a_2$ with $a_2\equivls_{Mb} a_0$, $a_2\equivls_{Mc} a_1$ and $a_2\ind^K_M bc$. By extension (\thref{cor:extension}) there is $b'\equivls_{Mc} b$ such that $b\ind^K_M b'c$, so $b'c\ind^K_M b$ by symmetry. By extension again, there is $c'\equivls_{Mb}c$ with $b'c\ind^K_M bc'$. As $b'c\equivls_M bc\equivls_M bc'$, we get by \thref{lem:find-kim-morley-cr-morley-sequence} that there is a $\ind_M^K$-Morley parallel-Morley in $\tp(bc/M)$ sequence $I=(b_i,c_i)_{i\in \mathbb{Z}}$  with $b_0c_0=bc'$ and $b_1c_1=b'c$. As $a_2\ind^K_M bc$, we get by \thref{lem:chain-lemma-kim-morley-sequences} that there is some $a$ such that $abc'\equivls_M a_2bc$, $I$ is $Ma$-indiscernible and $a\ind^K_M I$.
 
Then by monotonicity $a\ind^K_Mbc$. We also have $ab\equivls_{M} a_2b\equivls_{M} a_0b$, and, by indiscernibility, $ac\equivls_M ac'\equivls_Ma_2c\equivls_Ma_1c$. Since $b$ and $c$ were assumed to enumerate $\lambda_T$-saturated models we get $a \equivls_{Mb} a_0$ and $a \equivls_{Mc} a_1$. Also, $(b_i)_{i\leq 0}$ is an $Mac$-indiscernible parallel-Morley sequence in $\tp(b/M)$ with $b_0=b$, which gives $b\ind^K_M ac$ by \thref{cor:cr-kims-lemma}. Similarly, as $(c_i)_{i\geq 1}$ is an $Mab$-indiscernible parallel-Morley sequence in $\tp(c/M)$ with $c_1=c$, we get that $c\ind^K_M ab$.
\end{proof}

%% file: graphics/indep-thm-tree.tex
\begin{tikzpicture}[x=0.75pt,y=0.75pt,yscale=-0.9,xscale=0.75]

\draw [color={rgb, 255:red, 0; green, 0; blue, 0 }  ,draw opacity=1 ]   (30,60) -- (50,90) ;
\draw [color={rgb, 255:red, 0; green, 0; blue, 0 }  ,draw opacity=1 ]   (50,90) -- (70,60) ;
\draw [color={rgb, 255:red, 0; green, 0; blue, 0 }  ,draw opacity=1 ]   (110,60) -- (130,90) ;
\draw [color={rgb, 255:red, 0; green, 0; blue, 0 }  ,draw opacity=1 ]   (130,90) -- (150,60) ;
\draw [color={rgb, 255:red, 0; green, 0; blue, 0 }  ,draw opacity=1 ]   (50,90) -- (90,120) ;
\draw [color={rgb, 255:red, 0; green, 0; blue, 0 }  ,draw opacity=1 ]   (90,120) -- (130,90) ;
\draw [color={rgb, 255:red, 0; green, 0; blue, 0 }  ,draw opacity=1 ]   (90,120) -- (170,150) ;
\draw [color={rgb, 255:red, 0; green, 0; blue, 0 }  ,draw opacity=1 ]   (170,150) -- (250,120) ;
\draw [color={rgb, 255:red, 0; green, 0; blue, 0 }  ,draw opacity=1 ]   (170,150) -- (290,180) ;
\draw [color={rgb, 255:red, 0; green, 0; blue, 0 }  ,draw opacity=1 ]   (290,180) -- (410,150) ;
\draw [color={rgb, 255:red, 0; green, 0; blue, 0 }  ,draw opacity=1 ]   (290,180) -- (430,210) ;
\draw [color={rgb, 255:red, 0; green, 0; blue, 0 }  ,draw opacity=1 ]   (430,210) -- (570,180) ;
\draw  [color={rgb, 255:red, 0; green, 0; blue, 0 }  ,draw opacity=1 ] (30,57.5) .. controls (30,56.12) and (31.12,55) .. (32.5,55) .. controls (33.88,55) and (35,56.12) .. (35,57.5) .. controls (35,58.88) and (33.88,60) .. (32.5,60) .. controls (31.12,60) and (30,58.88) .. (30,57.5) -- cycle ;
\draw  [color={rgb, 255:red, 65; green, 117; blue, 5 }  ,draw opacity=1 ][fill={rgb, 255:red, 65; green, 117; blue, 5 }  ,fill opacity=1 ] (25,57.5) .. controls (25,56.12) and (26.12,55) .. (27.5,55) .. controls (28.88,55) and (30,56.12) .. (30,57.5) .. controls (30,58.88) and (28.88,60) .. (27.5,60) .. controls (26.12,60) and (25,58.88) .. (25,57.5) -- cycle ;
\draw  [color={rgb, 255:red, 0; green, 0; blue, 0 }  ,draw opacity=1 ] (70,57.5) .. controls (70,56.12) and (71.12,55) .. (72.5,55) .. controls (73.88,55) and (75,56.12) .. (75,57.5) .. controls (75,58.88) and (73.88,60) .. (72.5,60) .. controls (71.12,60) and (70,58.88) .. (70,57.5) -- cycle ;
\draw  [color={rgb, 255:red, 0; green, 0; blue, 0 }  ,draw opacity=1 ] (65,57.5) .. controls (65,56.12) and (66.12,55) .. (67.5,55) .. controls (68.88,55) and (70,56.12) .. (70,57.5) .. controls (70,58.88) and (68.88,60) .. (67.5,60) .. controls (66.12,60) and (65,58.88) .. (65,57.5) -- cycle ;

\draw  [color={rgb, 255:red, 0; green, 0; blue, 0 }  ,draw opacity=1 ] (110,57.5) .. controls (110,56.12) and (111.12,55) .. (112.5,55) .. controls (113.88,55) and (115,56.12) .. (115,57.5) .. controls (115,58.88) and (113.88,60) .. (112.5,60) .. controls (111.12,60) and (110,58.88) .. (110,57.5) -- cycle ;
\draw  [color={rgb, 255:red, 208; green, 2; blue, 27 }  ,draw opacity=1 ][fill={rgb, 255:red, 208; green, 2; blue, 27 }  ,fill opacity=1 ] (105,57.5) .. controls (105,56.12) and (106.12,55) .. (107.5,55) .. controls (108.88,55) and (110,56.12) .. (110,57.5) .. controls (110,58.88) and (108.88,60) .. (107.5,60) .. controls (106.12,60) and (105,58.88) .. (105,57.5) -- cycle ;
\draw  [color={rgb, 255:red, 0; green, 0; blue, 0 }  ,draw opacity=1 ] (150,57.5) .. controls (150,56.12) and (151.12,55) .. (152.5,55) .. controls (153.88,55) and (155,56.12) .. (155,57.5) .. controls (155,58.88) and (153.88,60) .. (152.5,60) .. controls (151.12,60) and (150,58.88) .. (150,57.5) -- cycle ;
\draw  [color={rgb, 255:red, 0; green, 0; blue, 0 }  ,draw opacity=1 ] (145,57.5) .. controls (145,56.12) and (146.12,55) .. (147.5,55) .. controls (148.88,55) and (150,56.12) .. (150,57.5) .. controls (150,58.88) and (148.88,60) .. (147.5,60) .. controls (146.12,60) and (145,58.88) .. (145,57.5) -- cycle ;
\draw  [color={rgb, 255:red, 65; green, 117; blue, 5 }  ,draw opacity=1 ][fill={rgb, 255:red, 65; green, 117; blue, 5 }  ,fill opacity=1 ] (50,87.5) .. controls (50,86.12) and (51.12,85) .. (52.5,85) .. controls (53.88,85) and (55,86.12) .. (55,87.5) .. controls (55,88.88) and (53.88,90) .. (52.5,90) .. controls (51.12,90) and (50,88.88) .. (50,87.5) -- cycle ;
\draw  [color={rgb, 255:red, 0; green, 0; blue, 0 }  ,draw opacity=1 ] (45,87.5) .. controls (45,86.12) and (46.12,85) .. (47.5,85) .. controls (48.88,85) and (50,86.12) .. (50,87.5) .. controls (50,88.88) and (48.88,90) .. (47.5,90) .. controls (46.12,90) and (45,88.88) .. (45,87.5) -- cycle ;
\draw  [color={rgb, 255:red, 208; green, 2; blue, 27 }  ,draw opacity=1 ][fill={rgb, 255:red, 208; green, 2; blue, 27 }  ,fill opacity=1 ] (130,87.5) .. controls (130,86.12) and (131.12,85) .. (132.5,85) .. controls (133.88,85) and (135,86.12) .. (135,87.5) .. controls (135,88.88) and (133.88,90) .. (132.5,90) .. controls (131.12,90) and (130,88.88) .. (130,87.5) -- cycle ;
\draw  [color={rgb, 255:red, 0; green, 0; blue, 0 }  ,draw opacity=1 ] (125,87.5) .. controls (125,86.12) and (126.12,85) .. (127.5,85) .. controls (128.88,85) and (130,86.12) .. (130,87.5) .. controls (130,88.88) and (128.88,90) .. (127.5,90) .. controls (126.12,90) and (125,88.88) .. (125,87.5) -- cycle ;
\draw  [color={rgb, 255:red, 0; green, 0; blue, 0 }  ,draw opacity=1 ] (90,117.5) .. controls (90,116.12) and (91.12,115) .. (92.5,115) .. controls (93.88,115) and (95,116.12) .. (95,117.5) .. controls (95,118.88) and (93.88,120) .. (92.5,120) .. controls (91.12,120) and (90,118.88) .. (90,117.5) -- cycle ;
\draw  [color={rgb, 255:red, 65; green, 117; blue, 5 }  ,draw opacity=1 ][fill={rgb, 255:red, 65; green, 117; blue, 5 }  ,fill opacity=1 ] (85,117.5) .. controls (85,116.12) and (86.12,115) .. (87.5,115) .. controls (88.88,115) and (90,116.12) .. (90,117.5) .. controls (90,118.88) and (88.88,120) .. (87.5,120) .. controls (86.12,120) and (85,118.88) .. (85,117.5) -- cycle ;
\draw [color={rgb, 255:red, 0; green, 0; blue, 0 }  ,draw opacity=1 ]   (190,60) -- (210,90) ;
\draw [color={rgb, 255:red, 0; green, 0; blue, 0 }  ,draw opacity=1 ]   (210,90) -- (230,60) ;
\draw [color={rgb, 255:red, 0; green, 0; blue, 0 }  ,draw opacity=1 ]   (270,60) -- (290,90) ;
\draw [color={rgb, 255:red, 0; green, 0; blue, 0 }  ,draw opacity=1 ]   (290,90) -- (310,60) ;
\draw [color={rgb, 255:red, 0; green, 0; blue, 0 }  ,draw opacity=1 ]   (210,90) -- (250,120) ;
\draw [color={rgb, 255:red, 0; green, 0; blue, 0 }  ,draw opacity=1 ]   (250,120) -- (290,90) ;
\draw  [color={rgb, 255:red, 0; green, 0; blue, 0 }  ,draw opacity=1 ] (190,57.5) .. controls (190,56.12) and (191.12,55) .. (192.5,55) .. controls (193.88,55) and (195,56.12) .. (195,57.5) .. controls (195,58.88) and (193.88,60) .. (192.5,60) .. controls (191.12,60) and (190,58.88) .. (190,57.5) -- cycle ;
\draw  [color={rgb, 255:red, 0; green, 0; blue, 0 }  ,draw opacity=1 ] (185,57.5) .. controls (185,56.12) and (186.12,55) .. (187.5,55) .. controls (188.88,55) and (190,56.12) .. (190,57.5) .. controls (190,58.88) and (188.88,60) .. (187.5,60) .. controls (186.12,60) and (185,58.88) .. (185,57.5) -- cycle ;
\draw  [color={rgb, 255:red, 0; green, 0; blue, 0 }  ,draw opacity=1 ] (230,57.5) .. controls (230,56.12) and (231.12,55) .. (232.5,55) .. controls (233.88,55) and (235,56.12) .. (235,57.5) .. controls (235,58.88) and (233.88,60) .. (232.5,60) .. controls (231.12,60) and (230,58.88) .. (230,57.5) -- cycle ;
\draw  [color={rgb, 255:red, 0; green, 0; blue, 0 }  ,draw opacity=1 ] (225,57.5) .. controls (225,56.12) and (226.12,55) .. (227.5,55) .. controls (228.88,55) and (230,56.12) .. (230,57.5) .. controls (230,58.88) and (228.88,60) .. (227.5,60) .. controls (226.12,60) and (225,58.88) .. (225,57.5) -- cycle ;
\draw  [color={rgb, 255:red, 0; green, 0; blue, 0 }  ,draw opacity=1 ] (270,57.5) .. controls (270,56.12) and (271.12,55) .. (272.5,55) .. controls (273.88,55) and (275,56.12) .. (275,57.5) .. controls (275,58.88) and (273.88,60) .. (272.5,60) .. controls (271.12,60) and (270,58.88) .. (270,57.5) -- cycle ;
\draw  [color={rgb, 255:red, 0; green, 0; blue, 0 }  ,draw opacity=1 ] (265,57.5) .. controls (265,56.12) and (266.12,55) .. (267.5,55) .. controls (268.88,55) and (270,56.12) .. (270,57.5) .. controls (270,58.88) and (268.88,60) .. (267.5,60) .. controls (266.12,60) and (265,58.88) .. (265,57.5) -- cycle ;
\draw  [color={rgb, 255:red, 0; green, 0; blue, 0 }  ,draw opacity=1 ] (310,57.5) .. controls (310,56.12) and (311.12,55) .. (312.5,55) .. controls (313.88,55) and (315,56.12) .. (315,57.5) .. controls (315,58.88) and (313.88,60) .. (312.5,60) .. controls (311.12,60) and (310,58.88) .. (310,57.5) -- cycle ;
\draw  [color={rgb, 255:red, 0; green, 0; blue, 0 }  ,draw opacity=1 ] (305,57.5) .. controls (305,56.12) and (306.12,55) .. (307.5,55) .. controls (308.88,55) and (310,56.12) .. (310,57.5) .. controls (310,58.88) and (308.88,60) .. (307.5,60) .. controls (306.12,60) and (305,58.88) .. (305,57.5) -- cycle ;
\draw  [color={rgb, 255:red, 0; green, 0; blue, 226 }  ,draw opacity=1 ][fill={rgb, 255:red, 0; green, 0; blue, 226 }  ,fill opacity=1 ] (210,87.5) .. controls (210,86.12) and (211.12,85) .. (212.5,85) .. controls (213.88,85) and (215,86.12) .. (215,87.5) .. controls (215,88.88) and (213.88,90) .. (212.5,90) .. controls (211.12,90) and (210,88.88) .. (210,87.5) -- cycle ;
\draw  [color={rgb, 255:red, 0; green, 0; blue, 0 }  ,draw opacity=1 ] (205,87.5) .. controls (205,86.12) and (206.12,85) .. (207.5,85) .. controls (208.88,85) and (210,86.12) .. (210,87.5) .. controls (210,88.88) and (208.88,90) .. (207.5,90) .. controls (206.12,90) and (205,88.88) .. (205,87.5) -- cycle ;
\draw  [color={rgb, 255:red, 0; green, 0; blue, 0 }  ,draw opacity=1 ] (290,87.5) .. controls (290,86.12) and (291.12,85) .. (292.5,85) .. controls (293.88,85) and (295,86.12) .. (295,87.5) .. controls (295,88.88) and (293.88,90) .. (292.5,90) .. controls (291.12,90) and (290,88.88) .. (290,87.5) -- cycle ;
\draw  [color={rgb, 255:red, 0; green, 0; blue, 0 }  ,draw opacity=1 ] (285,87.5) .. controls (285,86.12) and (286.12,85) .. (287.5,85) .. controls (288.88,85) and (290,86.12) .. (290,87.5) .. controls (290,88.88) and (288.88,90) .. (287.5,90) .. controls (286.12,90) and (285,88.88) .. (285,87.5) -- cycle ;
\draw  [color={rgb, 255:red, 0; green, 0; blue, 0 }  ,draw opacity=1 ] (250,117.5) .. controls (250,116.12) and (251.12,115) .. (252.5,115) .. controls (253.88,115) and (255,116.12) .. (255,117.5) .. controls (255,118.88) and (253.88,120) .. (252.5,120) .. controls (251.12,120) and (250,118.88) .. (250,117.5) -- cycle ;
\draw  [color={rgb, 255:red, 0; green, 0; blue, 226 }  ,draw opacity=1 ][fill={rgb, 255:red, 0; green, 0; blue, 226 }  ,fill opacity=1 ] (245,117.5) .. controls (245,116.12) and (246.12,115) .. (247.5,115) .. controls (248.88,115) and (250,116.12) .. (250,117.5) .. controls (250,118.88) and (248.88,120) .. (247.5,120) .. controls (246.12,120) and (245,118.88) .. (245,117.5) -- cycle ;
\draw [color={rgb, 255:red, 0; green, 0; blue, 0 }  ,draw opacity=1 ]   (350,90) -- (370,120) ;
\draw [color={rgb, 255:red, 0; green, 0; blue, 0 }  ,draw opacity=1 ]   (370,120) -- (390,90) ;
\draw [color={rgb, 255:red, 0; green, 0; blue, 0 }  ,draw opacity=1 ]   (430,90) -- (450,120) ;
\draw [color={rgb, 255:red, 0; green, 0; blue, 0 }  ,draw opacity=1 ]   (450,120) -- (470,90) ;
\draw [color={rgb, 255:red, 0; green, 0; blue, 0 }  ,draw opacity=1 ]   (370,120) -- (410,150) ;
\draw [color={rgb, 255:red, 0; green, 0; blue, 0 }  ,draw opacity=1 ]   (410,150) -- (450,120) ;
\draw  [color={rgb, 255:red, 0; green, 0; blue, 0 }  ,draw opacity=1 ] (350,87.5) .. controls (350,86.12) and (351.12,85) .. (352.5,85) .. controls (353.88,85) and (355,86.12) .. (355,87.5) .. controls (355,88.88) and (353.88,90) .. (352.5,90) .. controls (351.12,90) and (350,88.88) .. (350,87.5) -- cycle ;
\draw  [color={rgb, 255:red, 0; green, 0; blue, 0 }  ,draw opacity=1 ] (345,87.5) .. controls (345,86.12) and (346.12,85) .. (347.5,85) .. controls (348.88,85) and (350,86.12) .. (350,87.5) .. controls (350,88.88) and (348.88,90) .. (347.5,90) .. controls (346.12,90) and (345,88.88) .. (345,87.5) -- cycle ;
\draw  [color={rgb, 255:red, 0; green, 0; blue, 0 }  ,draw opacity=1 ] (390,87.5) .. controls (390,86.12) and (391.12,85) .. (392.5,85) .. controls (393.88,85) and (395,86.12) .. (395,87.5) .. controls (395,88.88) and (393.88,90) .. (392.5,90) .. controls (391.12,90) and (390,88.88) .. (390,87.5) -- cycle ;
\draw  [color={rgb, 255:red, 0; green, 0; blue, 0 }  ,draw opacity=1 ] (385,87.5) .. controls (385,86.12) and (386.12,85) .. (387.5,85) .. controls (388.88,85) and (390,86.12) .. (390,87.5) .. controls (390,88.88) and (388.88,90) .. (387.5,90) .. controls (386.12,90) and (385,88.88) .. (385,87.5) -- cycle ;
\draw  [color={rgb, 255:red, 0; green, 0; blue, 0 }  ,draw opacity=1 ] (430,87.5) .. controls (430,86.12) and (431.12,85) .. (432.5,85) .. controls (433.88,85) and (435,86.12) .. (435,87.5) .. controls (435,88.88) and (433.88,90) .. (432.5,90) .. controls (431.12,90) and (430,88.88) .. (430,87.5) -- cycle ;
\draw  [color={rgb, 255:red, 0; green, 0; blue, 0 }  ,draw opacity=1 ] (425,87.5) .. controls (425,86.12) and (426.12,85) .. (427.5,85) .. controls (428.88,85) and (430,86.12) .. (430,87.5) .. controls (430,88.88) and (428.88,90) .. (427.5,90) .. controls (426.12,90) and (425,88.88) .. (425,87.5) -- cycle ;
\draw  [color={rgb, 255:red, 0; green, 0; blue, 0 }  ,draw opacity=1 ] (470,87.5) .. controls (470,86.12) and (471.12,85) .. (472.5,85) .. controls (473.88,85) and (475,86.12) .. (475,87.5) .. controls (475,88.88) and (473.88,90) .. (472.5,90) .. controls (471.12,90) and (470,88.88) .. (470,87.5) -- cycle ;
\draw  [color={rgb, 255:red, 0; green, 0; blue, 0 }  ,draw opacity=1 ] (465,87.5) .. controls (465,86.12) and (466.12,85) .. (467.5,85) .. controls (468.88,85) and (470,86.12) .. (470,87.5) .. controls (470,88.88) and (468.88,90) .. (467.5,90) .. controls (466.12,90) and (465,88.88) .. (465,87.5) -- cycle ;
\draw  [color={rgb, 255:red, 0; green, 0; blue, 0 }  ,draw opacity=1 ] (370,117.5) .. controls (370,116.12) and (371.12,115) .. (372.5,115) .. controls (373.88,115) and (375,116.12) .. (375,117.5) .. controls (375,118.88) and (373.88,120) .. (372.5,120) .. controls (371.12,120) and (370,118.88) .. (370,117.5) -- cycle ;
\draw  [color={rgb, 255:red, 208; green, 2; blue, 27 }  ,draw opacity=1 ][fill={rgb, 255:red, 208; green, 2; blue, 27 }  ,fill opacity=1 ] (365,117.5) .. controls (365,116.12) and (366.12,115) .. (367.5,115) .. controls (368.88,115) and (370,116.12) .. (370,117.5) .. controls (370,118.88) and (368.88,120) .. (367.5,120) .. controls (366.12,120) and (365,118.88) .. (365,117.5) -- cycle ;
\draw  [color={rgb, 255:red, 0; green, 0; blue, 0 }  ,draw opacity=1 ] (450,117.5) .. controls (450,116.12) and (451.12,115) .. (452.5,115) .. controls (453.88,115) and (455,116.12) .. (455,117.5) .. controls (455,118.88) and (453.88,120) .. (452.5,120) .. controls (451.12,120) and (450,118.88) .. (450,117.5) -- cycle ;
\draw  [color={rgb, 255:red, 0; green, 0; blue, 0 }  ,draw opacity=1 ] (445,117.5) .. controls (445,116.12) and (446.12,115) .. (447.5,115) .. controls (448.88,115) and (450,116.12) .. (450,117.5) .. controls (450,118.88) and (448.88,120) .. (447.5,120) .. controls (446.12,120) and (445,118.88) .. (445,117.5) -- cycle ;
\draw  [color={rgb, 255:red, 208; green, 2; blue, 27 }  ,draw opacity=1 ][fill={rgb, 255:red, 208; green, 2; blue, 27 }  ,fill opacity=1 ] (410,147.5) .. controls (410,146.12) and (411.12,145) .. (412.5,145) .. controls (413.88,145) and (415,146.12) .. (415,147.5) .. controls (415,148.88) and (413.88,150) .. (412.5,150) .. controls (411.12,150) and (410,148.88) .. (410,147.5) -- cycle ;
\draw  [color={rgb, 255:red, 0; green, 0; blue, 0 }  ,draw opacity=1 ] (405,147.5) .. controls (405,146.12) and (406.12,145) .. (407.5,145) .. controls (408.88,145) and (410,146.12) .. (410,147.5) .. controls (410,148.88) and (408.88,150) .. (407.5,150) .. controls (406.12,150) and (405,148.88) .. (405,147.5) -- cycle ;
\draw [color={rgb, 255:red, 0; green, 0; blue, 0 }  ,draw opacity=1 ]   (510,120) -- (530,150) ;
\draw [color={rgb, 255:red, 0; green, 0; blue, 0 }  ,draw opacity=1 ]   (530,150) -- (550,120) ;
\draw [color={rgb, 255:red, 0; green, 0; blue, 0 }  ,draw opacity=1 ]   (590,120) -- (610,150) ;
\draw [color={rgb, 255:red, 0; green, 0; blue, 0 }  ,draw opacity=1 ]   (610,150) -- (630,120) ;
\draw [color={rgb, 255:red, 0; green, 0; blue, 0 }  ,draw opacity=1 ]   (530,150) -- (570,180) ;
\draw [color={rgb, 255:red, 0; green, 0; blue, 0 }  ,draw opacity=1 ]   (570,180) -- (610,150) ;
\draw  [color={rgb, 255:red, 0; green, 0; blue, 0 }  ,draw opacity=1 ] (510,117.5) .. controls (510,116.12) and (511.12,115) .. (512.5,115) .. controls (513.88,115) and (515,116.12) .. (515,117.5) .. controls (515,118.88) and (513.88,120) .. (512.5,120) .. controls (511.12,120) and (510,118.88) .. (510,117.5) -- cycle ;
\draw  [color={rgb, 255:red, 0; green, 0; blue, 0 }  ,draw opacity=1 ] (505,117.5) .. controls (505,116.12) and (506.12,115) .. (507.5,115) .. controls (508.88,115) and (510,116.12) .. (510,117.5) .. controls (510,118.88) and (508.88,120) .. (507.5,120) .. controls (506.12,120) and (505,118.88) .. (505,117.5) -- cycle ;
\draw  [color={rgb, 255:red, 0; green, 0; blue, 0 }  ,draw opacity=1 ] (550,117.5) .. controls (550,116.12) and (551.12,115) .. (552.5,115) .. controls (553.88,115) and (555,116.12) .. (555,117.5) .. controls (555,118.88) and (553.88,120) .. (552.5,120) .. controls (551.12,120) and (550,118.88) .. (550,117.5) -- cycle ;
\draw  [color={rgb, 255:red, 0; green, 0; blue, 0 }  ,draw opacity=1 ] (545,117.5) .. controls (545,116.12) and (546.12,115) .. (547.5,115) .. controls (548.88,115) and (550,116.12) .. (550,117.5) .. controls (550,118.88) and (548.88,120) .. (547.5,120) .. controls (546.12,120) and (545,118.88) .. (545,117.5) -- cycle ;
\draw  [color={rgb, 255:red, 0; green, 0; blue, 0 }  ,draw opacity=1 ] (590,117.5) .. controls (590,116.12) and (591.12,115) .. (592.5,115) .. controls (593.88,115) and (595,116.12) .. (595,117.5) .. controls (595,118.88) and (593.88,120) .. (592.5,120) .. controls (591.12,120) and (590,118.88) .. (590,117.5) -- cycle ;
\draw  [color={rgb, 255:red, 0; green, 0; blue, 0 }  ,draw opacity=1 ] (585,117.5) .. controls (585,116.12) and (586.12,115) .. (587.5,115) .. controls (588.88,115) and (590,116.12) .. (590,117.5) .. controls (590,118.88) and (588.88,120) .. (587.5,120) .. controls (586.12,120) and (585,118.88) .. (585,117.5) -- cycle ;

\draw  [color={rgb, 255:red, 0; green, 0; blue, 0 }  ,draw opacity=1 ] (630,117.5) .. controls (630,116.12) and (631.12,115) .. (632.5,115) .. controls (633.88,115) and (635,116.12) .. (635,117.5) .. controls (635,118.88) and (633.88,120) .. (632.5,120) .. controls (631.12,120) and (630,118.88) .. (630,117.5) -- cycle ;
\draw  [color={rgb, 255:red, 0; green, 0; blue, 0 }  ,draw opacity=1 ] (625,117.5) .. controls (625,116.12) and (626.12,115) .. (627.5,115) .. controls (628.88,115) and (630,116.12) .. (630,117.5) .. controls (630,118.88) and (628.88,120) .. (627.5,120) .. controls (626.12,120) and (625,118.88) .. (625,117.5) -- cycle ;
\draw  [color={rgb, 255:red, 0; green, 0; blue, 226 }  ,draw opacity=1 ][fill={rgb, 255:red, 0; green, 0; blue, 226 }  ,fill opacity=1 ] (530,147.5) .. controls (530,146.12) and (531.12,145) .. (532.5,145) .. controls (533.88,145) and (535,146.12) .. (535,147.5) .. controls (535,148.88) and (533.88,150) .. (532.5,150) .. controls (531.12,150) and (530,148.88) .. (530,147.5) -- cycle ;
\draw  [color={rgb, 255:red, 0; green, 0; blue, 0 }  ,draw opacity=1 ] (525,147.5) .. controls (525,146.12) and (526.12,145) .. (527.5,145) .. controls (528.88,145) and (530,146.12) .. (530,147.5) .. controls (530,148.88) and (528.88,150) .. (527.5,150) .. controls (526.12,150) and (525,148.88) .. (525,147.5) -- cycle ;
\draw  [color={rgb, 255:red, 0; green, 0; blue, 0 }  ,draw opacity=1 ] (610,147.5) .. controls (610,146.12) and (611.12,145) .. (612.5,145) .. controls (613.88,145) and (615,146.12) .. (615,147.5) .. controls (615,148.88) and (613.88,150) .. (612.5,150) .. controls (611.12,150) and (610,148.88) .. (610,147.5) -- cycle ;
\draw  [color={rgb, 255:red, 0; green, 0; blue, 0 }  ,draw opacity=1 ] (605,147.5) .. controls (605,146.12) and (606.12,145) .. (607.5,145) .. controls (608.88,145) and (610,146.12) .. (610,147.5) .. controls (610,148.88) and (608.88,150) .. (607.5,150) .. controls (606.12,150) and (605,148.88) .. (605,147.5) -- cycle ;
\draw  [color={rgb, 255:red, 0; green, 0; blue, 0 }  ,draw opacity=1 ] (570,177.5) .. controls (570,176.12) and (571.12,175) .. (572.5,175) .. controls (573.88,175) and (575,176.12) .. (575,177.5) .. controls (575,178.88) and (573.88,180) .. (572.5,180) .. controls (571.12,180) and (570,178.88) .. (570,177.5) -- cycle ;
\draw  [color={rgb, 255:red, 0; green, 0; blue, 226 }  ,draw opacity=1 ][fill={rgb, 255:red, 0; green, 0; blue, 226 }  ,fill opacity=1 ] (565,177.5) .. controls (565,176.12) and (566.12,175) .. (567.5,175) .. controls (568.88,175) and (570,176.12) .. (570,177.5) .. controls (570,178.88) and (568.88,180) .. (567.5,180) .. controls (566.12,180) and (565,178.88) .. (565,177.5) -- cycle ;
\draw  [color={rgb, 255:red, 65; green, 117; blue, 5 }  ,draw opacity=1 ][fill={rgb, 255:red, 65; green, 117; blue, 5 }  ,fill opacity=1 ] (170,147.5) .. controls (170,146.12) and (171.12,145) .. (172.5,145) .. controls (173.88,145) and (175,146.12) .. (175,147.5) .. controls (175,148.88) and (173.88,150) .. (172.5,150) .. controls (171.12,150) and (170,148.88) .. (170,147.5) -- cycle ;
\draw  [color={rgb, 255:red, 0; green, 0; blue, 0 }  ,draw opacity=1 ] (165,147.5) .. controls (165,146.12) and (166.12,145) .. (167.5,145) .. controls (168.88,145) and (170,146.12) .. (170,147.5) .. controls (170,148.88) and (168.88,150) .. (167.5,150) .. controls (166.12,150) and (165,148.88) .. (165,147.5) -- cycle ;
\draw  [color={rgb, 255:red, 0; green, 0; blue, 0 }  ,draw opacity=1 ] (290,177.5) .. controls (290,176.12) and (291.12,175) .. (292.5,175) .. controls (293.88,175) and (295,176.12) .. (295,177.5) .. controls (295,178.88) and (293.88,180) .. (292.5,180) .. controls (291.12,180) and (290,178.88) .. (290,177.5) -- cycle ;
\draw  [color={rgb, 255:red, 65; green, 117; blue, 5 }  ,draw opacity=1 ][fill={rgb, 255:red, 65; green, 117; blue, 5 }  ,fill opacity=1 ] (285,177.5) .. controls (285,176.12) and (286.12,175) .. (287.5,175) .. controls (288.88,175) and (290,176.12) .. (290,177.5) .. controls (290,178.88) and (288.88,180) .. (287.5,180) .. controls (286.12,180) and (285,178.88) .. (285,177.5) -- cycle ;
\draw  [color={rgb, 255:red, 65; green, 117; blue, 5 }  ,draw opacity=1 ][fill={rgb, 255:red, 65; green, 117; blue, 5 }  ,fill opacity=1 ] (430,207.5) .. controls (430,206.12) and (431.12,205) .. (432.5,205) .. controls (433.88,205) and (435,206.12) .. (435,207.5) .. controls (435,208.88) and (433.88,210) .. (432.5,210) .. controls (431.12,210) and (430,208.88) .. (430,207.5) -- cycle ;
\draw  [color={rgb, 255:red, 0; green, 0; blue, 0 }  ,draw opacity=1 ] (425,207.5) .. controls (425,206.12) and (426.12,205) .. (427.5,205) .. controls (428.88,205) and (430,206.12) .. (430,207.5) .. controls (430,208.88) and (428.88,210) .. (427.5,210) .. controls (426.12,210) and (425,208.88) .. (425,207.5) -- cycle ;
\draw  [color={rgb, 255:red, 0; green, 0; blue, 0 }  ,draw opacity=1 ] (575,40) .. controls (575,38.62) and (576.12,37.5) .. (577.5,37.5) .. controls (578.88,37.5) and (580,38.62) .. (580,40) .. controls (580,41.38) and (578.88,42.5) .. (577.5,42.5) .. controls (576.12,42.5) and (575,41.38) .. (575,40) -- cycle ;
\draw  [color={rgb, 255:red, 0; green, 0; blue, 0 }  ,draw opacity=1 ] (570,40) .. controls (570,38.62) and (571.12,37.5) .. (572.5,37.5) .. controls (573.88,37.5) and (575,38.62) .. (575,40) .. controls (575,41.38) and (573.88,42.5) .. (572.5,42.5) .. controls (571.12,42.5) and (570,41.38) .. (570,40) -- cycle ;

\draw [color={rgb, 255:red, 0; green, 0; blue, 0 }  ,draw opacity=1 ]   (540,40) -- (563,40) ;
\draw [shift={(565,40)}, rotate = 180] [color={rgb, 255:red, 0; green, 0; blue, 0 }  ,draw opacity=1 ][line width=0.75]    (10.93,-3.29) .. controls (6.95,-1.4) and (3.31,-0.3) .. (0,0) .. controls (3.31,0.3) and (6.95,1.4) .. (10.93,3.29)   ;
\draw [color={rgb, 255:red, 0; green, 0; blue, 0 }  ,draw opacity=1 ]   (610,40) -- (587,40) ;
\draw [shift={(585,40)}, rotate = 360] [color={rgb, 255:red, 0; green, 0; blue, 0 }  ,draw opacity=1 ][line width=0.75]    (10.93,-3.29) .. controls (6.95,-1.4) and (3.31,-0.3) .. (0,0) .. controls (3.31,0.3) and (6.95,1.4) .. (10.93,3.29)   ;
\draw   (510,20) -- (640,20) -- (640,60) -- (510,60) -- cycle ;
\draw  [dash pattern={on 4.5pt off 4.5pt}]  (490,230) -- (550,250) ;
\draw    (430,210) -- (490,230) ;
\draw   (580,110) -- (600,110) -- (600,125) -- (580,125) -- cycle ;
\draw  [dash pattern={on 0.84pt off 2.51pt}]  (510,60) -- (580,110) ;
\draw  [dash pattern={on 0.84pt off 2.51pt}]  (600,110) -- (640,60) ;

\draw (539,39.5) node [anchor=east] [inner sep=0.75pt]    {$d'_{\eta }$};
\draw (611,39.5) node [anchor=west] [inner sep=0.75pt]    {$e'_{\eta }$};
\draw (28.94,49.31) node [anchor=south east] [inner sep=0.75pt]  [color={rgb, 255:red, 65; green, 117; blue, 5 }  ,opacity=1 ]  {$g_{0}$};
\draw (61.06,88.6) node [anchor=south west] [inner sep=0.75pt]  [color={rgb, 255:red, 65; green, 117; blue, 5 }  ,opacity=1 ]  {$h_{0}$};
\draw (79,121.4) node [anchor=north east] [inner sep=0.75pt]  [color={rgb, 255:red, 65; green, 117; blue, 5 }  ,opacity=1 ]  {$g_{1}$};
\draw (171,137.6) node [anchor=south west] [inner sep=0.75pt]  [color={rgb, 255:red, 65; green, 117; blue, 5 }  ,opacity=1 ]  {$h_{1}$};
\draw (279,181.4) node [anchor=north east] [inner sep=0.75pt]  [color={rgb, 255:red, 65; green, 117; blue, 5 }  ,opacity=1 ]  {$g_{2}$};
\draw (441,199.31) node [anchor=south west] [inner sep=0.75pt]  [color={rgb, 255:red, 65; green, 117; blue, 5 }  ,opacity=1 ]  {$h_{2}$};
\draw (109,48.6) node [anchor=south east] [inner sep=0.75pt]  [color={rgb, 255:red, 208; green, 2; blue, 27 }  ,opacity=1 ]  {$g'_{0}$};
\draw (141.09,89.14) node [anchor=west] [inner sep=0.75pt]  [color={rgb, 255:red, 208; green, 2; blue, 27 }  ,opacity=1 ]  {$h'_{0}$};
\draw (359,119.5) node [anchor=east] [inner sep=0.75pt]  [color={rgb, 255:red, 208; green, 2; blue, 27 }  ,opacity=1 ]  {$g'_{1}$};
\draw (421,151.04) node [anchor=north west][inner sep=0.75pt]  [color={rgb, 255:red, 208; green, 2; blue, 27 }  ,opacity=1 ]  {$h'_{1}$};
\draw (221,88.96) node [anchor=south west] [inner sep=0.75pt]  [color={rgb, 255:red, 0; green, 0; blue, 226 }  ,opacity=1 ]  {$h''_{0}$};
\draw (243.63,119.14) node [anchor=east] [inner sep=0.75pt]  [color={rgb, 255:red, 0; green, 0; blue, 226 }  ,opacity=1 ]  {$g''_{1}$};
\draw (541,147.6) node [anchor=south west] [inner sep=0.75pt]  [color={rgb, 255:red, 0; green, 0; blue, 226 }  ,opacity=1 ]  {$h''_{1}$};
\draw (563,181.04) node [anchor=north east] [inner sep=0.75pt]  [color={rgb, 255:red, 0; green, 0; blue, 226 }  ,opacity=1 ]  {$g''_{2}$};

\end{tikzpicture}

%% file: tex/transitivity.tex
\section{Transitivity}
\label{sec:transitivity}
\begin{lemma}
\thlabel{lem:transitivity-find-cr-morley-sequence}
If $M \subseteq N$ are e.c.\ models of a thick NSOP$_1$ theory, $a\ind^K_M N$, and $\mu$ is a small cardinal, then there is a parallel-Morley in $\tp(a/N)$ sequence $(a_i)_{i\in \mu}$  with $a_0=a$ such that $a_i\ind^K_M Na_{<i}$ for every $i<\mu$.
\end{lemma}
\begin{proof}
 Put $\lambda = |Na| + \aleph_0$ and  (using \thref{lem:(*)}) choose  a global $N$-Ls-invariant extension $q$ of $\Lstp(a/N)$ satisfying $(*)_\lambda$.
 
By \thref{lem:kim-dividing-type-def}, compactness, finite character of Kim-independence, and an automorphism, it is enough to find for any given $k<\omega$ a parallel-Morley sequence $(a_i)_{i< k}$ in $q$ over $N$ such that $a_i\ind^K_M Na_{<i}$ for every $i<k$. 

So fix any $k<\omega$. By backward induction on $k'=k+1,k,\dots,1$ we will construct trees $(c_\eta)_{\eta\in S_{k'}}$, where $S_{k'}:=\{\xi\in {\omega}^{\leq k}:0^{k'-1}\trianglelefteq \xi\}$, such that for each $k'$ the tree $(c_\eta)_{\eta\in S_{k'}}$ satisfies the following conditions:
\begin{enumerate}[label=(A\arabic*)$_{k'}$]
\item For any $\eta\in S_{k'}$ we have $c_{\eta}\ind^K_M Nc_{ \triangleright \eta}$ and $c_\eta\equivls_N a$;
\item $(c_\eta)_{\eta\in S_{k'}}$ is $q$-spread-out over $N$.
\end{enumerate}

For $k'=k+1$ we let $c_{0^{k}}=a$.
For the inductive step, suppose we are done for some $k'$. 
By (A1)$_{k'}$ we have $c_{0^{k'-1}}\equivls_Na$, so by $(*)_\lambda$ there is a global $N$-Ls-invariant type $r(x,y)\supseteq  q(x)$ extending $\Lstp(c_{0^{k'-1}},(c_{\eta})_{\eta\in S^*_{k'}}/N)$ where $x$ corresponds to $c_{0^{k'-1}}$. Choose a Morley sequence $I:=((c_{\eta,i})_{\eta \in S_{k'}})_{i<{\omega}}$  in $r(x,y)$ over $N$ with $c_{\eta,0}=c_\eta$ for each $\eta\in S_{k'}$.
By extension (\thref{cor:extension}) there is $c'\equivls_N a$ with $c'\ind^K_M NI$.
Put $c_{0^{k'-2}\smallfrown i\smallfrown \zeta}:=c_{0^{k'-1}\smallfrown \zeta,i}$ for all $i<\omega$, $\zeta\in {\omega}^{\leq k+1-k'}  $, and $c_{0^{k'-2}}:=c'$. Then  (A2)$_{k'-1}$ follows by \thref{lem:q-spread-out-trees}(i), whereas (A1)$_{k'-1}$ with $\eta\in S_{k'-1}^*$ follows by invariance of Kim-independence, and (A1)$_{k'-1}$ with $\eta=0^{k'-2}$ follows by the choice of $c_{0^{k'-2}}=c'$. Thus the inductive step, and hence the construction of the tree $(c_{\eta})_{\eta\in {\omega}^{\leq k}}=(c_{\eta})_{\eta\in S_1}$, is completed.

Letting $(c'_{\eta})_{\eta\in {\omega}^{\leq k}}$ be an s-indiscernible over $N$ tree which is $EM_s$-based on $(c_{\eta})_{\eta\in {\omega}^{\leq k}}$ over $Na$, we get that $(c'_{\eta})_{\eta\in {\omega}^{\leq k}}$ satisfies (A1)$_{1}$ by \thref{lem:kim-dividing-type-def} and \thref{cor:g-compact}(ii), and is weakly $q$-spread-out over $N$ by \thref{lem:q-spread-out-trees}(ii). Thus, by \thref{lem:q-spread-out-trees}(iii), putting $a_i=c'_{0^{k-i}}$ for $i<k$ we get a parallel-Morley sequence $(a_i)_{i< k}$ in $q$ over $N$ satisfying the requirements.
\end{proof}
\begin{lemma}
\thlabel{lem:transitivity-technical-lemma}
Suppose $T$ is thick NSOP$_1$ and $M \subseteq N$ are e.c.\ models of $T$. If $a\ind^K_M N$ and $c\ind^K_M N$ then there is $c'\equivls_N c$ such that $ac'\ind^K_M N$ and $a\ind^K_N c'$.
\end{lemma}
\begin{proof}
By \thref{lem:kim-dividing-type-def} 
there is a type $\Gamma(x;N,a)$ equivalent to the condition $ax\ind^K_M N$.  
By \thref{lem:transitivity-find-cr-morley-sequence} there is a parallel-Morley in  $\tp(a/N)$ sequence $(a_i)_{i< \lambda_T}$   with $a_0=a$ such that $a_i\ind^K_M Na_{<i}$ for every $i<\lambda_T$. Replacing $(a_i)_{i< \lambda_T}$ with an $N$-indiscernible sequence based on it over $N$ and moving by an automorphism (to keep $a_0=a$), we may assume $(a_i)_{i< \lambda_T}$ is $N$-indiscernible.
 \begin{claim}
 $\bigcup_{i<\lambda_T} \Gamma(x;N,a_i)$ has a realisation $c''$ such that $c''\equivls_N c$.
 \end{claim}
\begin{proof}
By induction on $n<\omega$ we will find $c_n\equivls_N c$ such that $c_n\ind^K_M Na_{<n}$ and $c_n\models\bigcup_{i<n}\Gamma(x;N,a_i)$, which is enough by compactness, $N$-indiscernibility of $(a_i)_{i<\lambda_T}$ and \thref{cor:g-compact}(ii). For $n=0$ put $c_0=c$. Assume we have found $c_n$ and find by extension (\thref{cor:extension}) some $c' \equivls_M c$ be such that $c'\ind^K_M a_n$. By \thref{thm:strong-independence-theorem} there is $c_{n+1}$ with $c_{n+1}a_{<n}\equivls_{N}c_na_{<n}$, $c_{n+1}a_n \equivls_{M}c'a_n$, $c_{n+1}\ind^K_M Na_{<n+1}$ and $a_nc_{n+1} \ind^K_M Na_{<n}$. In particular $c_{n+1} \equivls_N c_{n}\equivls_N c$ and $c_{n+1}\models \bigcup_{i<n+1} \Gamma(x;N,a_i)$.
\end{proof}
Let $c''$ be given by the claim, and let $(a'_i)_{i<\omega}$ be an $Nc''$-indiscernible sequence based on $(a_i)_{i<\lambda_T}$ over $Nc''a$. Then $a'_0 \equivls_N a$ (as $a_i \equivls_N a$ for every $i<\lambda_T$) so there is a Lascar strong automorphism $f$ over $N$ sending $a'_0$ to $a=a_0$. Put $c':=f(c'')$. Then $(f(a'_i))_{i<\omega}$ is an $Nc'$-indiscernible parallel-Morley sequence in $\tp(a/N)$ starting with $a$, so $c'\ind^K_N a$ by \thref{cor:cr-kims-lemma}. Also, $c'\models \Gamma(x;N,a)$, so $ac'\ind^K_M N$ by the choice of $\Gamma$, and we are done.
\end{proof}
\begin{lemma}
\thlabel{lem:find-lifted-morley-sequence}
Suppose $T$ is thick NSOP$_1$ with e.c.\ models $M \subseteq N$ and  $a\ind_M^K N$. Then there is a $\ind_N^K$-Morley parallel-Morley in $\tp(a/M)$ sequence $(a_i)_{i<\omega}$  with $a=a_0$.
\end{lemma}
\begin{proof}
By extension (\thref{cor:extension}) we may assume that $a$ is a $\lambda_T$-saturated model extending $M$. By \thref{lem:(*)} there is  a global $M$-Ls-invariant extension $q(x)\supseteq \tp(a/M)$ satisfying the property $(*)_{\lambda}$ with $\lambda = |a| + \aleph_0$. We claim that it is enough to find for any given $k<\omega$ a parallel-Morley sequence $(a_i)_{i< k}$ in $q$ over $M$ such that $a_i\ind^K_N a_{<i}$ and $a_i \equiv_N a$ for every $i<k$: indeed, if we prove this, then, since the condition $(a_i \equiv_N a) \wedge (a_i \ind^K_N a_{<i})$ is type-definable by \thref{lem:kim-dividing-type-def}, we can find by compactness such a sequence of length $\lambda_{|T|+|Na|}$. Then taking an $N$-indiscernible sequence based on $(a_i)_{i<\lambda_{|T|+|Na|}}$ over $N$ and moving it by an automorphism we obtain a desired sequence.

So fix any $k<\omega$. 
By backward induction on $k'=k+1,k,\dots,1$ we will define trees $(c_\eta)_{\eta\in S_{k'}}$, where $S_{k'}:=\{\xi\in {\omega}^{\leq k}:0^{k'-1}\trianglelefteq \xi\}$, such that for each $k'$ the tree $(c_\eta)_{\eta\in S_{k'}}$ satisfies the following conditions:
\begin{enumerate}[label=(A\arabic*)$_{k'}$]
\item For any $\eta\in S_{k'}$ we have $c_{\eta}\ind^K_N c_{ \triangleright \eta}$ and $c_\eta \equivls_N a$;

\item $(c_\eta)_{\eta\in S_{k'}}$ is $q$-spread-out over $M$;

\item $(c_\eta)_{\eta\in S_{k'}}\ind^K_M N$.
\end{enumerate}

For $k'=k+1$ we let $c_{0^{k}}=a$.
For the inductive step, suppose we are done for some $k'$. By $(*)_\lambda$ and (A1)$_{k'}$ there is a global $M$-invariant type $r(x,y)$ extending $\Lstp(c_{0^{k'-1}},(c_\eta)_{\eta\in S^*_{k'}}/M)$ and $q(x)$. 
As $c_\eta$'s are $\lambda_T$-saturated models, we get by (A3)$_{k'}$ and \thref{cor:q-ls-dividing-is-q-dividing} that $\Lstp(N/(c_\eta)_{\eta\in S_{k'}})$ does not $r$-Ls-divide over $M$.  Thus there is an $N$-indiscernible Morley sequence $I=((c_{\eta,i})_{\eta\in S_{k'}})_{i<\omega}$ in $r(x,y)$ over $M$  with $c_{\eta,0}=c_\eta$ for each $\eta\in S_{k'}$ and $I\ind^K_M N$. By \thref{lem:transitivity-technical-lemma} there is $a'\equivls_N a$ such that $a'\ind^K_N I$ and $a'I\ind^K_M N$. Put $c_{0^{k'-2}\smallfrown i\smallfrown \zeta}:=c_{0^{k'-1}\smallfrown \zeta,i}$ for all $i<\omega$, $\zeta\in {\omega}^{\leq k+1-k'}  $, and $c_{0^{k'-2}}:=a'$. Then we get (A2)$_{k'-1}$ by \thref{lem:q-spread-out-trees}(i), (A1)$_{k'-1}$, using that $a'\ind^K_N I$, and (A3)$_{k'-1}$ holds as $a'I\ind^K_M N$. 
Thus the inductive step, and hence the construction of the tree $(c_{\eta})_{\eta\in {\omega}^{\leq k}}=(c_{\eta})_{\eta\in S_1}$, is completed.

Letting $(c'_{\eta})_{\eta\in {\omega}^{\leq k}}$ be an s-indiscernible over $N$ tree which is $EM_s$-based on $(c_{\eta})_{\eta\in {\omega}^{\leq k}}$ over $Na$, we get that $(c'_{\eta})_{\eta\in {\omega}^{\leq k}}$ is weakly $q$-spread-out over $M$ by \thref{lem:q-spread-out-trees}(ii) and satisfies (A1)$_{1}$ by \thref{lem:kim-dividing-type-def} and \thref{cor:g-compact}(ii).
Thus putting $a_i=c'_{0^{k-i}}$ for $i<k$ we get by \thref{lem:q-spread-out-trees}(iii) a parallel-Morley sequence $(a_i)_{i< k}$ in $q$ over $M$ satisfying the requirements.
\end{proof}
\begin{theorem}[Transitivity]
\thlabel{thm:transitivity}
 Suppose $T$ is thick NSOP$_1$ with models $M \subseteq N$. If $a\ind^K_M N$ and $a\ind^K_N c$, then $a\ind^K_M Nc$.
\end{theorem}
\begin{proof}
By \thref{lem:find-lifted-morley-sequence} there is a  $\ind_N^K$-Morley parallel-Morley in $\tp(a/M)$ sequence $I=(a_i)_{i<\omega}$  with $a_0=a$. As $a\ind^K_N c$, we get by \thref{lem:chain-lemma-kim-morley-sequences} an $Nc$-indiscernible sequence $I'=(a'_i)_{i<\omega}\equiv_{Na}I$. As $I'$ is also parallel-Morley in $\tp(a/M)$ and $a'_0=a$, we get by \thref{cor:cr-kims-lemma} that $Nc\ind^K_M a$, so by symmetry we are done.
\end{proof}

%% file: tex/kim-pillay.tex
\section{Kim-Pillay style theorem}
\label{sec:kim-pillay-style-theorem}
\begin{theorem}
\thlabel{thm:kim-pillay-style}
Let $T$ be a thick positive theory. Then $T$ is NSOP$_1$ if and only if there is an automorphism invariant ternary relation $\ind$ on small subsets of the monster model, only allowing e.c.\ models in the base, satisfying the following properties:
\begin{description}
\item[\textsc{Finite Character}] if $a \ind_M b_0$ for all finite $b_0 \subseteq b$ then $a \ind_M b$.
\item[\textsc{Existence}] $a \ind_M M$ for any model $M$.
\item[\textsc{Monotonicity}] $a a' \ind_M b b'$ implies $a \ind_M b$.
\item[\textsc{Symmetry}] $a \ind_M b$ implies $b \ind_M a$.
\item[\textsc{Local Character}] let $a$ be a finite tuple and $\kappa > |T|$ be regular then for every continuous chain $(M_i)_{i < \kappa}$, with $|M_i| < \kappa$ for all $i$, there is $i < \kappa$ such that $a \ind_{M_i} M$, where $M = \bigcup_{i < \kappa} M_i$.
\item[\textsc{Independence Theorem}] if $a \ind_M b$, $a' \ind_M c$ and $b \ind_M c$ with $a \equivls_M a'$ then there is $a''$ such that $a''b \equivls_M ab$, $a''c \equivls_M a'c$ and $a'' \ind_M bc$.
\item[\textsc{Extension}] if $a \ind_M b$ then for any $c$ there is $a' \equiv_{Mb} a$ such that $a' \ind_M bc$.
\item[\textsc{Transitivity}] if $a \ind_M N$ and $a \ind_N b$ with $M \subseteq N$ then $a \ind_M Nb$.
\end{description}
Furthermore, in this case $\ind = \ind^K$.
\end{theorem}
The properties in \thref{thm:kim-pillay-style} are not as strong as they could be. For example, we actually proved the strong independence theorem for $\ind^K$, see \thref{thm:strong-independence-theorem}. The slightly simpler formulation of the properties in \thref{thm:kim-pillay-style} is easier to verify for an arbitrary independence relation $\ind$. Then it follows immediately from $\ind = \ind^K$ that such an independence relation $\ind$ also satisfies the stronger formulations.
\begin{remark}
\thlabel{rem:strong-finite-character}
In the existing Kim-Pillay style theorems for full first-order logic \cite[Theorem 9.1]{kaplan_kim-independence_2020}, \cite[Theorem 6.11]{kaplan_transitivity_2019} and \cite[Theorem 5.1]{chernikov_transitivity_2020} there are still various properties that mention syntax. Our \thref{thm:kim-pillay-style} is completely syntax-free. One syntax-dependent property is mentioned in all of the above theorems, and is called \textsc{Strong Finite Character}: if $a \nind_M b$ then there is $\phi(x, b, m) \in \tp(a / Mb)$ such that for any $a' \models \phi(x, b, m)$ we have $a' \nind_M b$.

We could replace \textsc{Finite Character} and \textsc{Local Character} in \thref{thm:kim-pillay-style} by \textsc{Strong Finite Character}. Obviously \textsc{Strong Finite Character} implies \textsc{Finite Character} and modulo the other properties it also implies \textsc{Local Character} by \thref{lem:coheir-indep-implies-abstract-indep} and \thref{lem:abstract-indep-local-character}.
\end{remark}
\begin{remark}
\thlabel{rem:less-assumptions-for-nsop1}
To conclude that a theory is NSOP$_1$ it is enough to find an independence relation with the properties \textsc{Strong Finite Character}, \textsc{Existence}, \textsc{Monotonicity}, \textsc{Symmetry} and \textsc{Independence Theorem}, see \cite[Theorem 6.4]{haykazyan_existentially_2021}. However, that does not guarantee that the independence relation is also Kim-independence, see \cite[Remark 9.39]{kaplan_kim-independence_2020} for an example (already in full first-order logic). We also point out that \cite[Theorem 6.4]{haykazyan_existentially_2021} says nothing about the properties that Kim-independence generally has in NSOP$_1$ theories. Finally, our proof is also different because we do not rely on the syntactic property \textsc{Strong Finite Character}.
\end{remark}
\begin{remark}
\thlabel{rem:kim-pillay-thm-independence-theorem}
We point out a minor difference between \thref{thm:independence-theorem} and \textsc{Independence Theorem} in \thref{thm:kim-pillay-style}. In the former we get $a'' \equivls_{Mb} a$, which is generally stronger than the $a''b \equivls_M ab$ in the latter (and similar for $c$). Again, the reason is that the latter is easier to verify. Definitely in semi-Hausdorff theories, because then $a''b \equivls_M ab$ is equivalent to $a''b \equiv_M ab$, so we do not have to worry about Lascar strong types. For a concrete example of this, see \thref{fact:ecef-facts}(i). The only place where \textsc{Independence Theorem} is used, namely to get consistency along a certain sequence, we only need this weaker version.
\end{remark}
\begin{lemma}
\thlabel{lem:coheir-indep-implies-abstract-indep}
Let $\ind$ satisfy \textsc{Strong Finite Character}, \textsc{Existence}, \textsc{Monotonicity} and \textsc{Symmetry}. Then $a \ind_M^u b$ implies $a \ind_M b$.
\end{lemma}
\begin{proof}
Exactly as in \cite[Proposition 5.8]{chernikov_model-theoretic_2015}.
\end{proof}
\begin{lemma}
\thlabel{lem:abstract-indep-local-character}
Let $\ind$ be as in \thref{lem:coheir-indep-implies-abstract-indep}. Then it satisfies \textsc{Local Character}.
\end{lemma}
\begin{proof}
By \thref{lem:coheir-indep-implies-abstract-indep} the proof from \cite[Theorem 3.2]{kaplan_local_2017} applies. Our formulation of local character then follows.
\end{proof}
\begin{corollary}[Local Character]
\thlabel{cor:kim-independence-local-character}
In a thick NSOP$_1$ theory Kim-independence satisfies \textsc{Local Character}.
\end{corollary}
\begin{remark}
\thlabel{rem:kim-independence-local-character}
In \cite{kaplan_local_2017} there are also different formulations of \textsc{Local Character}, for example in terms of club sets of $[M]^{|T|}$. Since their arguments apply directly, these formulations also hold for Kim-independence in any thick NSOP$_1$ theory.
\end{remark}
The following definition is based on the notion of isi-dividing from \cite{kamsma_kim-pillay_2020}.
\begin{definition}
\thlabel{def:long-dividing}
We say that a type $p(x, b) = \tp(a/Cb)$ \emph{long divides} over $C$ if there is $\mu$ such that for every $\lambda \geq \mu$ there is a sequence $(b_i)_{i < \lambda}$ with $b_i \equiv_C b$ for all $i < \lambda$ such that for some $\kappa < \lambda$ and every $I \subseteq \lambda$ with $|I| \geq \kappa$ we have that $\bigcup_{i \in I} p(x, b_i)$ is inconsistent. We write $a \ind_C^{ld} b$ if $\tp(a/Cb)$ does not long divide over $C$.
\end{definition}
There is a close connection between long dividing and dividing. Even though we do not need this connection in our proofs, it is still interesting to explore it. Dividing implies long dividing. Given an indiscernible sequence that witnesses dividing of a type $p$, we can use compactness to make it as long as we wish. So we find arbitrarily long sequences where $p$ is inconsistent along any infinite subsequence, so $p$ long divides. The converse is not so clear to us.
\begin{question}
\thlabel{q:long-dividing}
Does long dividing imply dividing?
\end{question}
At least if we assume the existence of a proper class of Ramsey cardinals then the answer is positive. To see this, suppose that $p$ long divides and let $\lambda$ be a big enough Ramsey cardinal. Then there is some sequence $(b_i)_{i < \lambda}$ witnessing that $p$ long divides. Since we assumed $\lambda$ to be Ramsey there is a cofinal $I \subseteq \lambda$ such that $(b_i)_{i \in I}$ is indiscernible. By the definition of long dividing $\bigcup_{i \in I} p(x, b_i)$ is then inconsistent and so we conclude that $p$ divides.
\begin{lemma}
\thlabel{lem:ls-invariance-nondividing-implies-isi-nondividing}
We have that $a \ind_C^{iLs} b$ implies $a \ind_C^{ld} b$.
\end{lemma}
\begin{proof}
Let $p(x, y) = \tp(ab/C)$ and let $\lambda$ be any regular cardinal bigger than the number of Lascar strong types over $C$ (compatible with $b$). Let $(b_i)_{i < \lambda}$ be any sequence in $\tp(b/C)$. By choice of $\lambda$ there must be $I \subseteq \lambda$ such that $b_i \equivls_C b_j$ for all $i, j \in I$ and $|I| = \lambda$. Pick some $i_0 \in I$ and let $a'$ be such that $a'b_{i_0} \equiv_C ab$. By assumption $a \ind_C^{iLs} b$, so $a' \ind_C^{iLs} b_{i_0}$. Let $q \supseteq \tp(a'/Cb_{i_0})$ be a global $C$-Ls-invariant extension and let $\alpha \models q$. Then $\alpha b_i \equivls_C \alpha b_{i_0}$ for all $i \in I$, so $\bigcup \{ p(x, b_i) : i \in I \}$ is consistent.
\end{proof}
\begin{definition}
\thlabel{def:independent-sequence}
Let $\ind$ be some independence relation and let $(a_i)_{i < \kappa}$ be some sequence. Suppose furthermore that there is a continuous chain $(M_i)_{i < \kappa}$ of e.c.\ models, with $M \subseteq M_0$, such that $a_{<i} \subseteq M_i$ and $a_i \ind_M M_i$ for all $i < \kappa$. Then we call $(M_i)_{i < \kappa}$ an \emph{$\ind_M$-independence chain (for $(a_i)_{i < \kappa}$)}.
\end{definition}
\begin{remark}
\thlabel{rem:building-witnesses-of-independence}
Let $\ind$ be an independence relation satisfying \textsc{Existence} and \textsc{Extension}, let $a$ be any tuple and let $M$ be any model. Then as usual we can inductively build arbitrarily long sequences $(a_i)_{i < \kappa}$ together with an $\ind_M$-independence chain $(M_i)_{i < \kappa}$, such that $a \equiv_M a_i$ for all $i < \kappa$.
\end{remark}
The following is adapted from one half of the original Kim-Pillay theorem, and occurs in \cite[Theorem 1.1]{kamsma_kim-pillay_2020}. We just have to check that the use of base-monotonicity can be replaced with our more carefully formulated form of local character.
\begin{proposition}
\thlabel{prop:isi-nondividing-implies-abstract-independence}
Let $\ind$ be as in \thref{thm:kim-pillay-style} then $a \ind_M^{ld} b$ implies $a \ind_M b$.
\end{proposition}
Note that we will actually not need \textsc{Independence Theorem} here.
\begin{proof}
It follows directly from the definition of long dividing that $\ind^{ld}$ has monotonicity on the left side. So by \textsc{Finite Character} and \textsc{Symmetry} we may assume $a$ to be finite.

By \thref{rem:building-witnesses-of-independence} we find a sequence $(b_i)_{i < \kappa}$ with an $\ind_M$-independence chain $(M_i)_{i < \kappa}$ such that $b \equiv_M b_i$ for all $i < \kappa$. Picking the right $\kappa > (|T| + |M|)^+$ there must be $I \subseteq \kappa$ with order type $(|T| + |M|)^+$ such that $\bigcup_{i \in I} p(x, b_i)$ is consistent, where $p(x, y) = \tp(ab/M)$. Let $a'$ be a realisation of this set. By \textsc{Monotonicity} and downward L\"owenheim-Skolem we may assume that $(M_i)_{i \in I}$ is a continuous chain with $|M_i| \leq |T| + |M|$ for all $i \in I$. Then by \textsc{Local Character} there is $i_0 \in I$ such that $a' \ind_{M_{i_0}} M_I$, where $M_I = \bigcup_{i \in I} M_i$. By \textsc{Monotonicity} we have $a' \ind_{M_{i_0}} b_{i_0}$ and by construction we also have $b_{i_0} \ind_M M_{i_0}$. So by \textsc{Symmetry} and \textsc{Transitivity} we obtain $a' \ind_M b_{i_0}$. The result now follows because $a' b_{i_0} \equiv_M ab$.
\end{proof}
We note that in the above proof it is relevant that we work with long dividing instead of dividing. This is because the application of \textsc{Local Character} only really makes sense if the chain consists of e.c.\ models, as we only allow e.c.\ models in the base. At the same time we need those e.c.\ models to form an independence chain for the rest of the proof to work. If we would try to follow the same proof just for dividing then we would have to work with indiscernible sequences. Finding an indiscernible $\ind$-independent sequence is not an issue. This can be done as usual: we first build a very long $\ind$-independent sequence and then base an indiscernible sequence on it. This preserves being $\ind$-independent due to \textsc{Finite Character}, but it does not carry over the independence chain. In long dividing this is not an issue, because we work directly with the very long sequence we constructed. So any `decorations', such as the independence chain, are then at our disposal.

The following lemma and its proof are a weaker version of the chain condition for $\ind^K$-Morley sequences (\thref{lem:chain-lemma-kim-morley-sequences}) that works for long enough $\ind^K$-independent sequences.
\begin{lemma}
\thlabel{lem:kim-nondividing-implies-kim-isi-sequence-is-consistent}
Let $T$ be a thick NSOP$_1$ theory. Suppose that $a \ind_M^K b$. Let $(b_i)_{i < \kappa}$ be an $\ind_M^K$-independent sequence, where $\kappa$ is a regular cardinal larger than the number of Lascar strong types over $M$ (compatible with $b$) and where $b \equiv_M b_i$ for all $i < \kappa$. Then there is $I \subseteq \kappa$ with $|I| = \kappa$ such that $\bigcup_{i \in I} p(x, b_i)$ does not Kim-divide (and is thus consistent), where $p(x, y) = \tp(ab/M)$.
\end{lemma}
\begin{proof}
By the choice of $\kappa$ there is $I \subseteq \kappa$ with $|I| = \kappa$ such that $b_i \equivls_M b_j$ for all $i,j \in I$. We conclude by the generalised independence theorem (\thref{cor:g-compact}(iii)).


\end{proof}
\begin{proof}[Proof of \thref{thm:kim-pillay-style}]
We already proved that $\ind^K$ has all the listed properties if $T$ is NSOP$_1$. So now we assume that we have an abstract independence relation $\ind$ satisfying the listed properties and we prove that $\ind = \ind^K$ and that $T$ is NSOP$_1$.

\textbf{The direction $a \ind_M b \implies a \ind_M^K b$.} This proof is based on the proof of the same direction in \cite[Theorem 9.1]{kaplan_kim-independence_2020}. Let $p(x, b) = \tp(a/Mb)$ and let $q$ be any global $M$-Ls-invariant extension of $\tp(b/M)$. Then a Morley sequence $(b_i)_{i < \omega}$ in $q$ is a $\ind_M$-Morley sequence by \thref{lem:ls-invariance-nondividing-implies-isi-nondividing} and \thref{prop:isi-nondividing-implies-abstract-independence}. By the standard \textsc{Independence Theorem} argument we thus find that $\bigcup_{i < \omega} p(x, b_i)$ is consistent, hence $a \ind_M^K b$.

\textbf{The theory $T$ is NSOP$_1$.} We prove weak symmetry as in \thref{thm:symmetry-iff-nsop1}. So suppose that $a \ind_M^{iLs} b$. Then combining \thref{lem:ls-invariance-nondividing-implies-isi-nondividing} and \thref{prop:isi-nondividing-implies-abstract-independence} again we get $a \ind_M b$. So by \textsc{Symmetry} we have $b \ind_M a$ and then $b \ind_M^K a$ follows from the above.

\textbf{The direction $a \ind_M^K b \implies a \ind_M b$.} This proof is based on the proof of the same direction in \cite[Theorem 5.1]{chernikov_transitivity_2020}. By \thref{rem:building-witnesses-of-independence} we obtain a long enough sequence $(b_i)_{i < \kappa}$ with an $\ind_M$-independence chain $(M_i)_{i < \kappa}$ and $b_i \equiv_M b$ for all $i < \kappa$. By the above $(M_i)_{i < \kappa}$ is also an $\ind_M^K$-independence chain. So by \thref{lem:kim-nondividing-implies-kim-isi-sequence-is-consistent} there is $I \subseteq \kappa$ with order type $\kappa$ such that $\bigcup_{i \in I} p(x, b_i)$ is consistent, where $p(x, b) = \tp(a/Mb)$. Let $a'$ be a realisation of this set. By deleting an end-segment, \textsc{Monotonicity} and downward L\"owenheim-Skolem we may assume that $(M_i)_{i \in I}$ is a continuous chain with $|M_i| \leq |T| + |M|$ for all $i \in I$ and $I$ has order type $(|T| + |M|)^+$. By \textsc{Local Character} there is $i_0 \in I$ such that $a' \ind_{M_{i_0}} M_I$, where $M_I = \bigcup_{i \in I} M_i$, and thus $a' \ind_{M_{i_0}} b_{i_0}$. We also have $b_{i_0} \ind_M M_{i_0}$ so by \textsc{Symmetry} and \textsc{Transitivity} we get $a' \ind_M b_{i_0}$, hence $a \ind_M b$.
\end{proof}

%% file: tex/examples.tex
\section{Examples}
\label{sec:examples}
In this section we present some examples of thick NSOP$_1$ theories. First, we recall Poizat's example of a thick non-semi-Hausdorff theory (which is bounded hence NSOP$_1$). Next, we look at (the JEP refinements of) the positive theory of existentially closed exponential fields, which was shown to be NSOP$_1$ in \cite{haykazyan_existentially_2021} by constructing a suitable independence relation. We deduce from the known results that this theory is Hausdorff (hence thick), and then we show that Kim-independence coincides in it with the independence relation studied in \cite{haykazyan_existentially_2021}. Finally, we show that NSOP$_1$ is preserved under taking hyperimaginary extensions; in particular, the hyperimaginary extension of an arbitrary NSOP$_1$ theory in full first-order logic is a Hausdorff NSOP$_1$ theory.

Let us also briefly mention the class of non-simple NSOP$_1$ thick theories found recently in \cite{acf0g}.
For any integral domain $R$, the authors consider in the language of rings enriched by a predicate $P$ and constants for elements of $R$ the theory $F_{R\text{-module}}$: the theory of fields together with the quantifier-free diagram of $R$ and where $P$ defines an $R$-submodule. By \cite[Theorem 4.2, Theorem 4.8]{acf0g}, for any integral domain $R$, the theory  $F_{R\text{-module}}$ is non-simple and NSOP$_1$ in the sense of positive logic. Also, by \cite[Remark 4.9]{acf0g} it is thick and Kim-independence in the sense of our paper coincides there with weak independence, as defined in \cite[Definition 4.4]{acf0g}. In the particular case $R=\mathbb{Z}$ this shows that the theory of algebraically closed fields of characteristic zero with a generic additive subgroup, which is known to be non-companionable by \cite[Remark 1.20]{delbee0}, is non-simple and NSOP$_1$ in positive logic (cf.\ \cite[Remark 5.35]{delbee}).

\input{tex/thick-non-semi-hausdorff}
\input{tex/ecef}
\input{tex/hyperimaginaries}

%% file: tex/thick-non-semi-hausdorff.tex
\subsection{A thick, non-semi-Hausdorff theory}
\label{subsec:thick-non-semi-hausdorff}
The following is an example of a thick non-semi-Hausdorff theory from \cite[Section 4]{poizat_positive_2018}. Consider a language $L=\{P_n,R_n:n<\omega\}\cup\{r\}$ where $P_n$'s and $R_n$'s are unary relation symbols and $r$ is a binary relation symbol. Let $M = \{a_n,b_n:n<\omega \}$ be an $L$-structure with $a_0,b_0,a_1,a_2,\dots$ pairwise distinct, in which $P_n$ is interpreted as $\{a_n,b_n\}$, $R_n$ as the complement of $P_n$, and $r$ as the symmetric anti-reflexive relation $\{(a_n,b_n),(b_n,a_n):n<\omega\}$. Let $T$ be the  h-inductive theory of the structure $M$. Then the models of $T$ are bounded (in fact any e.c.\ extension of $M$ adds at most two new points), so $T$ is thick (and also NSOP$_1$). However, $T$ is not semi-Hausdorff. In fact, it was observed by Rosario Mennuni that the unique non-algebraic maximal type over $M$ does not have any global $M$-invariant extensions. This shows that, in the definition of Kim-independence in thick theories, it is necessary to work with Ls-invariant types rather than just invariant types. This is also an example where having the same type over an e.c.\ model does not guarantee having the same Lascar strong type (over that model).

%% file: tex/ecef.tex
\subsection{Existentially closed exponential fields}
\label{subsec:ecef}
In \cite{haykazyan_existentially_2021} the class of existentially closed exponential fields is studied using positive logic. They prove that this is NSOP$_1$ by providing a nice enough independence relation. We verify that this independence relation is indeed Kim-independence.
\begin{definition}
\thlabel{def:ecef}
An \emph{exponential field} or \emph{E-field} is a field of characteristic zero with a group homomorphism $E$ from the additive group to the multiplicative group. We call such a field an \emph{EA-field} if it is also an algebraically closed field. We can axiomatise EA-fields by a positive theory and call this theory $T_\EAfield$. The existentially closed exponential fields are then the e.c.\ models of $T_\EAfield$.
\end{definition}
Our definition is slightly different from \cite{haykazyan_existentially_2021} where they consider the class of e.c.\ models of just the theory of E-fields. However, these classes of e.c.\ models coincide, see \cite[Proposition 3.3]{haykazyan_existentially_2021} and the discussion after it.

There are also many different JEP-refinements, see \cite[Corollary 4.6]{haykazyan_existentially_2021}. To work in a monster model we need to fix one such JEP-refinement. This is not an issue, since everything we discuss here works in any JEP-refinement.
\begin{definition}[{\cite[Definition 5.1]{haykazyan_existentially_2021}}]
\thlabel{def:ecef-independence}
For any set $A$ write $\langle A \rangle^\EA$ for the smallest EA-subfield containing $A$. We define an independence notion $\ind$ as follows
\[
A \ind_C B \quad \Longleftrightarrow \quad \langle AC \rangle^\EA \ind_{\langle C \rangle^\EA}^\ACF \langle BC \rangle^\EA,
\]
where $\ind^\ACF$ is the usual independence relation in algebraically closed fields.
\end{definition}
Note that the independence relation $\ind$ actually makes sense over arbitrary sets. It would be interesting to compare this once Kim-independence over arbitrary sets has been developed in positive logic (see \thref{q:sets} below). For now we will restrict ourselves to working over e.c.\ models.
\begin{fact}
\thlabel{fact:ecef-facts}
We recall the following facts about $T_\EAfield$.
\begin{enumerate}[label=(\roman*)]
\item The independence relation $\ind$ satisfies \textsc{Strong Finite Character}, \textsc{Existence}, \textsc{Monotonicity}, \textsc{Symmetry}, \textsc{Independence Theorem}.
\item Any span $F_1 \leftarrow F \to F_2$ of embeddings of EA-fields can be amalgamated in such a way that, after embedding the result into the monster model, $F_1 \ind_F F_2$.
\item For EA-fields $F_1$ and $F_2$, if $\qftp(F_1) = \qftp(F_2)$ then $\tp(F_1) = \tp(F_2)$.
\end{enumerate}
\end{fact}
\begin{proof}
(i) This is \cite[Theorem 6.5]{haykazyan_existentially_2021}. They do not mention Lascar strong types in their formulation of \textsc{Independence Theorem}. However, as we will see in \thref{prop:ecef-hausdorff}, the theory is Hausdorff, so the types over e.c.\ models are Lascar strong types.

(ii) This is \cite[Theorem 4.3]{haykazyan_existentially_2021}. The fact that $F_1 \ind_F F_2$ is not mentioned in \cite[Theorem 4.3]{haykazyan_existentially_2021}, but it is direct from their proof.

(iii) This follows directly from (ii).
\end{proof}

To apply our theorem, \thref{thm:kim-pillay-style}, we need to verify a few more things.
\begin{proposition}
\thlabel{prop:ecef-hausdorff}
The theory $T_\EAfield$ is Hausdorff.
\end{proposition}
\begin{proof}
Let $T_k$ be the set of all h-inductive sentences that are true in all e.c.\ models of $T_\EAfield$. By \cite[Theorem 8]{poizat_positive_2018} being Hausdorff is equivalent to the models of $T_k$ being amalgamation bases. By \thref{fact:ecef-facts}(ii) the models of $T_\EAfield$ are already amalgamation bases, so the models of $T_k$ are in particular also amalgamation bases. So we conclude that $T_\EAfield$ is indeed Hausdorff.
\end{proof}
Note that Hausdorff is the best we can get, because \cite[Corollary 3.8]{haykazyan_existentially_2021} tells us that $T_\EAfield$ cannot be Boolean. They prove this by showing that in every e.c.\ model $F$ of $T_\EAfield$ we have for all $a \in F$ that
\[
a \in \Z \quad \Longleftrightarrow \quad F \models \forall x(E(x) = 1 \to E(ax) = 1),
\]
so if the theory were Boolean this would contradict compactness.
\begin{proposition}
\thlabel{prop:ecef-independence-extension-transitivity}
The independence relation $\ind$ in $T_\EAfield$ satisfies \textsc{Extension} and \textsc{Transitivity}.
\end{proposition}
\begin{proof}
We first prove \textsc{Transitivity}. Let $A \ind_B C$ and $A \ind_C D$ with $B \subseteq C$. So we have $\langle AB \rangle^\EA \ind_{\langle B \rangle^\EA}^\ACF \langle BC \rangle^\EA$, which is just $\langle AB \rangle^\EA \ind_{\langle B \rangle^\EA}^\ACF \langle C \rangle^\EA$. We also have $\langle AC \rangle^\EA \ind_{\langle C \rangle^\EA}^\ACF \langle CD \rangle^\EA$ and thus by monotonicity of ACF-independence $\langle AB \rangle^\EA \ind_{\langle C \rangle^\EA}^\ACF \langle CD \rangle^\EA$. Then by transitivity of ACF-independence the result follows.

Now we prove \textsc{Extension}. Let $a \ind_C b$ and let $d$ be arbitrary, so from the definition we directly get $a \ind_C Cb$. We apply \thref{fact:ecef-facts}(ii) to $\langle Cab \rangle^\EA \supseteq \langle Cb \rangle^\EA \subseteq \langle Cbd \rangle^\EA$, and we can embed the amalgamation in the monster in such a way that $\langle Cbd \rangle^\EA$ remains the same. So we get some EA-field $F$ with $\qftp(F / \langle Cb \rangle^\EA) = \qftp(\langle Cab \rangle^\EA / \langle Cb \rangle^\EA)$ and $F \ind_{\langle Cb \rangle^\EA} \langle Cbd \rangle^\EA$, which simplifies to $F \ind_{Cb} Cbd$. By \thref{fact:ecef-facts}(iii) and restricting ourselves to the copy $a' \in F$ of $a$ we thus have $\tp(a' / Cb) = \tp(a / Cb)$. So we get $a' \ind_C Cb$ and $a' \ind_{Cb} Cbd$, and $a' \ind_C bd$ follows from \textsc{Transitivity} and \textsc{Monotonicity}.
\end{proof}
\begin{corollary}
\thlabel{cor:ecef-kim-dividing}
The independence relation $\ind$ in $T_\EAfield$ is the same as Kim-independence over e.c.\ models.
\end{corollary}
\begin{proof}
This is a direct application of \thref{thm:kim-pillay-style}, using \thref{rem:strong-finite-character} to replace \textsc{Local Character} by \textsc{Strong Finite Character}.
\end{proof}

%% file: tex/hyperimaginaries.tex
\subsection{Hyperimaginaries}
\label{subsec:hyperimaginaries}
One of the main motivations for studying positive logic in \cite{ben-yaacov_positive_2003} was to be able to add hyperimaginaries in the same way we usually add imaginaries. It is well known that by doing so we leave the framework of full first-order logic, for example because we might get a bounded infinite definable set. However, we do stay within the framework of positive logic. We show that adding hyperimaginaries as real elements does not essentially change anything. So working with hyperimaginaries in positive logic requires no special treatment.

The construction in this section is based on \cite[Example 2.16]{ben-yaacov_positive_2003}, but we work things out in far greater detail. This then allows us to prove that certain properties are invariant under adding hyperimaginaries.

We fix the following things throughout the rest of this section. A positive theory $T$ in a signature $\L$ with monster model $\MM$. For simplicity we assume $\L$ is single sorted (extending this to the multi-sorted setting is straightforward). Let $\E$ be a set of partial types (over $\emptyset$) $E(x, y)$, where $x$ and $y$ are (possibly infinite but small) tuples of variables, such that each $E$ defines an equivalence relation in $\MM$.
\begin{definition}
\thlabel{def:hyperimaginary-language}
We define the \emph{hyperimaginary language} $\L_\E$ as a multi-sorted extension of $\L$. The sort of $\L$ will be called the \emph{real sort} and is denoted by $S_\real$. Then for each $E \in \E$ we add a sort $S_E$, called a \emph{hyperimaginary sort}. For a variable $y$ of sort $S_E$ we denote by $y_r$ a tuple of variables of the real sort, matching the length of the representatives of the $E$-equivalence classes.

For all $E_1, \ldots, E_n \in \E$ we add a relation symbol $R_\phi(x, y_1, \ldots, y_n)$ of sort $S_\real^{|x|} \times S_{E_1} \times \ldots \times S_{E_n}$ for each $\L$-formula $\phi(x, y_{1,r}, \ldots, y_{n,r})$.
\end{definition}
In the above definition, not all variables in $\phi(x, y_{1,r}, \ldots, y_{n,r})$ need to actually appear in the formula. In particular, it is not problem for the $y_{i,r}$ to be infinite tuples. Similarly, when we write something like $\exists y_r \phi(y_r)$, then we really only quantify over the variables that actually appear in $\phi$.
\begin{definition}
\thlabel{def:add-hyperimaginaries}
We extend $\MM$ to an $\L_\E$-structure $\MM^\E$ as follows. The real sort $S_\real$ is just $\MM$, and for each $E \in \E$ the sort $S_E$ is $\MM^\alpha / E$, where $\alpha$ is the length of the tuples of free variables in $E$. From now on we will use the shorthand notation $\MM/E$ and not mention $\alpha$. For $E_1, \ldots, E_n \in \E$ and $\phi(x, y_{1,r}, \ldots, y_{n,r})$ we interpret the relation symbol $R_\phi$ as follows. We let $\MM^\E \models R_\phi(a, c_1, \ldots, c_n)$ if and only if there are representatives $b_1, \ldots, b_n$ of $c_1, \ldots, c_n$ such that $\MM \models \phi(a, b_1, \ldots, b_n)$.
\end{definition}
For a real tuple $b$ and some $E \in \E$ we will write $[b]$ for the corresponding hyperimaginary in $\MM/E$. To prevent cluttering of notation, we will actually also use the notation $[b]$ for a tuple of hyperimaginaries. This notation leaves implicit which sort(s) $[b]$ belongs to, but that should not be a problem in what follows.
\begin{definition}
\thlabel{def:hyperimaginary-theory}
We define the $\L_\E$-theory $T^\E$ as the set of all h-inductive $\L_\E$-sentences true in $\MM^\E$.
\end{definition}
In this construction $\MM^\E$ will be a monster model of $T^\E$ (\thref{thm:hyperimaginaries-monster-preserved}). Being Hausdorff / semi-Hausdorff / thick is preserved under adding hyperimaginaries (\thref{thm:hyperimaginaries-preserving-hausdorff-thickness}). We have that $T$ is NSOP$_1$ if and only if $T^\E$ is NSOP$_1$ (\thref{thm:hyperimaginaries-nsop1}). So in particular this means that if we start with an NSOP$_1$ theory $T$ in full first-order logic, viewed as a positive theory, then $T^\E$ is a Hausdorff (and thus thick) NSOP$_1$ theory, and all our results apply. Finally, we also have that $T$ satisfies the existence axiom for forking if and only if $T^\E$ satisfies the existence axiom for forking (\thref{thm:hyperimaginaries-existence-axiom-for-forking}).

We set up our construction in such a way that we can add any set $\E$ of hyperimaginaries. If we wish to study $\MM^\text{heq}$, where we have added all hyperimaginaries, we would have to add a proper class of hyperimaginaries. We can formalise this by taking $\E$ to be the set of all equivalence relations $E(x, y)$ where $|x| \leq |T|$. Then, by \cite[Corollary 3.3]{ben-yaacov_thickness_2003}, every possible hyperimaginary is interdefinable with a set of hyperimaginaries in $\E$. So we can take $\MM^\text{heq}$ and $T^\text{heq}$ to be $\MM^\E$ and $T^\E$.
\begin{lemma}
\thlabel{lem:translation-lemma-formula}
Let $\phi(x, y)$ be an $\L_\E$-formula, where $x$ is a tuple of real variables and $y$ is a tuple of hyperimaginary variables. Then there is a set of $\L$-formulas $\Sigma_\phi(x, y_r)$ such that $\MM \models \Sigma_\phi(a, b)$ if and only if $\MM^\E \models \phi(a, [b])$.
\end{lemma}
\begin{proof}
We first assume that $\phi(x, y)$ is of the form
\[
\exists w z \left( \psi(x, w) \wedge \varepsilon(y, z) \wedge \bigwedge_{i \in I} R_{\chi_i}(x, w, y, z) \right).
\]
Here $w$ is a tuple of real variables and $z$ a tuple of hyperimaginary variables. The formula $\psi(x, w)$ is an $\L$-formula and $\varepsilon(y, z)$ is a conjunction of equalities of hyperimaginaries.

We define the partial type $\Gamma_\phi$ as follows. For each $i \in I$ we introduce tuples of real variables $y_i$ and $z_i$ matching $y_r$ and $z_r$ respectively. We let $E_\varepsilon(y_r, z_r)$ be the union of partial types in $\E$ expressing $\varepsilon([y_r], [z_r])$, and we close $E_\varepsilon$ under conjunctions. Then we set
\begin{align}
\nonumber
&\Gamma_\phi(x, y_r, w, z_r, (y_i)_{i \in I}, (z_i)_{i \in I}) &= \\
\label{eq:gamma_phi-actual-formulas}
&\left\{ \psi(x, w) \wedge \epsilon(y_r, z_r) \wedge \bigwedge_{i \in I} \chi_i(x, w, y_i, z_i) : \epsilon \in E_\varepsilon \right\} &\cup \\
\label{eq:gamma_phi-E_y}
&\bigcup \{ E_y(y_r, y_i) : i \in I \} &\cup \\
\label{eq:gamma_phi-E_z}
&\bigcup\{ E_z(z_r, z_i) : i \in I \}.
\end{align}
Here $E_y$ and $E_z$ are the equivalence relations corresponding to the hyperimaginary variables $y$ and $z$ respectively.

Let $\Sigma_\phi(x, y_r)$ express the following:
\[
\exists w z_r (y_i)_{i \in I} (z_i)_{i \in I} \Gamma_\phi(x, y_r, w, z_r, (y_i)_{i \in I}, (z_i)_{i \in I}).
\]
Now suppose that $a, b$ are such that $\MM \models \Sigma_\phi(a, b)$. Then we find realisations such that
\[
\MM \models \Gamma_\phi(a, b, c, d, (b_i)_{i \in I}, (d_i)_{i \in I}).
\]
Then (\ref{eq:gamma_phi-E_y}) and (\ref{eq:gamma_phi-E_z}) tell us that $[b] = [b_i]$ and $[d] = [d_i]$ for all $i \in I$, while (\ref{eq:gamma_phi-actual-formulas}) guarantees that $\MM^\E \models \phi(a, [b])$. This proves the forward direction and the converse is straightforward by just taking representatives of the hyperimaginaries that are involved.

We assumed $\phi$ to be of a particular form. Since every formula can be written as a disjunction of regular formulas (i.e.\ formulas built using conjunction and existential quantification), we are only left an induction step for disjunction. So let $\phi_1(x, y)$ and $\phi_2(x, y)$ with $\Sigma_{\phi_1}(x, y_r)$ and $\Sigma_{\phi_2}(x, y_r)$ be given. We define $\Sigma_{\phi_1 \vee \phi_2}(x, y_r)$ as
\[
\{ \psi_1 \vee \psi_2 : \psi_1 \in \Sigma_{\phi_1}, \psi_2 \in \Sigma_{\phi_2} \}.
\]
One easily checks that $\MM \models \Sigma_{\phi_1 \vee \phi_2}(a, b)$ precisely when $\MM \models \Sigma_{\phi_1}(a, b)$ or $\MM \models \Sigma_{\phi_2}(a, b)$ or both, and the result follows.
\end{proof}
\begin{lemma}
\thlabel{lem:translation-lemma}
Let $\Gamma(x, y)$ be a set of $\L_\E$-formulas, where $x$ is a tuple of real variables and $y$ is a tuple of hyperimaginary variables. Then there is a set of $\L$-formulas $\Sigma_\Gamma(x, y_r)$ such that $\MM \models \Sigma_\Gamma(a, b)$ if and only if $\MM^\E \models \Gamma(a, [b])$.
\end{lemma}
\begin{proof}
Define
\[
\Sigma_\Gamma(x, y_r) = \bigcup_{\phi \in \Gamma} \Sigma_\phi(x, y_r),
\]
where $\Sigma_\phi$ is as in \thref{lem:translation-lemma-formula}.
\end{proof}
\begin{lemma}
\thlabel{lem:type-hyperimaginaries-determined-by-type-reals}
If $\tp(a[b]) = \tp(a'[b'])$ then there is $b''$ such that $\tp(ab) = \tp(a'b'')$ and $[b'] = [b'']$.
\end{lemma}
\begin{proof}
Define
\[
\Sigma(x, y) = \tp_\L(ab) \cup E(b', y).
\]
It is enough to prove that $\Sigma(a', y)$ is finitely satisfiable. Let $\phi(x, y) \in \tp_\L(ab)$. Then $\MM^\E \models R_\phi(a, [b])$, so $\MM^\E \models R_\phi(a', [b'])$. So there is $b'' \in \MM$ with $\MM \models E(b', b'')$ and $\MM \models \phi(a', b'')$, as required.
\end{proof}
\begin{lemma}
\thlabel{lem:hyperimaginary-projection-type}
For every tuple of hyperimaginary variables $y$ there is a partial $\L_\E$-type $\Xi(y_r, y)$ such that $\MM^\E \models \Xi(a, [a'])$ if and only if $[a] = [a']$.
\end{lemma}
\begin{proof}
We define
\[
\Xi(y_r, y) = \{ R_\varepsilon(y_r, y) : \varepsilon \in E \},
\]
where $E$ is the equivalence relation corresponding to $y$. The right to left direction is clear. For the forward direction we suppose $\MM^\E \models \Xi(a, [a'])$. Consider the partial type
\[
\Gamma(y_r) = E(a, y_r) \cup E(y_r, a').
\]
For any $\varepsilon(a, y_r) \in E(a, y_r)$ we have $\MM^\E \models R_\varepsilon(a, [a'])$. So there must be $a^* \in \MM$ such that $[a^*] = [a']$ and $\MM \models \varepsilon(a, a^*)$. Thus $\MM \models \varepsilon(a, a^*) \wedge E(a^*, a')$. We thus see that $\Gamma$ is finitely satisfiable, so there is a realisation $a''$. We conclude that $[a] = [a''] = [a']$.
\end{proof}
\begin{lemma}
\thlabel{lem:extend-automorphism}
Any automorphism $f: \MM \to \MM$ uniquely extends to an automorphism $f^\E: \MM^\E \to \MM^\E$ by setting $f^\E([b]) = [f(b)]$.
\end{lemma}
\begin{proof}
It is straightforward to check that $f^\E$ is well-defined and bijective. We need to show that $f^\E$ preserves and reflects truth of the new relation symbols in $\L_\E$ (preservation of equality is just saying that $f^\E$ is well-defined). Suppose that $\MM^\E \models R_\phi(a, [b])$. By definition there is $b'$ such that $[b'] = [b]$ and $\MM \models \phi(a, b')$. Then $\MM \models \phi(f(a), f(b'))$ and hence $\MM^\E \models R_\phi(f(a), [f(b')])$ which is just $\MM^\E \models R_\phi(f^\E(a), f^\E([b]))$. The converse follows in a similar way.

Finally we check uniqueness of $f^\E$. Suppose that $g: \MM^\E \to \MM^\E$ also extends $f$. For $[b] \in \MM^\E$ we have $\MM^\E \models \Xi(b, [b])$ by \thref{lem:hyperimaginary-projection-type}. So if $g$ is an automorphism we must have $\MM^\E \models \Xi(g(b), g([b]))$, which means that $g([b]) = [g(b)] = [f(b)]$, as required.
\end{proof}
\begin{theorem}
\thlabel{thm:hyperimaginaries-monster-preserved}
The structure $\MM^\E$ is a monster model of $T^\E$.
\end{theorem}
\begin{proof}
We prove that $\MM^\E$ is e.c.\ and is just as saturated and homogeneous as $\MM$. So  let $\kappa$ be such that $\MM$ is $\kappa$-saturated and $\kappa$-homogeneous. Note that this means that $\kappa$ is definitely bigger than the length of any tuple representing a hyperimaginary.

\underline{Existentially closed.} We will use (iii) from \thref{def:ec-model}. Suppose that $\MM^\E \not \models \phi(a, [b])$. Then $\MM \not \models \Sigma_\phi(a, b)$, where $\Sigma_\phi$ is from \thref{lem:translation-lemma}. So there is $\psi(x, y_r) \in \Sigma_\phi(x, y_r)$ such that $\MM \not \models \psi(a, b)$. Because $\MM$ is e.c.\ we find $\chi(x, y_r)$ with $T \models \neg \exists x y_r (\psi(x, y_r) \wedge \chi(x, y_r))$ and $\MM \models \chi(a, b)$. We thus have $\MM^\E \models R_\chi(a, [b])$. We will conclude by proving that $\MM^\E \models \neg \exists xy(\phi(x, y) \wedge R_\chi(x, y))$. Suppose for a contradiction that there are $a'$ and $b'$ such that $\MM^\E \models \phi(a', [b']) \wedge R_\chi(a', [b'])$. Then there is $b''$ with $[b'] = [b'']$ and $\MM \models \chi(a', b'')$. So $\MM^\E \models \phi(a', [b''])$ and thus $\MM \models \Sigma_\phi(a', b'')$. We then get that $\MM \models \psi(a', b'') \wedge \chi(a', b'')$, which cannot happen.

\underline{Saturation.} Let $\Gamma(x, y, c, [d])$ be a finitely satisfiable partial $\L_\E$-type with $|c[d]| < \kappa$. Let $\Sigma_\Gamma(x, y, c, d)$ be the set of $\L$-formulas from \thref{lem:translation-lemma}.  By the construction there we have
\[
\Sigma_\Gamma(x, y, c, d) = \bigcup_{\phi \in \Gamma} \Sigma_\phi(x, y, c, d),
\]
where $\Sigma_\phi$ is as in \thref{lem:translation-lemma-formula}. So finite satisfiability of $\Gamma(x, y, c, [d])$ implies finite satisfiability of $\Sigma_\Gamma(x, y, c, d)$. We thus find $a, b \in \MM$ with $\MM \models \Sigma_\Gamma(a, b, c, d)$ and hence $\MM^\E \models \Gamma(a, [b], c, [d])$.

\underline{Homogeneity.} If $\tp(a[b]) = \tp(a'[b'])$, then by \thref{lem:type-hyperimaginaries-determined-by-type-reals} there is $b''$ such that $[b''] = [b']$ and $\tp(ab) = \tp(a'b'')$. Let $f: \MM \to \MM$ be an automorphism with $f(ab) = a'b''$. Then by \thref{lem:extend-automorphism} we find $f^\E: \MM^\E \to \MM^\E$ with $f^\E(a[b]) = f(a)[f(b)] = a'[b''] = a'[b']$, as required.
\end{proof}
\begin{lemma}
\thlabel{lem:indiscernible-sequences-real-representation}
A sequence $(a_i [b_i])_{i \in I}$ is indiscernible if and only if there are representatives $b_i'$ of $[b_i]$ such that $(a_i b_i')_{i \in I}$ is indiscernible.
\end{lemma}
\begin{proof}
We first prove the left to right direction. By compactness we may assume $I$ to be long enough. We can find indiscernible $(a_i^* b_i^*)_{i \in I}$ based on $(a_i b_i)_{i \in I}$. Let $p((x_i y_{i, r})_{i \in I}) = \tp((a_i^* b_i^*)_{i \in I})$ and define the following type
\[
\Gamma = p((a_i y_{i, r})_{i \in I}) \cup \{ \Xi(y_{i, r}, [b_i]) : i \in I \}.
\]
Then a realisation of $\Gamma$ is precisely what we need, so we prove that $\Gamma$ is finitely satisfiable. That is, for $i_1 < \ldots < i_n \in I$ we will produce a realisation of $\Gamma$ restricted to the variables $y_{i_1, r} \ldots y_{i_n, r}$ and parameters $a_{i_1}, \ldots, a_{i_n}, [b_{i_1}], \ldots, [b_{i_n}]$. By construction there are $j_1 < \ldots < j_n \in I$ such that $\tp(a_{i_1}^* b_{i_1}^* \ldots a_{i_n}^* b_{i_n}^*) = \tp(a_{j_1} b_{j_1} \ldots a_{j_n} b_{j_n})$. As $\tp(a_{i_1} [b_{i_1}] \ldots a_{i_n} [b_{i_n}]) = \tp(a_{j_1} [b_{j_1}] \ldots a_{j_n} [b_{j_n}])$, by \thref{lem:type-hyperimaginaries-determined-by-type-reals} we can find $b_{i_1}' \ldots b_{i_n}'$  with $\tp(a_{i_1} b_{i_1}' \ldots a_{i_n} b_{i_n}') = \tp(a_{j_1} b_{j_1} \ldots a_{j_n} b_{j_n})$ while also $[b_{i_k}'] = [b_{i_k}]$ for all $1 \leq k \leq n$. So $ b_{i_1}' \ldots  b_{i_n}'$ is the desired realisation of $\Gamma$ restricted to $ y_{i_1, r} \ldots  y_{i_n, r}$ and $a_{i_1}, \ldots, a_{i_n}, [b_{i_1}], \ldots, [b_{i_n}]$.

For the right to left direction we note that for any $i_1 < \ldots < i_n \in I$ and $j_1 < \ldots < j_n \in I$ we have
\[
\Sigma_{\tp(a_{i_1} [b_{i_1}] \ldots a_{i_n} [b_{i_n}])} \subseteq
\tp(a_{i_1} b_{i_1}' \ldots a_{i_n} b_{i_n}') =
\tp(a_{j_1} b_{j_1}' \ldots a_{j_n} b_{j_n}').
\]
So $\tp(a_{i_1} [b_{i_1}] \ldots a_{i_n} [b_{i_n}]) \subseteq \tp(a_{j_1} [b_{j_1}] \ldots a_{j_n} [b_{j_n}])$, and the claim follows by maximality of types.
\end{proof}
\begin{theorem}
\thlabel{thm:hyperimaginaries-preserving-hausdorff-thickness}
The following properties of $T$ are preserved when adding hyperimaginaries:
\begin{itemize}
\item Hausdorff,
\item semi-Hausdorff,
\item thick.
\end{itemize}
That is, if $T$ has the property then $T^\E$ has it as well.
\end{theorem}
\begin{proof}
\underline{Hausdorff.}
Let $a[b] \not \equiv a'[b']$. Then there is $\phi \in \tp(a[b])$ such that $\phi \not \in \tp(a'[b'])$. So there is a negation $\psi \in \tp(a'[b'])$ of $\phi$. By \thref{lem:translation-lemma} we have $\Sigma_\phi$ and $\Sigma_\psi$ are consistent while $\Sigma_\phi \cup \Sigma_\psi$ is inconsistent.

Fix some type $q$ of $T$ such that $\Sigma_\psi \subseteq q$. We will produce formulas $\alpha_q$ and $\beta_q$ such that $\Sigma_\phi \cup \{\alpha_q\}$ is inconsistent, $\beta_q \not \in q$ and $T \models \forall xy_r(\alpha_q(x, y_r) \vee \beta_q(x, y_r))$. Let $p \supseteq \Sigma_\phi$ be a type of $T$. Then because $T$ is Hausdorff there are formulas $\chi_p$ and $\theta_p$ such that $\chi_p \not \in p$ and $\theta_p \not \in q$, while $T \models \forall x y_r(\chi_p(x, y_r) \vee \theta_p(x, y_r))$. Then $\Sigma_\phi \cup \{\chi_p : p \supseteq \Sigma_\phi\}$ is inconsistent, so there are $p_1, \ldots, p_n$ such that $\Sigma_\phi \cup \{\chi_{p_1} \wedge \ldots \wedge \chi_{p_n}\}$ is inconsistent. We can now take $\alpha_q$ to be $\chi_{p_1} \wedge \ldots \wedge \chi_{p_n}$ and $\beta_q$ to be $\theta_{p_1} \vee \ldots \vee \theta_{p_n}$.

Now $\Sigma_\psi \cup \{ \beta_q : q \supseteq \Sigma_\psi \}$ is inconsistent. So there are $q_1, \ldots, q_k$ such that $\Sigma_\psi \cup \{ \beta_{q_1} \wedge \ldots \wedge \beta_{q_k} \}$ is inconsistent. We set $\beta = \beta_{q_1} \wedge \ldots \wedge \beta_{q_k}$ and $\alpha = \alpha_{q_1} \vee \ldots \vee \alpha_{q_n}$. We then also have that $\Sigma_\phi \cup \{\alpha\}$ is inconsistent and $T \models \forall xy_r (\alpha(x, y_r) \vee \beta(x, y_r))$.

Now consider the formulas $R_\alpha(x, y)$ and $R_\beta(x, y)$. By construction we have $T^\E \models \forall xy(R_\alpha(x, y) \vee R_\beta(x, y))$. We claim that $R_\alpha \not \in \tp(a[b])$. Suppose for a contradiction that $\MM^\E \models R_\alpha(a, [b])$. Then there is $b^*$ with $[b^*] = [b]$ such that $\MM \models \alpha(a, b^*)$. Since $\phi \in \tp(a[b]) = \tp(a[b^*])$ we also have $\MM \models \Sigma_\phi(a, b^*)$, contradicting that $\Sigma_\phi \cup \{\alpha\}$ is inconsistent. So indeed $R_\alpha \not \in \tp(a[b])$. Analogously we get that  $R_\beta \not \in \tp(a'[b'])$, which concludes the proof that $T^\E$ is Hausdorff.

\underline{Semi-Hausdorff.} Suppose that equality of $\L$-types is type-definable by a partial $\L$-type $\Omega$. Then for a tuple $x$ of real variables and a tuple $y$ of hyperimaginary variables, we consider the partial $\L_\E$-type $\Omega^\E(xy, x'y')$ that expresses the following:
\[
\exists y_r y_r' (\Xi(y_r, y) \wedge \Xi(y_r', y') \wedge \Omega(xy_r, x'y_r')).
\]
We claim that $\Omega^\E$ expresses equality of $\L_\E$-types.

If $\MM^\E \models \Omega^\E(a[b], a'[b'])$ then we find $c, c'$ such that $\MM^\E \models \Xi(c, [b]) \wedge \Xi(c', [b']) \wedge \Omega(ac, a'c')$. By \thref{lem:hyperimaginary-projection-type} we have that $[c] = [b]$ and $[c'] = [b']$. Hence $\phi \in \tp(a[b]) = \tp(a[c])$ iff $\Sigma_\phi \subseteq \tp(ac) = \tp(a'c')$ iff $\phi \in \tp(a'[c']) = \tp(a'[b'])$. So $\tp(a[b]) = \tp(a'[b'])$, as required.

Conversely, if $\tp(a[b]) = \tp(a'[b'])$ then by \thref{lem:type-hyperimaginaries-determined-by-type-reals} we find $b''$ such that $[b''] = [b']$ and $\tp(ab) = \tp(a'b'')$. Hence $\models \Xi(b, [b]) \wedge \Xi(b'', [b']) \wedge \Omega(ab, a'b'')$.

\underline{Thick.} Let $\Theta$ express indiscernibility of a sequence of real tuples, then
\[
\exists (y_{i,r})_{i < \omega} \left( \Theta((x_i y_{i,r})_{i < \omega}) \wedge \bigwedge_{i < \omega} \Xi(y_{i,r}, y_i) \right)
\]
expresses indiscernibility of $(x_i y_i)_{i < \omega}$ in $T^\E$. Here we use that a sequence in $\MM^\E$ is indiscernible if and only if there is an indiscernible sequence of real representatives, see \thref{lem:indiscernible-sequences-real-representation}.
\end{proof}
\begin{theorem}
\thlabel{thm:hyperimaginaries-nsop1}
The theory $T$ is NSOP$_1$ if and only if $T^\E$ is NSOP$_1$.
\end{theorem}
The technique in the proof of \thref{thm:hyperimaginaries-nsop1} can also be applied to other combinatorial properties, such as the order property, TP, TP$_2$, IP, etc. Of course, to do this, one first needs to write down a proper definition of these properties for positive logic, such as \thref{def:sop1} for SOP$_1$ or \cite[Definition 6.1]{haykazyan_existentially_2021} for TP$_2$.
\begin{proof}
One direction is trivial: if $T$ has a formula with SOP$_1$, then so has $T^\E$.

We prove the other direction: suppose that $T^\E$ has a formula with SOP$_1$, we will show that $T$ already has a formula with SOP$_1$. So let $\phi(x, y; w, z)$ be an $\L_\E$-formula with SOP$_1$. Here $x$ and $w$ are tuples of real variables, and $y$ and $z$ are tuples of hyperimaginary variables. Let $(a_\eta [b_\eta] : \eta \in 2^{<\omega})$ and $\psi(w_1, z_1; w_2, z_2)$ be witnesses of SOP$_1$. Let $\Sigma_\phi(x, y_r; w, z_r)$ and $\Sigma_\psi(w_1, z_{1,r}; w_2, z_{2,r})$ be as in \thref{lem:translation-lemma-formula}. Then
\[
\Sigma_\psi(w_1, z_{1,r}; w_2, z_{2,r}) \cup
\Sigma_\phi(x, y_r, w_1; z_{1,r}) \cup 
\Sigma_\phi(x, y_r, w_2; z_{2,r})
\]
is inconsistent. Hence there are finite $\phi' \in \Sigma_\phi$ and $\psi' \in \Sigma_\psi$ that are inconsistent with each other. That is
\begin{align}
\label{eq:inconsistency-witness-sop1}
T \models \neg \exists x y_r w_1 z_{1,r} w_2 z_{2,r} (\psi'(w_1, z_{1,r}, w_2, z_{2,r}) \wedge \phi'(x, y_r, w_1, z_{1,r}) \wedge \phi'(x, y_r, w_2, z_{2,r})).
\end{align}
As usual, any variables not actually appearing in the formulas should be ignored in the existential quantifier. We claim that $\phi'$ has SOP$_1$, which is witnessed by $(a_\eta b_\eta : \eta \in 2^{<\omega})$ and $\psi'$. We check the items in \thref{def:sop1}.
\begin{enumerate}[label=(\roman*)]
\item Let $\sigma \in 2^\omega$, then $\{\phi(x, y, a_{\sigma|_n}, [b_{\sigma|_n}]) : n < \omega\}$ is consistent. So there are $c$ and $[d]$ such that $\models \phi(c, [d], a_{\sigma|_n}, [b_{\sigma|_n}])$ for all $n < \omega$. That is, we have $\Sigma_\phi(c, d, a_{\sigma|_n}, b_{\sigma|_n})$ for all $n < \omega$. In particular $\{\phi'(x, y_r, a_{\sigma|_n}, b_{\sigma|_n}) : n < \omega\}$ is consistent.
\item By construction, see (\ref{eq:inconsistency-witness-sop1}).
\item Let $\eta, \nu \in \omega^{<\omega}$ such that $\eta^\frown 0 \preceq \nu$. Then $\models \psi(a_{\eta^\frown 1}, [b_{\eta^\frown 1}], a_\nu, [b_\nu])$, so $\models \Sigma_\psi(a_{\eta^\frown 1}, b_{\eta^\frown 1}, a_\nu, b_\nu)$ and in particular $\models \psi'(a_{\eta^\frown 1}, b_{\eta^\frown 1}, a_\nu, b_\nu)$.
\end{enumerate}
\end{proof}
\begin{definition}
\thlabel{def:existence-axiom}
We say that a theory satisfies the \emph{existence axiom for forking} if $\tp(a/B)$ does not fork over $B$ for any $a$ and $B$.
\end{definition}
\begin{theorem}
\thlabel{thm:hyperimaginaries-existence-axiom-for-forking}
The theory $T$ satisfies the existence axiom for forking if and only if $T^\E$ satisfies the existence axiom for forking.
\end{theorem}
\begin{proof}
One direction is immediate: anything witnessing forking in $T$ will also be in $T^\E$. We prove the other direction. So assume that there is $\tp(a[b]/C[D])$ that forks over $C[D]$. That is, it implies a (possibly infinite) disjunction $\bigvee_{i \in I} \phi_i(xy, e^i[f^i])$ with $\phi_i(xy, e^i[f^i])$ dividing over $C[D]$ for each $i \in I$. For each $i \in I$ we let $(e^i_j [f^i_j])_{j \in J}$ be a long enough $C[D]$-indiscernible sequence with $e^i_0 [f^i_0] = e^i [f^i]$ such that $\{ \phi_i(xy, e^i_j [f^i_j])  : j \in J \}$ is inconsistent. By \thref{lem:type-hyperimaginaries-determined-by-type-reals} we may assume that $e^i_jf^i_j\equiv e^if^i$ for every $j\in J$. 
We claim that $\Sigma_{\phi_i}(x, y_r, e^i, f^i)$ (see \thref{lem:translation-lemma-formula}) divides over $CD$ for all $i \in I$. Note that $\Sigma_{\phi_i}$ may contain parameters from $CD$.

To prove the claim let $k$ be such that $\{ \phi_i(xy, e^i_j [f^i_j])  : j \in J_0 \}$ is inconsistent for all $J_0 \subseteq J$ with $|J_0| = k$. So $\bigcup_{j \in J_0} \Sigma_{\phi_i}(x, y_r, e^i_j, f^i_j)$ is inconsistent for all such $J_0$. Let $(e_n f_n)_{n < \omega}$ be a $CD$-indiscernible sequence based on $(e^i_j f^i_j)_{j \in J}$ over $CD$. Then there are $j_1 < \ldots < j_k \in J$ such that $e_1 f_1 \ldots e_k f_k \equiv_{CD} e^i_{j_1} f^i_{j_1} \ldots e^i_{j_k} f^i_{j_k}$, so $\bigcup_{n < \omega} \Sigma_{\phi_i}(x, y_r, e_n, f_n)$ is inconsistent. We conclude that $\Sigma_{\phi_i}(x, y_r, e^i, f^i)$ divides over $CD$, as claimed.

By the claim there is $\psi_i(x, y_r, e^i, f^i)$ that is implied by $\Sigma_{\phi_i}(x, y_r, e^i, f^i)$ such that $\psi_i(x, y_r, e^i, f^i)$ divides over $CD$, for all $i \in I$. Let $p = \tp(a[b]C[D])$, then $\Sigma_p(x, y_r, C, D)$ implies $\bigvee_{i \in I} \Sigma_{\phi_i}(x, y_r, e^i, f^i)$. We thus have that $\Sigma_p(x, y_r, C, D)$ implies $\bigvee_{i \in I} \psi_i(x, y_r, e^i, f^i)$. So $\Sigma_p(x, y_r, C, D)$ forks over $CD$.
\end{proof}
In the discussion following Definition 4.1 in \cite{Kim20} it is stated that one may produce results for Kim-independence for the hyperimaginary extension $M^{\text{heq}}$ of a first-order structure $M$ parallel with those for first-order structures, provided that $M^{\text{heq}}$ satisfies the existence axiom for forking (which, by the above theorem, is equivalent to the assumption that $T$ satisfies this axiom). More generally, one can ask if our results on Kim-independence over models in thick NSOP$_1$ theories can be extended to arbitrary base sets assuming the existence axiom for forking:
\begin{question}\thlabel{q:sets}
Suppose $T$ is a thick positive NSOP$_1$ theory satisfying the existence axiom for forking. Can $\ind^K$ be extended to an automorphism-invariant ternary relation between arbitrary small sets which satisfies the properties listed in \thref{thm:kim-pillay-style}?
\end{question}